\documentclass[draft]{amsart}

\usepackage{cite,amsmath,amssymb,amsthm,latexsym,a4,graphicx,gastex,rotating,color,cancel,soul,tikz}

\newtheorem{theorem}{Theorem}[section]

\newtheorem{proposition}[theorem]{Proposition}

\newtheorem{lemma}[theorem]{Lemma}

\newtheorem{corollary}[theorem]{Corollary}

\newtheorem{remark}[theorem]{Remark}

\newtheorem{observation}[theorem]{Observation}

\theoremstyle{definition}

\newtheorem{problem}[theorem]{Problem}

\newtheorem{question}[theorem]{Question}

\DeclareMathOperator{\con}{con}

\DeclareMathOperator{\var}{var}

\newcommand{\contpart}[2]{\textbf{Part \ref{#1}. #2}\dotfill\pageref{#1}}

\newcommand{\contsection}[2]{\ref{#1}. #2\dotfill\pageref{#1}}

\newcommand{\contspecsection}[2]{#2\dotfill\pageref{#1}}

\begin{document}

\title{The lattice of varieties of monoids}
\thanks{The first and third authors were supported by the Ministry of Science and Higher Education of the Russian Federation (project FEUZ-2020-0016).}

\author[S.\,V.\@ Gusev]{Sergey V.\@ Gusev}
\address{Sergey V.\@ Gusev: Institute of Natural Sciences and Mathematics, Ural Federal University, Lenina str.\@ 51, 620000 Ekaterinburg, Russia}
\email{sergey.gusb@gmail.com}

\author[E.\,W.\,H.\@ Lee]{Edmond W.\@ H.\@ Lee}
\address{Edmond W.\@ H.\@ Lee: Department of Mathematics, Nova Southeastern University, Fort Lauderdale, Florida~33314, USA}
\email{edmond.lee@nova.edu}

\author[B.\,M.\@ Vernikov]{Boris M.\@ Vernikov}
\address{Boris M.\@ Vernikov: Institute of Natural Sciences and Mathematics, Ural Federal University, Lenina str.\@ 51, 620000 Ekaterinburg, Russia}
\email{bvernikov@gmail.com}

\dedicatory{In memory of Professor Lev N.\@ Shevrin (1935--2021)}

\keywords{Monoid, variety, lattice of varieties}

\subjclass{Primary 20M07, secondary 08B15}

\begin{abstract}
We survey results devoted to the lattice of varieties of monoids.
Along with known results, some unpublished results are given with proofs.
A number of open questions and problems are also formulated.
\end{abstract}

\maketitle

\section*{Contents}

\begin{itemize}
\item[]\contsection{intr}{Introduction}
\item[]\contpart{MON and its sublat}{The lattice $\mathbb{MON}$ and its sublattices}
\begin{itemize}
\item[]\contsection{init}{Initial information}
\item[]\contsection{covers}{The covering property}
\item[]\contsection{complex}{Varietal lattices with complex structures}
\item[]\contsection{struct of sublat}{Structure of certain sublattices of $\mathbb{MON}$}
\end{itemize}
\item[]\contpart{restrict}{Varieties with restrictions to subvariety lattices}
\begin{itemize}
\item[]\contsection{ident}{Identities and related conditions}
\item[]\contsection{finite}{Finiteness conditions}
\item[]\contsection{other}{Other restrictions}
\end{itemize}
\item[]\contpart{elements}{Distinctive elements in $\mathbb{MON}$}
\begin{itemize}
\item[]\contsection{spec elem}{Special elements}
\item[]\contsection{definable}{Definable varieties and sets of varieties}
\end{itemize}
\item[]\contspecsection{refer}{References}
\end{itemize}

\section{Introduction}
\label{intr}

\subsection{General remarks}
\label{gen rem}
A \emph{variety} is a class of algebras of the same type that is closed under the formation of subalgebras, homomorphic images, and arbitrary direct products.
By the celebrated Birkhoff's theorem~\cite{Birkhoff-35}, varieties are precisely classes of algebras satisfying a given set of identities and so can be investigated, in principle, by both semantic and syntactic methods.
The theory of varieties is one of the most fruitful branches of modern general algebra.
As McKenzie \textit{et~al.}\@ \cite[p.\,244]{McKenzie-McNulty-Taylor-87} capaciously said, ``in order to guide research and organize knowledge, we group algebras into varieties.''
Moreover, according to Hobby and McKenzie \cite[p.\,12]{Hobby-McKenzie-88}, grouping algebras into varieties ``has proved so fruitful that it has no serious competitor.''
Varieties of algebras have been systematically examined since the 1950s.
Significant contributions to this area were made by many authoritative figures, such as G.\@ Gr\"atzer, B.\@ J\'onsson, A.\,I.\@ Mal'cev, R.\,N.\@ McKenzie, A.\@ Tarski, and several others.

The family of all varieties of algebras of a given type forms a lattice under set-theoretical inclusion, and the examination of such lattices is one of the main research directions in the theory of varieties.
``The study of such lattices reveals an extraordinary rich structure in varieties and helps to organize our knowledge about individual algebras and important families of algebras'' \cite[p.\,vii]{McKenzie-McNulty-Taylor-87}.

The present article is devoted to a survey of work on the lattice $\mathbb{MON}$ of all varieties of monoids, where monoids are considered as semigroups equipped with an additional \mbox{0-ary} operation that fixes the identity element.
It is astute to compare the investigation of the lattice $\mathbb{MON}$ with that of the lattice $\mathbb{SEM}$ of all varieties of semigroups since the latter lattice has been the subject of intensive examination that started as early as the 1960s.
Over 200 articles on the lattice $\mathbb{SEM}$ have so far been published; several survey articles have also been written \cite{Aizenshtat-Boguta-79,Evans-71,Shevrin-Vernikov-Volkov-09,Vernikov-15}, with Shevrin \textit{et~al.}\@~\cite{Shevrin-Vernikov-Volkov-09} being the most complete and up to date, while Vernikov~\cite{Vernikov-15} focused on special elements of $\mathbb{SEM}$.

The situation with the study of the lattice $\mathbb{MON}$ differs sharply from that of $\mathbb{SEM}$.
The first article concerning the lattice $\mathbb{MON}$ was published by Head~\cite{Head-68} back in the 1960s.
For the next 50 years or so, results established in this area were scarce and fragmented: since the pioneering work of Head~\cite{Head-68}, only two more articles published before 2018---by Poll\'ak~\cite{Pollak-81} and Wismath~\cite{Wismath-86}---were devoted, in whole or to a large extent, to the study of the lattice $\mathbb{MON}$.
Significant information on this lattice can also be deduced from the work of Vachuska~\cite{Vachuska-93} on varieties of monoids with an additional unary operation.
Intermediate results related to $\mathbb{MON}$ can also be found in several articles that were mainly devoted to the study of identities of monoids; see, for example, Jackson~\cite{Jackson-05b}  and Lee \cite{Lee-08,Lee-12b,Lee-14}.
In practically every case, an explicit description of the subvariety lattice of some variety is exhibited.

Recently, interest in the lattice $\mathbb{MON}$ has grown significantly.
Since 2018, many results that are fully or partially devoted to this lattice have been established \cite{Gusev-18a,Gusev-18b,Gusev-19,Gusev-20a,Gusev-20b,Gusev-22+a,Gusev-22+b,Gusev-arx,Gusev-Lee-20,Gusev-Lee-21,Gusev-Li_Y-Zhang-arx,Gusev-Sapir_O-22,Gusev-Vernikov-18,Gusev-Vernikov-21,Jackson-Lee-18,Jackson-Zhang-21,Lee-22+,Zhang-Luo-arx}.
In view of the large amount of information on the lattice $\mathbb{MON}$ accumulated to date, the time seems ripe to survey these results and discuss directions of further research.
The present article aims to achieve these goals.
To this end, it is natural to rely on the much richer experience in studying the lattice $\mathbb{SEM}$.
In Shevrin \textit{et~al.}\@~\cite{Shevrin-Vernikov-Volkov-09}, three main directions of research of $\mathbb{SEM}$ were proposed:
\begin{itemize}
\item[(i)] examine properties of the whole lattice $\mathbb{SEM}$ and the structure of its important sublattices;
\item[(ii)] characterize varieties with given properties of their subvariety lattices;
\item[(iii)] describe varieties that occupy, in a sense, a distinctive position in $\mathbb{SEM}$.
\end{itemize}
In the study of the lattice $\mathbb{MON}$, we note that presently, there are both significant advances and numerous open problems in the ``monoid versions'' of all three aforementioned directions.

Since the amount of information on the structure of the lattice $\mathbb{MON}$ available to date is still relatively small, we are able to mention all articles known to us that are, partially or fully, within the scope of the survey.

While organizing this survey, we found a number of natural open questions that were not difficult to answer using existing results and techniques.
All such results are included in the survey with explicit proofs as a rule or, sometimes, with explanation on how they can be deduced from known results.

The structure of the article is outlined in the table of contents.
The article consists of ten sections.
Sections~\ref{init}--\ref{definable} are grouped into three parts which correspond to the three research directions in (i)--(iii) above.

\subsection{Lattices of pseudovarieties}
\label{pseudovar}
A strong motivation for investigating monoid varieties comes from computer science, more specifically from the theory of languages.
We briefly outline this important connection and refer the reader to the monographs by Almeida~\cite{Almeida-94} and Rhodes and Steinberg~\cite{Rhodes-Steinberg-09} for a comprehensive treatment.

A \emph{language} over an alphabet $X$ is an arbitrary subset of the free monoid $F^1$ over $X$.
To each language $L \subseteq F^1$ is assigned its \emph{syntactic monoid} $M_L$ defined as the quotient of $F^1$ over the largest congruence $\rho$ such that $L$ is a union of $\rho$-classes.
The syntactic monoid captures several crucial properties of $L$; in particular, a language over a finite alphabet is regular if and only if its syntactic monoid is finite.

The correspondence $L \mapsto M_L$ is neither injective nor surjective even when its input is restricted to regular languages.
It was Eilenberg~\cite{Eilenberg-76} who realized that this correspondence is bijective if the input and output are raised to the level of certain classes of regular languages and finite monoids, respectively.
The classes on the language side, commonly called \emph{varieties of regular languages}, are closed under certain natural language-theoretical operations; the classes on the monoid side are precisely \emph{pseudovarieties}---classes of finite monoids that are closed under the formation of submonoids, homomorphic images, and finitary direct products.
For any variety $\mathbf V$ of algebras of any type, the class $\mathbf V_\mathsf{fin}$ of all finite members from $\mathbf V$ is a pseudovariety, but not all pseudovarieties arise in this manner.
Pseudovarieties that are not of the form $\mathbf V_\mathsf{fin}$ include the class of all finite groups and the class of all finite monoids that are \emph{aperiodic} in the sense that all subgroups are trivial.

Both varieties of regular languages and pseudovarieties of finite monoids form complete lattices under inclusion, and Eilenberg's correspondence is an isomorphism between these two lattices.
Therefore, results concerning lattices of pseudovarieties can be reinterpreted in terms of regular languages.
Conversely, language theory distinguishes certain varieties of regular languages by their combinatorial or logical properties and thus motivates the study of the algebraic counterparts of these distinguished varieties.
An important example is the class of star-free languages, whose algebraic counterpart is the pseudovariety of aperiodic monoids; see Sch\"utzenberger~\cite{Schutzenberger-65}.
This is one of the reasons why the lattice $\mathbb{APER}$ of varieties of aperiodic monoids is a prominent sublattice of $\mathbb{MON}$.

Although the focus of the present survey is only on monoid varieties, many results overviewed will be applicable to monoid pseudovarieties, due to the following result of Aglian\'o and Nation~\cite{Agliano-Nation-89}: if $\mathbf P$ is a pseudovariety of algebras of any type and $K_{\mathbf P}$ is the class of subvariety lattices of varieties generated by monoids in $\mathbf P$, then the lattice of all subpseudovarieties of $\mathbf P$ is a homomorphic image of a sublattice of an ultraproduct of lattices from $K_{\mathbf P}$.
In particular, any positive universal sentence that holds in the subvariety lattice of a monoid variety $\mathbf V$, such as any lattice identity, also holds in the subpseudovariety lattice of $\mathbf V_\mathsf{fin}$.
Further, if $\mathbf V$ is \emph{locally finite}---that is, every finitely generated monoid in $\mathbf V$ is finite---then the mapping $\mathbf X \mapsto \mathbf X_\mathsf{fin}$ is an isomorphism between the lattice of subvarieties of $\mathbf V$ and the lattice of subpseudovarieties of $\mathbf V_\mathsf{fin}$; see Hall and Johnston \cite[Theorem~5.3]{Hall-Johnston-89}.

Throughout this survey, for many finite monoids $M$, the subvariety lattice of the monoid variety $\var M$ generated by $M$ will be explicitly described.
As a result of the aforementioned isomorphism, each of these subvariety lattices is isomorphic to the subpseudovariety lattice of the monoid pseudovariety generated by $M$.

\subsection{Terminology and notation}
\label{notat}
We assume that the reader is familiar with rudiments of semigroup theory, lattice theory, and universal algebra.
We adopt standard terminology and notation from Clifford and Preston~\cite{Clifford-Preston-61} and Howie~\cite{Howie-95} for semigroups and monoids, Gr\"atzer~\cite{Gratzer-11} for lattices, and Burris and Sankappanavar~\cite{Burris-Sankappanavar-81} and McKenzie \textit{et~al.}\@~\cite{McKenzie-McNulty-Taylor-87} for universal algebra.

The monoid obtained by adjoining a new identity element to a semigroup $S$ is denoted by $S^1$.
The free semigroup and the free monoid over a countably infinite alphabet $X$ are denoted by $F$ and $F^1$, respectively.
As usual, elements of $X$ and $F^1$ are called \emph{letters} and \emph{words}, respectively; the empty word is the identity element of $F^1$.
Words, unlike letters, are written in bold.
The two words forming an identity are connected by the symbol $\approx$, while the symbol $=$ denotes, among other things, the equality relation on the free semigroup or monoid.

Let $\mathbf T$ denote the trivial variety of algebras of any type.
We use the standard symbol $\mathbb N$ for the set of all natural numbers.
For any $n\in\mathbb N$, let $\mathbf A_n$ denote the variety of Abelian groups of exponent dividing $n$; in particular, $\mathbf A_1=\mathbf T$.
We will often put $\mathsf{sem}$ in the subscript to distinguish between a particular semigroup object and its monoid namesake.
For instance, the variety of all commutative [respectively, semilattice] monoids is denoted by $\mathbf{COM}$ [respectively, $\mathbf{SL}$], the corresponding semigroup variety is denoted by $\mathbf{COM}_\mathsf{sem}$ [respectively, $\mathbf{SL}_\mathsf{sem}$].
The variety of all monoids [respectively, semigroups] is denoted by $\mathbf{MON}$ [respectively, $\mathbf{SEM}$].
A variety of monoids [respectively, semigroups] is \emph{proper} if it is different from $\mathbf{MON}$ [respectively, $\mathbf{SEM}$].
The variety of monoids [respectively, semigroups] defined by an identity system $\Sigma$ is denoted by $\var\Sigma$ [respectively, $\var_\mathsf{sem}\Sigma$].
The variety of monoids generated by a monoid $M$ is denoted by $\var M$.

A semigroup or monoid is \emph{completely regular} if it is a union of groups.
A variety of semigroups or monoids is \emph{commutative} [respectively, \emph{periodic}, \emph{completely regular}] if it consists of commutative [respectively, periodic, completely regular] semigroups or monoids; a variety of semigroups [respectively, monoids] is \emph{overcommutative} if it contains the variety $\mathbf{COM}_\mathsf{sem}$ [respectively, $\mathbf{COM}$].
In accordance with the above convention, the lattice of all periodic [respectively, commutative, overcommutative, completely regular] varieties of monoids is denoted by $\mathbb{PER}$ [respectively, $\mathbb{COM}$, $\mathbb{OC}$, $\mathbb{CR}$], while the eponymous varietal lattice in the semigroup case is denoted by $\mathbb{PER}_\mathsf{sem}$ [respectively, $\mathbb{COM}_\mathsf{sem}$, $\mathbb{OC}_\mathsf{sem}$, $\mathbb{CR}_\mathsf{sem}$].

For a possibly empty set $W$ of words, let $I(W)$ denote the set of all words that are not subwords of any word from $W$.
It is clear that $I(W)$ is an ideal of $F^1$.
Let $S(W)$ denote the Rees quotient monoid $F^1/I(W)$.
Since every Rees quotient monoid $S(W)$ in the present article involves a singleton set $W = \{ \mathbf{w} \}$, it is more convenient to write $S(\mathbf w)$ instead of $S(\{\mathbf w\})$.
Monoids of the form $S(W)$ appeared in the literature as early as the 1940s; their construction was attributed by Morse and Hedlund~\cite{Morse-Hedlund-44} to R.\,P.\@ Dilworth.
One of the earliest discovered examples of non-finitely based finite semigroups, due to Perkins~\cite{Perkins-69} in the 1960s, is the monoid $S(\{ xyzyx, xzyxy, xyxy, x^2z \})$.
Since the beginning of the millennium, such monoids were consistently and systematically used in the articles of M.\@ Jackson, O.\,B.\@ Sapir, and other authors; see Jackson and Lee \cite[Remark~2.4]{Jackson-Lee-18} for more references.

For a variety $\mathbf X$ of algebras, we denote its subvariety lattice by $L(\mathbf X)$.
If $\mathbf X$ is a variety of monoids or semigroups, then we denote by $\overleftarrow{\mathbf X}$ the variety \emph{dual to} $\mathbf X$, that is, the variety consisting of anti-isomorphic images of algebras from $\mathbf X$.

\subsection{Why do $\mathbb{MON}$ and $\mathbb{SEM}$ satisfy different properties?}
\label{why differ}
Since semigroups and monoids are closely related types of algebras, it seems plausible that the varieties they generate should have subvariety lattices that satisfy more or less similar properties.
However, this is very far from the case.
Significant differences between the lattices $\mathbb{SEM}$ and $\mathbb{MON}$ can already be found in some early works, such as Head~\cite{Head-68} and Poll\'ak~\cite{Pollak-81}; see Subsections~\ref{COM} and~\ref{covers exists}, respectively.
Throughout the article, we will often highlight differences in properties of these two lattices.
In some cases, the lattices $\mathbb{SEM}$ and $\mathbb{MON}$ satisfy similar properties, but such instances are rare.
In this subsection, we briefly discuss what causes the mentioned differences.

When we deal with identities of monoids, we can ``eliminate'' all occurrences of a letter by substituting~1 for it.
This fundamentally affects the deducibility of identities and essentially changes the structure of varietal lattices.
We give a simple but striking example:
\[
\mathbf K=\var_\mathsf{sem}\{x^2\approx yzy\}.\label{variety K}
\]
The lattice $L(\mathbf K)$ is extremely complex because it contains an isomorphic copy of every finite lattice \cite[Lemma~3]{Volkov-89}.
But the monoid variety $\var\{x^2\approx yzy\}$ is trivial because every monoid in it satisfies the identity $1\approx z$.
A deeper reason leading to significant differences between the lattices $\mathbb{MON}$ and $\mathbb{SEM}$ will be elaborated at the end of Subsection~\ref{sublat}.

\part{The lattice $\mathbb{MON}$ and its sublattices}
\label{MON and its sublat}

\section{Initial information}
\label{init}

\subsection{Embedding of $\mathbb{MON}$ in $\mathbb{SEM}$}
\label{embed}
The following easily verifiable result plays a fundamental role in the study of the lattice $\mathbb{MON}$.
It was explicitly noted, for example, in Almeida \cite[Section~7.1]{Almeida-94} and Jackson and Lee \cite[Subsection~1.1]{Jackson-Lee-18}.

\begin{proposition}
\label{MON is sublat of SEM}
The mapping from $\mathbb{MON}$ into $\mathbb{SEM}$ that sends a monoid variety generated by a monoid $M$ to the semigroup variety generated by the semigroup reduct of $M$ is an embedding of the lattice $\mathbb{MON}$ into the lattice $\mathbb{SEM}$.
\end{proposition}

Proposition~\ref{MON is sublat of SEM} allows us to transfer various results on $\mathbb{SEM}$ to results on $\mathbb{MON}$.
Concrete examples of how Proposition~\ref{MON is sublat of SEM} enables us to obtain new information about the lattice $\mathbb{MON}$ will appear repeatedly below.

\subsection{Some basic properties of $\mathbb{MON}$}
\label{basic prop}
Before proceeding to the discussion of specific properties of the lattice $\mathbb{MON}$, we note that this lattice possesses all textbook properties of subvariety lattices of varieties of algebras: it is complete, atomic, and coalgebraic, and its cocompact elements are precisely all finitely based monoid varieties.
Several more specific properties that hold in subvariety lattices of all varieties of algebras are listed in Lampe~\cite{Lampe-91}.

A description of the atoms of the lattice $\mathbb{SEM}$ was found back in the 1950s~\cite{Kalicki-Scott-55}; the following result is thus deducible from Proposition~\ref{MON is sublat of SEM}.

\begin{observation}
\label{atoms}
The varieties $\mathbf A_p$, where $p$ ranges over the primes, and $\mathbf{SL}$ are the only atoms of the lattice $\mathbb{MON}$.
\end{observation}

The variety $\mathbf{SEM}$ is join-irreducible in the lattice $\mathbb{SEM}$ and this lattice does not have coatoms; see Evans \cite[Section~X]{Evans-71}.
These properties also hold for monoid varieties.

\begin{observation}
\label{top}
\quad
\begin{itemize}
\item[a)] The variety $\mathbf{MON}$ is join-irreducible in the lattice $\mathbb{MON}$.
\item[b)] The lattice $\mathbb{MON}$ does not have coatoms.
\end{itemize}
\end{observation}

\begin{proof}
a) This follows from Proposition~\ref{MON is sublat of SEM} and the aforementioned fact that the variety $\mathbf{SEM}$ is join-irreducible in $\mathbb{SEM}$.

\smallskip

b) Consider any proper monoid variety $\mathbf V=\var\{\mathbf u_i\approx\mathbf v_i\mid i\in I\}$.
Let $\xi$ denote the substitution that maps every letter $x$ to $x^2$.
Then $\mathbf V^\ast=\var\{ \xi(\mathbf u_i) \approx \xi (\mathbf v_i) \mid i\in I\}$ is a variety such that $\mathbf V \subset \mathbf V^\ast \subset \mathbf{MON}$, so that $\mathbf V$ is not a coatom of $\mathbb{MON}$.
\end{proof}

Recall that a lattice $\langle L;\vee,\wedge\rangle$ with a least element~0 is 0-\emph{distributive} if it satisfies the following implication:
\[
\forall x,y,z\in L:\quad x\wedge z=y\wedge z=0\longrightarrow(x\vee y)\wedge z=0.
\]
Lattices of varieties of classical types of algebras such as groups, rings, semigroups, and lattices are well known to be 0-distributive.
For semigroup varieties, this is a folklore result; see, for example, Shevrin \textit{et~al.}\@ \cite[Section~1]{Shevrin-Vernikov-Volkov-09}.
Varieties of monoids are of no exception to this trend.

\begin{observation}
\label{0-distr}
The lattice $\mathbb{MON}$ is 0-distributive.
\end{observation}

To prove this result, it suffices to check that whenever an atom of $\mathbb{MON}$ is not contained in two monoid varieties, then it is not contained in their join.
This easily follows from Proposition~\ref{MON is sublat of SEM} and the fact that the lattice $\mathbb{SEM}$ is 0-distributive.

It is evident that the map $\delta$ [respectively, $\delta_\mathsf{sem}$] that sends every variety $\mathbf V$ to its dual $\overleftarrow{\mathbf V}$ is an automorphism of the lattice $\mathbb{MON}$ [respectively, $\mathbb{SEM}$].
It is known that there exist infinitely many non-trivial injective endomorphisms of the lattice $\mathbb{SEM}$ different from $\delta_\mathsf{sem}$ (see Shevrin \textit{et~al.}\@ \cite[Section~1]{Shevrin-Vernikov-Volkov-09}), but the question of whether there exists a non-trivial automorphism of $\mathbb{SEM}$ different from $\delta_\mathsf{sem}$ remains open so far.
The following question is still open too.

\begin{question}
\label{aut exist?}
Does a non-trivial automorphism [injective endomorphism] of $\mathbb{MON}$ different from $\delta$ exist?
\end{question}

\subsection{Basic sublattices of $\mathbb{MON}$}
\label{sublat}
The lattice $\mathbb{MON}$ has several interesting and important sublattices.
First, $\mathbb{MON}$ is a disjoint union of two big sublattices: the ideal $\mathbb{PER}$ of all periodic varieties and the coideal $\mathbb{OC}$ of all overcommutative varieties.
The class of all completely regular monoid varieties forms a sublattice $\mathbb{CR}$ in $\mathbb{PER}$.
In turn, the lattice $\mathbb{CR}$ contains the sublattice $\mathbb{GR}$ of all periodic group varieties.
The ``antipode'' of $\mathbb{GR}$ and one more sublattice in $\mathbb{PER}$ is the lattice $\mathbb{APER}$ of all aperiodic varieties that was first introduced in Subsection~\ref{pseudovar}.
The intersection of the lattices $\mathbb{CR}$ and $\mathbb{APER}$ coincides with the lattice $\mathbb{BAND}$ of all varieties of band monoids, where $\mathbf{BAND}=\var\{x\approx x^2\}$ is the largest element.
To conclude the list of the main sublattices of the lattice $\mathbb{MON}$, we mention the lattice $\mathbb{COM}$ of all commutative varieties of monoids.
The sublattices of the lattice $\mathbb{MON}$ mentioned above and their relative location within $\mathbb{MON}$ are shown in Fig.~\ref{map}.

\begin{figure}[tbh]
\begin{center}
\unitlength=0.8mm
\begin{picture}(152,100)
\linethickness{1pt}
\bezier{460}(0,54)(0,11)(70,4)
\bezier{460}(0,54)(0,91)(70,95)
\bezier{460}(70,4)(140,9)(140,54)
\bezier{460}(70,95)(140,91)(140,54)
\linethickness{0.4pt}
\put(70,4){\circle*{1.9}}
\put(70,95){\circle*{1.9}}
\put(80.75,38.15){\circle*{1.9}}
\put(138.5,64){\circle*{1.9}}
\bezier{200}(2,64)(42,52)(70,4)
\bezier{230}(10,77)(78,65)(96,8)
\bezier{200}(29,88)(70,64)(138.5,64)
\bezier{230}(36,11)(135,17)(138.5,64)
\bezier{230}(70,4)(60,49)(140,56)
\put(108,46.75){\makebox(0,0)[cc]{$\mathbb{APER}$}}
\put(80,24){\makebox(0,0)[cc]{$\mathbb{BAND}$}}
\put(114,23){\makebox(0,0)[cc]{$\mathbb{COM}$}}
\put(50,59.75){\makebox(0,0)[cc]{$\mathbb{CR}$}}
\put(34,42){\makebox(0,0)[cc]{$\mathbb{GR}$}}
\put(80,72.5){\makebox(0,0)[cc]{$\mathbb{OC}$}}
\put(75,66.5){\makebox(0,0)[cc]{$\mathbb{PER}$}}
\put(78.75,38.3){\makebox(0,0)[rc]{$\mathbf{BAND}$}}
\put(140.5,64){\makebox(0,0)[lc]{$\mathbf{COM}$}}
\put(70,98.5){\makebox(0,0)[cc]{$\mathbf{MON}$}}
\put(70,0){\makebox(0,0)[cc]{$\mathbf T$}}
\end{picture}
\caption{The ``map'' of the lattice $\mathbb{MON}$}
\label{map}
\end{center}
\end{figure}

The lattices $\mathbb{COM}$, $\mathbb{CR}$, and $\mathbb{BAND}$ were examined by Head~\cite{Head-68}, Vachuska~\cite{Vachuska-93}, and Wismath~\cite{Wismath-86}, respectively.
We discuss results of these articles in greater detail in Section~\ref{struct of sublat}.

The lattice $\mathbb{OC}$ has not been systematically examined but some of its properties follow from known results.
For instance, the lattice $\mathbb{OC}_\mathsf{sem}$ is residually finite \cite[Corollary~2.3]{Volkov-94}; the following result is thus deducible from Proposition~\ref{MON is sublat of SEM}.

\begin{proposition}
\label{OC is res fin}
The lattice $\mathbb{OC}$ is residually finite.
\end{proposition}

Further properties of $\mathbb{OC}$ will be established in the next two sections; see Corollary~\ref{without covers PER,APER,OC}, Remark~\ref{Pi_infty is not in OC}, and Corollary~\ref{no qid APER,OC}.

The lattice $\mathbb{MON}$ does not contain analogues of two important sublattices of $\mathbb{SEM}$: the lattice $\mathbb{NIL}_\mathsf{sem}$ of nil-varieties of semigroups and the lattice $\mathbb{PERM}_\mathsf{sem}$ of all \emph{permutative} varieties, that is, varieties satisfying an identity of the form
\[
x_1x_2\cdots x_n\approx x_{\pi(1)}x_{\pi(2)}\cdots x_{\pi(n)},
\]
where $\pi$ is a non-trivial permutation of the set $\{1,2,\dots,n\}$.
The semigroup variety $\mathbf K$ given on page~\pageref{variety K} is an example of a nil-variety, while every variety of commutative semigroups is permutative.

The lattice $\mathbb{NIL}_\mathsf{sem}$ has a very complex structure; see Shevrin \textit{et~al.}\@ \cite[Section~7]{Shevrin-Vernikov-Volkov-09}; in particular, it does not satisfy any non-trivial identity \cite{Jezek-76,Korjakov-82}.
However, since every nil-monoid is obviously singleton, the lattice $\mathbb{MON}$ does not contain non-trivial nil-varieties.

The lattice $\mathbb{PERM}_\mathsf{sem}$ also does not satisfy any non-trivial identity~\cite{Burris-Nelson-71b}.
It is obvious, however, that every permutative monoid is commutative.
Thus, in the monoid case, the lattice of permutative varieties ``collapses'' to the lattice $\mathbb{COM}$, whose structure turns out to be very simple; see Theorem~\ref{struct of COM}.

At first glance, it seems that the absence of non-trivial nil-varieties and non-commutative permutative varieties from the lattice $\mathbb{MON}$ should greatly simplify the study of this lattice.
This is true up to a certain extent, for example, the lattice $\mathbb{COM}$ has a much simpler structure than the lattice $\mathbb{COM}_\mathsf{sem}$.
But the opposite turns out to be more prevalent.
In general, the study of the lattice $\mathbb{MON}$ tends to be more complex due to the absence of the two aforementioned types of varieties.
The solutions of many problems related to the lattice $\mathbb{SEM}$ began with the construction of counterexamples that allowed one to find strong necessary conditions, and thereby greatly narrow the field for further investigation.
These counterexamples are often constructed from permutative or nil-varieties.
With fewer opportunities to construct analogous counterexamples in the case of monoids, the investigation of many problems related to monoid varieties is severely hindered.
This is the case, for example, in the study of varieties of monoids with a modular or distributive subvariety lattice; see Subsections~\ref{mod} and~\ref{distr}.

\subsection{Minimal forbidden subvarieties for certain classes of varieties}
\label{almost smth}
All sublattices of the lattice $\mathbb{MON}$ introduced in Subsection~\ref{sublat} can be considered as (non-ordered) classes of varieties.
All these classes, with the exception of $\mathbb{OC}$, can be characterized by minimal ``forbidden subvarieties''.
Corresponding results are summarized in Table~\ref{forbid var}.
We use here and below the following notation:
\begin{align*}
\mathbf C_n&=\var\{x^n\approx x^{n+1}\},\\
\mathbf D_1&=\var\{x^2\approx x^3,\,x^2y\approx xyx\approx yx^2\},\\
\mathbf{LRB}&=\var\{xyx\approx xy\},\\
\text{and}\enskip\mathbf{RRB}&=\var\{xyx\approx yx\}.
\end{align*}
Note that $\mathbf C_1= \mathbf{SL}$ and $\mathbf C_{n+1}=\var S(x^n)$ for any $n\in\mathbb N$; this readily follows from Almeida \cite[Corollary~6.1.5]{Almeida-94}.
The variety $\mathbf D_1$ belongs to a countably infinite series of varieties $\mathbf D_k$ which will be defined in Subsection~\ref{chain}.

\begin{table}[tbh]
\begin{center}
\caption{A characterization of certain classes of monoid varieties}
\label{forbid var}
\begin{tabular}{|c|c|c|}
\hline
&A variety $\mathbf V$ of monoids lies in&if and only if $\mathbf V$ does not contain\\
\hline
1&\rule{0pt}{10pt}$\mathbb{PER}$&$\mathbf{COM}$\\
\hline
2&\rule{0pt}{10pt}$\mathbb{APER}$&$\mathbf A_p$ for all prime $p$\\
\hline
3&\rule{0pt}{10pt}$\mathbb{GR}$&$\mathbf{SL}$\\
\hline
4&\rule{0pt}{10pt}$\mathbb{CR}$&$\mathbf C_2$\\
\hline
5&\rule{0pt}{10pt}$\mathbb{BAND}$&$\mathbf A_p$ for all prime $p$ and $\mathbf C_2$\\
\hline
6&\rule{0pt}{10pt}$\mathbb{COM}$& $\mathbf D_1$, $\mathbf{LRB}$, $\mathbf{RRB}$, and all minimal\\
&&non-Abelian group varieties\\
\hline
\end{tabular}
\end{center}
\end{table}

The result in line~1 of Table~\ref{forbid var} is generally known, while the result in line~2 holds because the varieties $\mathbf A_p$ with prime $p$ are the only atoms of the lattice of group varieties.
The results in lines~3 and~4 are well known; see Gusev and Vernikov \cite[Lemma~2.1 and Corollary~2.6]{Gusev-Vernikov-18}.
The result in line~5 immediately follows from the results in lines~2 and~4.

Finally, the result in line~6 is new and due to Gusev.
One can discuss it in greater detail.
Since every commutative monoid variety is finitely based~\cite{Head-68}, it follows from Zorn's lemma that every non-commutative monoid variety contains a minimal non-commutative subvariety.
To classify all minimal non-commutative monoid varieties, we need to describe, in particular, all minimal non-Abelian varieties of periodic groups.
This problem is extremely difficult in view of the following result.
Recall that a variety of algebras is \emph{locally finite} if every finitely generated member is finite.

\begin{theorem}[Kozhevnikov {\cite[Theorem~5 and its proof]{Kozhevnikov-12}}]
\label{chain gr uncount}
For every sufficiently large prime $p$, there exist uncountably many periodic, non-locally finite (in particular, non-Abelian), non-finitely based group varieties whose proper subvarieties are all contained in $\mathbf A_p$.
\end{theorem}

Therefore, it is natural to consider only the non-group case.

\begin{theorem}
\label{just non-com}
The varieties $\mathbf D_1$, $\mathbf{LRB}$, and $\mathbf{RRB}$ are the only non-group minimal non-commutative varieties of monoids.
\end{theorem}

In comparison, there are precisely five non-group minimal non-commutative varieties of semigroups~\cite{Biryukov-76}.

To prove Theorem~\ref{just non-com}, we need the following result.

\begin{lemma}[Gusev and Vernikov {\cite[Lemma~2.14]{Gusev-Vernikov-18}}]
\label{non-cr and non-commut}
Any monoid variety that is neither completely regular nor commutative contains $\mathbf D_1$.
\end{lemma}

\begin{proof}[Proof of Theorem~\ref{just non-com}]
Let $\mathbf V$ be any non-group minimal non-commutative monoid variety.
By Lemma~\ref{non-cr and non-commut}, we may assume that $\mathbf V$ is completely regular.
Let $\mathbf V_\mathsf{sem}$ denote the semigroup variety generated by the semigroup reduct of a monoid that generates $\mathbf V$.
Then the variety $\mathbf V_\mathsf{sem}$ is non-commutative and completely regular.
Therefore, there exists a minimal non-commutative completely regular subvariety $\mathbf X$ of $\mathbf V_\mathsf{sem}$.
If $\mathbf X$ is a group variety, then $\mathbf V$ contains a non-Abelian group that generates $\mathbf X$; but this is impossible because $\mathbf V$ is a non-group minimal non-commutative monoid variety.
Thus, $\mathbf X$ is a non-group variety.
In view of Biryukov \cite[Theorem~2]{Biryukov-76}, there exist only two non-group completely regular minimal non-commutative semigroup varieties: $\mathbf{LZ}=\var_\mathsf{sem}\{xy\approx x\}$ and $\mathbf{RZ}=\var_\mathsf{sem}\{xy\approx y\}$.
Suppose that $\mathbf X=\mathbf{LZ}$.
The variety $\mathbf{LZ}$ is generated by the 2-element left zero semigroup $L_2$.
Since $\mathbf V_\mathsf{sem}$ contains the semigroup $L_2$ and is generated by a monoid, $L_2^1\in\mathbf V_\mathsf{sem}$ by Jackson \cite[Lemma~1.1]{Jackson-05c}.
It is well known that $\var L_2^1=\mathbf{LRB}$, so that $\mathbf{LRB}\subseteq\mathbf V$ by Proposition~\ref{MON is sublat of SEM}.
Since $\mathbf V$ is a minimal non-commutative monoid variety, we have $\mathbf V=\mathbf{LRB}$.
By symmetry, if $\mathbf X=\mathbf{RZ}$, then $\mathbf V=\mathbf{RRB}$.
\end{proof}

\section{The covering property}
\label{covers}

Let $S$ be a partially ordered set and $x,y\in S$.
Then $y$ is a \emph{cover} of $x$ if $x<y$ and there are no elements $z\in S$ such that $x<z<y$.
If every non-maximal element of $S$ has a cover, then $S$ has the \emph{covering property}.
As Shevrin \textit{et~al.}\@ \cite[Section~3]{Shevrin-Vernikov-Volkov-09} accurately narrated:
\begin{quote}
The study of the cover relation in varietal lattices attracted considerable attention on the early stage of development of the theory of varieties.
Evidently, there were anticipations that the structure of lattices of varieties can be revealed by moving ``upward'': from the trivial variety to its covers, that is, atoms, from the atoms to their covers, etc.
Although this hope with respect to ``big'' varietal lattices such as $\mathbb{SEM}$ has turned out to be somewhat naive, investigations of the cover relation in $\mathbb{SEM}$ and related varietal lattices have brought a number of interesting results.
\end{quote}
The above also closely describes the situation with the lattice $\mathbb{MON}$.

\subsection{The existence of covers}
\label{covers exists}
General properties of coalgebraic lattices imply that every proper variety of semigroups [respectively, monoids] defined by finitely many identities has a cover in $\mathbb{SEM}$ [respectively, $\mathbb{MON}$].
However, there exist varieties of semigroups and monoids that cannot be defined by finitely many identities.
Trakhtman \cite[Theorem~1]{Trakhtman-74} proved that the subvariety lattice of any overcommutative semigroup variety has the covering property.
It follows that the lattice $\mathbb{SEM}$ has this property because $\mathbb{SEM}$ is nothing but the subvariety lattice of the overcommutative variety $\mathbf{SEM}$.
Further details concerning the covering property in $\mathbb{SEM}$ can be found in Shevrin \textit{et~al.}\@ \cite[Section~3]{Shevrin-Vernikov-Volkov-09} or Volkov \cite[Section~3]{Volkov-02}.

The analog of the aforementioned result of Trakhtman~\cite{Trakhtman-74} for the lattice $\mathbb{MON}$ does not hold.
To give some corresponding examples, we need some notation.
Let
\[
\mathbf M_1=\var\{y^2x_1^2x_2^2\cdots x_k^2y\approx yx_1^2x_2^2\cdots x_k^2y^2\mid k\in\mathbb N\}
\]
and for each $k\in\mathbb N$, let $\mathbf N_k=\var\{\mathbf p_k\approx\mathbf q_k\}$, where
\begin{align*}
\mathbf p_k & =yxt_1t_2\cdots t_kzyt_kt_{k-1}\cdots t_1xz \\ 
\text{and}\enskip\mathbf q_k & =yxt_1t_2\cdots t_kzxyt_kt_{k-1}\cdots t_1xz.
\end{align*}
Define
\[
\mathbf N=\bigwedge_{k\in\mathbb N}\mathbf N_k=\var\{\mathbf p_k\approx\mathbf q_k\mid k\in\mathbb N\}.
\]

\begin{theorem}
\label{without covers}
The varieties $\mathbf M_1$ and $\mathbf N$ have no covers in $\mathbb{MON}$. 
Therefore, the lattice $\mathbb{MON}$ does not have the covering property.
\end{theorem}

The variety $\mathbf M_1$, due to Poll\'ak \cite[Theorem~1]{Pollak-81}, is the first published example of an overcommutative variety with no covers in $\mathbb{MON}$; other overcommutative examples can also be deduced from more recent results, such as Jackson \cite[Proposition~4.1]{Jackson-05a} and O.\,B.\@ Sapir \cite[proof of Lemma~5.1]{Sapir_O-00}.
In contrast, the variety $\mathbf N$ is aperiodic; it is a new example that is due to Gusev.
The following intermediate result is required to show that $\mathbf N$ has no covers in $\mathbb{MON}$.

\begin{lemma}
\label{ident in N_k}
Let $n,k\in\mathbb N$.
Suppose that the variety $\mathbf N_k$ satisfies a non-trivial identity $\mathbf u\approx\mathbf v$.
If $\mathbf u$ coincides with one of the words $\mathbf p_n$ or $\mathbf q_n$, then $\mathbf v$ coincides with the other word and $k\ge n$.
\end{lemma}

\begin{proof}
The \emph{content} of a word $\mathbf w$, denoted by $\con(\mathbf w)$, is the set of letters occurring in $\mathbf w$.
The \emph{head} of a word $\mathbf w$, denoted by $h(\mathbf w)$, is the first letter of $\mathbf w$.
A letter is \emph{multiple} in a word $\mathbf w$ if it occurs at least twice in $\mathbf w$.
Let $\lambda$ denote the empty word.

By assumption, there is a deduction of the identity $\mathbf u\approx\mathbf v$ from the identity $\mathbf p_k\approx\mathbf q_k$, that is, a sequence $\mathbf u = \mathbf w_0, \mathbf w_1, \ldots, \mathbf w_m = \mathbf v$ of words, where for each $i=0,1,\dots,m-1$, there exist words $\mathbf a_i,\mathbf b_i\in F^1$ and an endomorphism $\xi_i$ on $F^1$ such that $\{ \mathbf w_i, \mathbf w_{i+1} \} = \{ \mathbf a_i\xi_i(\mathbf p_k)\mathbf b_i, \mathbf a_i\xi_i(\mathbf q_k)\mathbf b_i \}$.
By induction on $m$, it suffices to consider the case when $\mathbf u=\mathbf a\xi(\mathbf s)\mathbf b$ and $\mathbf v=\mathbf a\xi(\mathbf t)\mathbf b$ for some $\mathbf a,\mathbf b\in F^1$, endomorphism $\xi$ on $F^1$, and words $\{\mathbf s,\mathbf t\}=\{\mathbf p_k,\mathbf q_k\}$.
Since any subword of $\mathbf u$ of length greater than~1 occurs only once in $\mathbf u$ and all letters occurring in $\mathbf s$ are multiple, the following result holds.

\begin{observation}
\label{xi(a)}
For any letter $a\in\con(\mathbf s)$, the word $\xi(a)$ is either empty or a letter.
\end{observation}

If $\xi(x)=\lambda$, then $\xi(\mathbf s)=\xi(\mathbf t)$, whence $\mathbf u=\mathbf v$; but this contradicts the assumption that the identity $\mathbf u\approx\mathbf v$ is non-trivial.
Thus, $\xi(x)\ne\lambda$.
Observation~\ref{xi(a)} now implies that $\xi(x)$ is a letter.
We use this fact below without further reference.

We are going to verify that $\mathbf a=\lambda$.
Arguing by contradiction, suppose that $\mathbf a\ne\lambda$.
The letter $y$ occurs in each of the words $\mathbf u$ and $\mathbf v$ exactly twice.
Since $\mathbf a\ne\lambda$ and $y=h(\mathbf u)=h(\mathbf a\xi(\mathbf s)\mathbf b)$, we have $y\in\con(\mathbf a)$, whence $y$ appears in $\xi(\mathbf s)\mathbf b$ at most once.
Then $y\notin\con(\xi(\mathbf s))$ because every letter in $\xi(\mathbf s)$ is multiple.
Hence, either $y$ is multiple in $\mathbf a$ or $y\in\con(\mathbf b)$.
Suppose that $y$ is multiple in $\mathbf a$.
Then the word $\xi(\mathbf s)\mathbf b$ is a suffix of the word $t_nt_{n-1}\cdots t_1xz$; in particular, $\xi(\mathbf s)$ is a subword of the word $t_nt_{n-1}\cdots t_1xz$.
But this is impossible because the latter word does not contain multiple letters and every letter from $\mathbf s$ is multiple.
Thus, the case when $y$ is multiple in $\mathbf a$ is impossible, whence $y\in\con(\mathbf b)$.
It follows that the word $yt_1t_2\cdots t_nxz$ is a suffix of $\mathbf b$.
Therefore, since all letters in the word $\xi(\mathbf s)$ are multiple, we have $y,z,t_1,t_2,\dots,t_n\notin\con(\xi(\mathbf s))$.
It follows that $\xi(\mathbf s)=\lambda$, a contradiction.

Hence, $\mathbf a=\lambda$.
Analogous arguments imply that $\mathbf b=\lambda$.
We see that $\mathbf u=\xi(\mathbf s)$ and $\mathbf v=\xi(\mathbf t)$.
In particular, these equalities and Observation~\ref{xi(a)} imply that $k\ge n$.

As observed above, the word $\xi(x)$ is a letter.
Suppose that $\xi(y)=\lambda$.
Then $h(\xi(\mathbf s))=\xi(x)$.
Since $\xi(\mathbf s)=\mathbf u$ and $h(\mathbf u)=y$, we obtain $\xi(x)=y$.
The equality $\xi(\mathbf s)=\mathbf u$ implies that the word $yt_nt_{n-1}\cdots t_1xz$ is a suffix of $\xi(\mathbf s)$.
On the other hand, the word $\xi(\mathbf s)$ has a suffix $\xi(x)\xi(z)=y\xi(z)$.
This implies that $\xi(z)=t_nt_{n-1}\cdots t_1xz$, which contradicts Observation~\ref{xi(a)}.
Therefore, $\xi(y)\ne\lambda$.
Now Observation~\ref{xi(a)} again applies the conclusion that $\xi(y)$ is a letter.
Since $y=h(\mathbf s)$, we have $\xi(y)=h(\xi(\mathbf s))$.
But $\xi(\mathbf s)=\mathbf u$ and $h(\mathbf u)=y$.
Thus, $\xi(y)=y$. 
Analogous arguments imply that $\xi(z)=z$.

Further, the word $\mathbf s$ has the prefix $yx$.
Therefore, the word $\xi(\mathbf s)$ has the prefix $\xi(y)\xi(x)=y\xi(x)$.
On the other hand, the word $\mathbf u=\xi(\mathbf s)$ has the prefix $yx$.
Since $\xi(x)$ is a letter, we have $\xi(x)=x$.
Then $\xi(t_1t_2\cdots t_k)=t_1t_2\cdots t_n$ and $\xi(t_kt_{k-1}\cdots t_1)=t_nt_{n-1}\cdots t_1$.

The number of occurrences of the letter $x$ in the word $\xi(\mathbf p_k)$ [respectively, $\xi(\mathbf q_k)$] is a multiple of two [respectively, three].
However, the letter $x$ occurs in the word $\mathbf p_n$ [respectively, $\mathbf q_n$] exactly two [respectively, three] times.
Therefore, if $\mathbf u=\mathbf p_n$, then $\mathbf s=\mathbf p_k$.
This fact and the arguments in the previous paragraph imply that $\mathbf v=\mathbf q_n$.
Finally, if $\mathbf u=\mathbf q_n$, then $\mathbf s=\mathbf q_k$, whence $\mathbf v=\mathbf p_n$.
\end{proof}

\begin{proof}[Proof of Theorem~\ref{without covers}]
Since $\mathbf M_1$ has no covers in $\mathbb{MON}$ \cite[Theorem~1]{Pollak-81}, it suffices to consider $\mathbf N$.
By the lemma in Poll\'ak~\cite{Pollak-81}, it suffices to verify that the following statements hold for any $n\in\mathbb N$:
\begin{itemize}
\item[(i)] the identity $\mathbf p_n\approx\mathbf q_n$ follows from the identity $\mathbf p_{n+1}\approx\mathbf q_{n+1}$;
\item[(ii)] if the variety $\mathbf N$ satisfies the identity $\mathbf p_{n+1}\approx\mathbf u$ and the identity $\mathbf u\approx\mathbf v$ follows from the identity $\mathbf p_n\approx\mathbf q_n$, then $\mathbf u=\mathbf v$.
\end{itemize}
To verify claim~(i), it suffices to note that if we substitute~1 for $t_{n+1}$ in the identity $\mathbf p_{n+1}\approx\mathbf q_{n+1}$, then we obtain the identity $\mathbf p_n\approx\mathbf q_n$.
Claim~(ii) follows from Lemma~\ref{ident in N_k}.
\end{proof}

Since the variety $\mathbf M_1$ is overcommutative and the variety $\mathbf N$ is aperiodic and so also periodic, the following result holds.

\begin{corollary}
\label{without covers PER,APER,OC}
The lattices $\mathbb{PER}$, $\mathbb{APER}$, and $\mathbb{OC}$ do not have the covering property.
\end{corollary}

On the other hand, the lattices $\mathbb{COM}$, $\mathbb{BAND}$, $\mathbb{CR}$, and $\mathbb{GR}$ have the covering property.
Specifically, the covering property for the lattices $\mathbb{COM}$ and $\mathbb{BAND}$ follows from Head~\cite{Head-68} and Wismath~\cite{Wismath-86}, respectively (see Theorems~\ref{struct of COM} and~\ref{struct of BAND}); as for the lattices $\mathbb{CR}$ and $\mathbb{GR}$, it suffices to refer to the proof of Proposition~\ref{number of covers cr} in Subsection~\ref{decompos}.

\subsection{The number of covers}
\label{number of covers}
It is fundamental to question how many covers a variety can have, if it has any at all.
There exist monoid varieties with infinitely many covers, with the trivial variety $\mathbf T$ being a mundane example; see Observation~\ref{atoms}.
The following result, the proof of which is deferred to Subsection~\ref{decompos}, provides some non-trivial examples.

\begin{proposition}
\label{number of covers cr}
Every completely regular monoid variety has infinitely many covers in the lattice $\mathbb{MON}$.
\end{proposition}

Non-completely regular varieties with infinitely many covers also exist.
For instance, for any $n \in \mathbb N$ and prime $p$, the variety $\mathbf C_n\vee\mathbf A_p$ covers $\mathbf C_n$; this follows from the description of the lattice $\mathbb{COM}$~\cite{Head-68} (see Theorem~\ref{struct of COM}).

It follows from Theorem~\ref{chain gr uncount} that for all sufficiently large prime $p$, the group variety $\mathbf A_p$ has uncountably many covers in $\mathbb{GR}$.

\begin{question}
\label{number of covers?}
Is there a non-completely regular variety of monoids with uncountably many covers in $\mathbb{MON}$?
More specifically, is there an aperiodic variety of monoids with uncountably many covers in $\mathbb{APER}$?
\end{question}

This question is presently open but has been affirmatively answered within the context of semigroup varieties.
Indeed, Trakhtman exhibited a semigroup variety with uncountably many covers in $\mathbb{SEM}$; these varieties are all aperiodic and non-completely regular \cite[Theorem~2]{Trakhtman-74}.

We now consider varieties with very few covers.
There exist monoid varieties with a finite number of covers and moreover, with a unique cover.
This follows from the following universal-algebraic result.

\begin{proposition}
\label{unique cover general}
Let $\mathbf X$ be any variety of algebras and $\mathbf V$ be any proper subvariety of $\mathbf X$ defined within $\mathbf X$ by a single identity $\mathbf u\approx\mathbf v$ with the following property: if $\mathbf V$ satisfies an identity of the form $\mathbf u\approx\mathbf w$, then $\mathbf w\in\{\mathbf u,\mathbf v\}$.
Then $\mathbf V$ has a unique cover $\mathbf U$ in the lattice $L(\mathbf X)$ and $\mathbf U$ is the meet of all subvarieties of $\mathbf X$ that properly contain $\mathbf V$.
\end{proposition}

\begin{proof}
The variety $\mathbf V$ is finitely based within $\mathbf X$.
As it is generally known, this implies that $\mathbf V$ has at least one cover in the lattice $L(\mathbf X)$.
Let $\mathbf U$ be a cover of $\mathbf V$ in $L(\mathbf X)$ and $\mathbf W$ be a variety such that $\mathbf V\subset\mathbf W\subseteq\mathbf X$.
Then it is clear that either $\mathbf U\subseteq\mathbf W$ or $\mathbf U\wedge\mathbf W=\mathbf V$.
Suppose that $\mathbf U\wedge\mathbf W=\mathbf V$.
Then there is a deduction of the identity $\mathbf u\approx\mathbf v$ from the identities of the varieties $\mathbf U$ and $\mathbf W$, that is, a sequence $\mathbf u = \mathbf w_0,\mathbf w_1,\dots,\mathbf w_n = \mathbf v$ of words such that for each $i=0,1,\dots,n-1$, the identity $\mathbf w_i\approx\mathbf w_{i+1}$ holds in either $\mathbf U$ or $\mathbf W$.
In any case, the identity $\mathbf w_i\approx\mathbf w_{i+1}$ holds in $\mathbf V$.
Then by hypothesis, $\mathbf w_i\in\{\mathbf u,\mathbf v\}$ for all $i=0,1,\dots,n$.
This contradicts the assumption that $\mathbf V$ is a proper subvariety of $\mathbf U$ and $\mathbf W$.
Therefore, $\mathbf U\subseteq\mathbf W$.
This evidently implies the required conclusion.
\end{proof}

\begin{remark}
\label{COM,M_2,N_k}
The hypothesis of Proposition~\ref{unique cover general} holds whenever $\mathbf X=\mathbf{MON}$ and $\mathbf V$ is either $\mathbf{COM}$, $\mathbf N_k$ for any $k\in\mathbb N$, or
\[
\mathbf M_2=\var\{yxyzxz\approx yxzxyxz\}.
\]
\end{remark}

\begin{proof}
It is well known and easily shown that if $\mathbf V = \mathbf{COM} =\var\{ xy \approx yx\}$ satisfies $xy \approx \mathbf v$, then $\mathbf v \in \{ xy,yx \}$.
If $\mathbf V = \mathbf N_k =\var\{\mathbf p_k\approx\mathbf q_k\}$ satisfies $\mathbf p_k\approx \mathbf v$, then $\mathbf v \in \{\mathbf p_k,\mathbf q_k\}$ by Lemma~\ref{ident in N_k}.
Finally, if $\mathbf V = \mathbf M_2=\var\{yxyzxz\approx yxzxyxz\}$ satisfies $yxyzxz \approx \mathbf v$, then it follows from Gusev \cite[proof of Lemma~3.2]{Gusev-18b} that $\mathbf v \in \{yxyzxz, yxzxyxz\}$.
\end{proof}

Proposition~\ref{unique cover general} and Remark~\ref{COM,M_2,N_k} imply the following result.

\begin{corollary}
\label{unique cover concrete}
Each of the varieties $\mathbf{COM}$, $\mathbf M_2$, and $\mathbf N_k$ for any $k\in\mathbb N$ has a unique cover in the lattice $\mathbb{MON}$.
\end{corollary}

Corollary~\ref{unique cover concrete} shows that varieties with a unique cover in $\mathbb{MON}$ include both overcommutative varieties and aperiodic ones.

\begin{remark}
\label{COM vee D_1}
The variety $\mathbf{COM}\vee\mathbf D_1$ is the unique cover of $\mathbf{COM}$.
\end{remark}

\begin{proof}
Let $\mathbf V$ be any monoid variety such that $\mathbf{COM}\subset\mathbf V$.
Then evidently, $\mathbf V$ is neither completely regular nor commutative.
Since $\mathbf D_1\subseteq\mathbf V$ by Lemma~\ref{non-cr and non-commut}, the inclusion $\mathbf{COM}\vee\mathbf D_1\subseteq\mathbf V$ follows.
\end{proof}

\section{Varietal lattices with complex structures}
\label{complex}

\subsection{Varietal lattices without non-trivial identities}
\label{without ident}
It is general knowledge that the lattice of all group varieties is not only modular but also Arguesian.
In contrast, the lattice $\mathbb{SEM}$ and even some of its sublattices do not satisfy any non-trivial lattice identity.
The proofs of these results rely heavily upon a classical result in lattice theory concerning the lattice $\Pi_n$ of all partitions of the set $\{1,2,\ldots, n\}$.

\begin{lemma}[Sachs~\cite{Sachs-61}]
\label{Sachs Pi}
The class $\{ \Pi_n \mid n\in\mathbb N \}$ does not satisfy any non-trivial lattice identity.
Consequently, the lattice $\Pi_\infty$ of all partitions of $\mathbb N$ also does not satisfy any non-trivial lattice identity.
\end{lemma}

For each $n \in \mathbb N$, the lattice dual to $\Pi_n$ is embeddable in $\mathbb{COM}_\mathsf{sem}$ and $\mathbb{OC}_\mathsf{sem}$ \cite{Burris-Nelson-71b,Volkov-94}, while the lattice dual to $\Pi_\infty$ is isomorphic to a subinterval of $L(\var_\mathsf{sem}\{x^2 \approx x^3\})$~\cite{Burris-Nelson-71a}.
Therefore, by Lemma~\ref{Sachs Pi}, the sublattices $\mathbb{COM}_\mathsf{sem}$, $\mathbb{OC}_\mathsf{sem}$, and $\mathbb{APER}_\mathsf{sem}$ of $\mathbb{SEM}$ do not satisfy any non-trivial identity.

As for the lattice $\mathbb{MON}$, the question of whether it satisfies a non-trivial identity remained open until recently.
In 2018, Gusev~\cite{Gusev-18a} proved that for any $n\in\mathbb N$, the lattice dual to $\Pi_n$ is a homomorphic image of some sublattice of $\mathbb{OC}$.
It follows from Lemma~\ref{Sachs Pi} that $\mathbb{OC}$, and so also $\mathbb{MON}$, do not satisfy any non-trivial identity.

Stronger results were more recently established.
Let $\mathbf M_3$ denote the variety of monoids defined by the identities
\begin{align*}
\sigma_1:&\enskip xyzxty\approx yxzxty,\\
\sigma_2:&\enskip xtyzxy\approx xtyzyx,\\
\sigma_3:&\enskip xzxyty\approx xzyxty
\end{align*}
and define $\mathbf M_4= \mathbf M_3 \wedge \var\{x^3\approx x^4,\,x^3y\approx yx^3\}$.

\begin{theorem}[Gusev and Lee {\cite[Theorems~3.1 and~4.1]{Gusev-Lee-20}}]
\label{fin univ M_3 and M_4}
For any $n\in\mathbb N$, the lattice $\Pi_n$ is anti-isomorphic to
\begin{itemize}
\item[a)] some subinterval of $L(\mathbf M_4)$;
\item[b)] some subinterval of $[\mathbf{COM},\mathbf M_3]$.
\end{itemize}
Consequently, the lattices $\mathbb{APER}$ and $\mathbb{OC}$ do not satisfy any non-trivial identity.
\end{theorem}

An even stronger result is given in Corollary~\ref{no qid APER,OC} below.

Now since every subvariety of $\mathbf M_4$ is finitely based~\cite{Lee-12a}, some subvariety of $\mathbf M_4$ must be minimal with respect to having a subvariety lattice that does not satisfy any non-trivial identity.
But an explicit example has not yet been found.

\begin{problem}
\label{min without ident}
\quad
\begin{itemize}
\item[a)] Find a monoid variety that is minimal with respect to having a subvariety lattice that does not satisfy any non-trivial identity.
\item[b)] Describe monoid varieties that are minimal with respect to having a subvariety lattice that does not satisfy any non-trivial identity.
\end{itemize}
\end{problem}

\subsection{Finitely universal varieties}
\label{fin univ}
A variety $\mathbf X$ of algebras is \emph{finitely universal} if every finite lattice is embeddable in $L(\mathbf X)$.
Pudl\'ak and T\.uma~\cite{Pudlak-Tuma-80} proved that every finite lattice is embeddable in $\Pi_n$ for some $n \in \mathbb N$.
It follows that a variety $\mathbf X$ is finitely universal if and only if for any $n\in\mathbb N$, the lattice $L(\mathbf X)$ contains an anti-isomorphic copy of $\Pi_n$.

Examples of finitely universal semigroup varieties have been available since the early 1970s \cite{Burris-Nelson-71a,Burris-Nelson-71b}.
Moreover, the varieties $\mathbf{COM}_\mathsf{sem}$ and $\mathbf K$ on page~\pageref{variety K} are minimal finitely universal semigroup varieties; see Shevrin \textit{et~al.}\@ \cite[Section~12]{Shevrin-Vernikov-Volkov-09}.

However, examples of finitely universal monoid varieties have been elusive for a long time, and their existence has only recently been questioned \cite[Question~6.3]{Jackson-Lee-18}.
Now Theorem~\ref{fin univ M_3 and M_4} provides an affirmative answer since it implies that the variety $\mathbf M_4$ is finitely universal.
It turns out that $\mathbf M_4$ is the smallest finitely universal monoid variety currently known, but a minimal example has not been found.

\begin{problem}
\label{min fin univ}
\quad
\begin{itemize}
\item[a)] Find an example of a minimal finitely universal monoid variety.
\item[b)] Characterize minimal finitely universal monoid varieties.
\end{itemize}
\end{problem}

The following question is also open.

\begin{question}
\label{is there min fin univ?}
Is there a finitely universal variety of monoids that does not contain any minimal finitely universal variety?
\end{question}

Recall that a variety of algebras is \emph{finitely generated} if it is generated by a finite algebra.
Although the variety $\mathbf M_4$ is not finitely generated \cite[Theorem~5.1]{Lee-14}, it is contained in some finitely generated variety \cite[Theorem~5.1]{Gusev-Lee-20}.

\begin{proposition}
\label{fin gen fin univ}
There exists a finitely generated finitely universal monoid variety.
\end{proposition}

The smallest possible order of a semigroup that generates a finitely universal variety is four, and up to isomorphism and anti-isomorphism, there are precisely four examples~\cite{Lee-07}.
But similar information for finitely universal monoid varieties is presently unknown.

\begin{problem}
\label{min order fin univ}
Find the smallest possible order of a monoid that generates a finitely universal monoid variety.
\end{problem}

Up to isomorphism and anti-isomorphism, every monoid of order five or less, with one exception, generates a monoid variety with finitely many subvarieties~\cite{Lee-Zhang-14} and so is not finitely universal; the exception is the monoid $P_2^1$, where
\[
P_2=\langle a,b\mid a^2=ab=a,\,b^2a=b^2\rangle=\{a,b,ba,b^2\}.
\]
Let $\mathbf P_2^1$ denote the variety generated by $P_2^1$.
It follows from Lee and Li~\cite{Lee-Li_J-11} that
\[
\mathbf P_2^1=\var\{xyxz\approx xyxzx,\, \sigma_2\}
\]
(the identity basis for $\mathbf P_2^1$ published earlier \cite[Lemma~9]{Tishchenko-80} turns out to be incorrect).

\begin{proposition}[Gusev \textit{et~al.}\@~\cite{Gusev-Li_Y-Zhang-arx}]
\label{P_2^1 is not fin univ}
The lattice $L(\mathbf P_2^1)$ is given in Fig.~\ref{L(P_2^1)}.
In particular, the variety $\mathbf P_2^1$ is not finitely universal because every element in $L(\mathbf P_2^1)$ has at most two covers.
\end{proposition}

\begin{figure}[p]
\unitlength=0.88mm
\linethickness{0.4pt}
\begin{center}
\begin{picture}(112,229)
\put(9,3){\circle*{1.51}}
\put(9,13){\circle*{1.51}}
\put(9,16){\circle*{1.51}}
\put(9,19){\circle*{1.51}}
\put(9,23){\circle*{1.51}}
\put(9,26){\circle*{1.51}}
\put(9,29){\circle*{1.51}}
\put(9,33){\circle*{1.51}}
\put(9,43){\circle*{1.51}}
\put(9,46){\circle*{1.51}}
\put(9,49){\circle*{1.51}}
\put(9,53){\circle*{1.51}}
\put(9,63){\circle*{1.51}}
\put(9,73){\circle*{1.51}}
\put(9,83){\circle*{1.51}}
\put(9,86){\circle*{1.51}}
\put(9,89){\circle*{1.51}}
\put(9,93){\circle*{1.51}}
\put(9,103){\circle*{1.51}}
\put(9,113){\circle*{1.51}}
\put(9,123){\circle*{1.51}}
\put(9,133){\circle*{1.51}}
\put(9,143){\circle*{1.51}}
\put(9,153){\circle*{1.51}}
\put(9,163){\circle*{1.51}}
\put(9,173){\circle*{1.51}}
\put(19,18){\circle*{1.51}}
\put(19,28){\circle*{1.51}}
\put(19,38){\circle*{1.51}}
\put(19,48){\circle*{1.51}}
\put(19,58){\circle*{1.51}}
\put(19,68){\circle*{1.51}}
\put(19,78){\circle*{1.51}}
\put(19,88){\circle*{1.51}}
\put(19,98){\circle*{1.51}}
\put(19,108){\circle*{1.51}}
\put(19,118){\circle*{1.51}}
\put(19,128){\circle*{1.51}}
\put(19,138){\circle*{1.51}}
\put(19,148){\circle*{1.51}}
\put(19,158){\circle*{1.51}}
\put(19,168){\circle*{1.51}}
\put(19,178){\circle*{1.51}}
\put(29,53){\circle*{1.51}}
\put(29,73){\circle*{1.51}}
\put(29,83){\circle*{1.51}}
\put(29,93){\circle*{1.51}}
\put(29,113){\circle*{1.51}}
\put(29,123){\circle*{1.51}}
\put(29,133){\circle*{1.51}}
\put(29,143){\circle*{1.51}}
\put(29,153){\circle*{1.51}}
\put(29,163){\circle*{1.51}}
\put(29,173){\circle*{1.51}}
\put(29,183){\circle*{1.51}}
\put(39,88){\circle*{1.51}}
\put(39,98){\circle*{1.51}}
\put(39,128){\circle*{1.51}}
\put(39,138){\circle*{1.51}}
\put(39,148){\circle*{1.51}}
\put(39,158){\circle*{1.51}}
\put(39,168){\circle*{1.51}}
\put(39,178){\circle*{1.51}}
\put(39,188){\circle*{1.51}}
\put(49,103){\circle*{1.51}}
\put(49,143){\circle*{1.51}}
\put(49,153){\circle*{1.51}}
\put(49,163){\circle*{1.51}}
\put(49,173){\circle*{1.51}}
\put(49,183){\circle*{1.51}}
\put(49,193){\circle*{1.51}}
\put(59,108){\circle*{1.51}}
\put(59,158){\circle*{1.51}}
\put(59,168){\circle*{1.51}}
\put(59,178){\circle*{1.51}}
\put(59,188){\circle*{1.51}}
\put(59,198){\circle*{1.51}}
\put(69,113){\circle*{1.51}}
\put(69,173){\circle*{1.51}}
\put(69,183){\circle*{1.51}}
\put(69,193){\circle*{1.51}}
\put(69,203){\circle*{1.51}}
\put(79,118){\circle*{1.51}}
\put(79,188){\circle*{1.51}}
\put(79,198){\circle*{1.51}}
\put(79,208){\circle*{1.51}}
\put(89,203){\circle*{1.51}}
\put(89,213){\circle*{1.51}}
\put(99,218){\circle*{1.51}}
\put(109,223){\circle*{1.51}}
\gasset{AHnb=0,linewidth=0.4}
\drawline(9,3)(9,73)(39,88)
\drawline(9,13)(19,18)(19,78)
\drawline(9,23)(19,28)
\drawline(9,33)(19,38)
\drawline(9,43)(29,53)(29,83)
\drawline(9,53)(19,58)
\drawline(9,63)(29,73)
\drawline(9,93)(19,98)
\drawline(9,103)(29,113)
\drawline(9,113)(39,128)
\drawline(9,123)(49,143)
\drawline(9,153)(59,178)
\drawline(9,173)(59,198)
\drawline(19,88)(19,138)
\drawline(19,148)(19,168)
\drawline(29,93)(29,143)
\drawline(29,153)(29,173)
\drawline(69,183)(79,188)
\drawline(69,203)(89,213)
\drawline(79,198)(89,203)
\drawline(99,218)(109,223)
\drawpolygon(9,83)(9,133)(59,158)(59,108)
\drawpolygon(9,143)(9,163)(59,188)(59,168)
\drawpolygon(39,98)(39,148)
\drawpolygon(39,158)(39,178)
\drawpolygon(49,103)(49,153)
\drawpolygon(49,163)(49,183)
\drawpolygon(69,113)(69,193)(79,198)(79,118)
\drawline[dash={0.2 0.8}{0}](9,73)(9,83)
\drawline[dash={0.2 0.8}{0}](9,133)(9,143)
\drawline[dash={0.2 0.8}{0}](9,163)(9,173)
\drawline[dash={0.2 0.8}{0}](19,78)(19,88)
\drawline[dash={0.2 0.8}{0}](19,138)(19,148)
\drawline[dash={0.2 0.8}{0}](19,168)(19,178)
\drawline[dash={0.2 0.8}{0}](29,83)(29,93)
\drawline[dash={0.2 0.8}{0}](29,143)(29,153)
\drawline[dash={0.2 0.8}{0}](29,173)(29,183)
\drawline[dash={0.2 0.8}{0}](39,88)(39,98)
\drawline[dash={0.2 0.8}{0}](39,148)(39,158)
\drawline[dash={0.2 0.8}{0}](39,178)(39,188)
\drawline[dash={0.2 0.8}{0}](49,98)(49,103)
\drawline[dash={0.2 0.8}{0}](49,153)(49,163)
\drawline[dash={0.2 0.8}{0}](49,183)(49,193)
\drawline[dash={0.2 0.8}{0}](59,103)(59,108)(69,113)(69,108)
\drawline[dash={0.2 0.8}{0}](59,158)(59,168)(69,173)
\drawline[dash={0.2 0.8}{0}](59,178)(69,183)
\drawline[dash={0.2 0.8}{0}](79,198)(79,208)
\drawline[dash={0.2 0.8}{0}](89,203)(89,213)(99,218)
\drawpolygon[dash={0.2 0.8}{0}](59,188)(59,198)(69,203)(69,193)
\put(8,15.5){\makebox(0,0)[rc]{\footnotesize$\mathbf C_2$}}
\put(8,19.5){\makebox(0,0)[rc]{\footnotesize$\mathbf D_1$}}
\put(9,176){\makebox(0,0)[cc]{$\mathbf F$}}
\put(8,23){\makebox(0,0)[rc]{\footnotesize$\mathbf F_0$}}
\put(8,26){\makebox(0,0)[rc]{\footnotesize$\mathbf F_1$}}
\put(8,46){\makebox(0,0)[rc]{\footnotesize$\mathbf F_2$}}
\put(8,86.5){\makebox(0,0)[rc]{\footnotesize$\mathbf F_k$}}
\put(8,29.5){\makebox(0,0)[rc]{\footnotesize$\mathbf H_1$}}
\put(8,49.5){\makebox(0,0)[rc]{\footnotesize$\mathbf H_2$}}
\put(8,89.5){\makebox(0,0)[rc]{\footnotesize$\mathbf H_k$}}
\put(8,33){\makebox(0,0)[rc]{\footnotesize$\mathbf I_1$}}
\put(8,53){\makebox(0,0)[rc]{\footnotesize$\mathbf I_2$}}
\put(8,93){\makebox(0,0)[rc]{\footnotesize$\mathbf I_k$}}
\put(8,42){\makebox(0,0)[rc]{\footnotesize$\mathbf J_1^1$}}
\put(8,63){\makebox(0,0)[rc]{$\mathbf J_2^1$}}
\put(8,73){\makebox(0,0)[rc]{$\mathbf J_2^2$}}
\put(8,82){\makebox(0,0)[rc]{\footnotesize$\mathbf J_{k-1}^{k-1}$}}
\put(8,103){\makebox(0,0)[rc]{$\mathbf J_k^1$}}
\put(8,113){\makebox(0,0)[rc]{$\mathbf J_k^2$}}
\put(8,123){\makebox(0,0)[rc]{$\mathbf J_k^3$}}
\put(8,133){\makebox(0,0)[rc]{$\mathbf J_k^4$}}
\put(8,143){\makebox(0,0)[rc]{$\mathbf J_k^{k-2}$}}
\put(8,153){\makebox(0,0)[rc]{$\mathbf J_k^{k-1}$}}
\put(8,163){\makebox(0,0)[rc]{$\mathbf J_k^k$}}
\put(20,38){\makebox(0,0)[lc]{$\mathbf K_1$}}
\put(20,58){\makebox(0,0)[lc]{$\mathbf K_2$}}
\put(20,98){\makebox(0,0)[lc]{$\mathbf K_k$}}
\put(99,221){\makebox(0,0)[cc]{$\mathbf L$}}
\put(19,181){\makebox(0,0)[cc]{$\mathbf L_1$}}
\put(29,186){\makebox(0,0)[cc]{$\mathbf L_2$}}
\put(39,191){\makebox(0,0)[cc]{$\mathbf L_3$}}
\put(49,196){\makebox(0,0)[cc]{$\mathbf L_4$}}
\put(59,201){\makebox(0,0)[cc]{$\mathbf L_5$}}
\put(67,206){\makebox(0,0)[cc]{$\mathbf L_{k-1}$}}
\put(77,211){\makebox(0,0)[cc]{$\mathbf L_k$}}
\put(89,216){\makebox(0,0)[cc]{$\mathbf L_{k+1}$}}
\put(20,18){\makebox(0,0)[lc]{$\mathbf{LRB}$}}
\put(30,53){\makebox(0,0)[lc]{$\mathbf O_1^1$}}
\put(30,73){\makebox(0,0)[lc]{$\mathbf O_2^1$}}
\put(40,88){\makebox(0,0)[lc]{$\mathbf O_2^2$}}
\put(80,118){\makebox(0,0)[lc]{$\mathbf O_{k-1}^{k-1}$}}
\put(30,113){\makebox(0,0)[lc]{$\mathbf O_k^1$}}
\put(40,128){\makebox(0,0)[lc]{$\mathbf O_k^2$}}
\put(50,143){\makebox(0,0)[lc]{$\mathbf O_k^3$}}
\put(60,158){\makebox(0,0)[lc]{$\mathbf O_k^4$}}
\put(70,173){\makebox(0,0)[lc]{\small$\mathbf O_k^{k-2}$}}
\put(80,188){\makebox(0,0)[lc]{$\mathbf O_k^{k-1}$}}
\put(90,203){\makebox(0,0)[lc]{$\mathbf O_k^k$}}
\put(109,226){\makebox(0,0)[cc]{$\mathbf P_2^1$}}
\put(8,12){\makebox(0,0)[rc]{\footnotesize$\mathbf{SL}$}}
\put(9,0){\makebox(0,0)[cc]{$\mathbf T$}}
\end{picture}
\end{center}
\caption{The lattice $L(\mathbf P_2^1)$}
\label{L(P_2^1)}
\end{figure}

A description of the main varieties in Fig.~\ref{L(P_2^1)} requires the following words:
\[
\mathbf b_{k,m}=x_{k-1}x_kx_{k-2}x_{k-1}\cdots x_{m-1}x_m,
\]
where $k,m\in\mathbb N$ with $m\le k$.
For brevity, write $\mathbf b_k = \mathbf b_{k,1}$ and $\mathbf b_0 = \lambda$.
Then
\begin{align*}
\mathbf F&=\var\{xyx\approx xyx^2,\,x^2y\approx x^2yx,\,x^2y^2\approx y^2x^2\},\\
\mathbf F_0&=\mathbf F\wedge\var\{x^2y\approx xyx\},\\
\mathbf F_k&=\mathbf F\wedge\var\{x_ky_kx_{k-1}x_ky_k\mathbf b_{k-1}\approx y_kx_kx_{k-1}x_ky_k\mathbf b_{k-1}\},\\
\mathbf H_k&=\mathbf F\wedge\var\{xx_kx\mathbf b_k\approx x_kx^2\mathbf b_k\},\\
\mathbf I_k&=\mathbf F\wedge\var\{y_1y_0x_ky_1\mathbf b_k\approx y_1y_0y_1x_k\mathbf b_k\},\\
\mathbf J_k^m&=\mathbf F\wedge\var\left\{\!\!
\begin{array}{r}
y_{m+1}y_mx_ky_{m+1}\mathbf b_{k,m}y_m\mathbf b_{m-1} \\
\approx y_{m+1}y_my_{m+1}x_k\mathbf b_{k,m}y_m\mathbf b_{m-1}
\end{array}
\!\!\right\},\\
\mathbf K_k&=\mathbf P_2^1\wedge\var\{y_1y_0x_ky_1\mathbf b_k\approx y_1y_0y_1x_k\mathbf b_k\},\\
\mathbf L&=\mathbf P_2^1\wedge\var\{(xy)^2\approx x^2y^2\},\\
\mathbf L_k&=\mathbf P_2^1\wedge\var\{y_kx_{k-1}xy_kx\mathbf b_{k-1} \approx y_kx_{k-1}y_kx^2\mathbf b_{k-1}\},\\
\text{and}\enskip\mathbf O_k^m&=\mathbf P_2^1\wedge\var\left\{\!\!
\begin{array}{r}
y_{m+1}y_mx_ky_{m+1}\mathbf b_{k,m}y_m\mathbf b_{m-1} \\
\approx y_{m+1}y_my_{m+1}x_k\mathbf b_{k,m}y_m\mathbf b_{m-1}
\end{array}
\!\!\right\}.
\end{align*}
Only the sublattice $L(\mathbf F)$ of $L(\mathbf P_2^1)$ requires further elaboration.
The subvarieties of $\mathbf F$ form a chain that begins with $\mathbf T \subset \mathbf{SL} \subset \mathbf C_2 \subset \mathbf D_1 \subset \mathbf F_0$, followed successively by the intervals $[\mathbf F_1,\mathbf F_2], [\mathbf F_2,\mathbf F_3], [\mathbf F_3,\mathbf F_4], \ldots,$ where each $[\mathbf F_k,\mathbf F_{k+1}]$ is the chain
\[
\mathbf F_k\subset\mathbf H_k\subset\mathbf I_k\subset\mathbf J_k^1\subset\mathbf J_k^2\subset\cdots\subset\mathbf J_k^k\subset\mathbf F_{k+1}
\]
consisting of $k+4$ varieties.

\begin{corollary}
\label{order <6}
The monoid variety generated by any monoid of order five or less is not finitely universal.
Consequently, the smallest possible order of a monoid that generates a finitely universal monoid variety is at least six.
\end{corollary}

For any $k,m\in\mathbb N$ such that $k < m$, define the \emph{Burnside monoid variety}
\[
\mathbf B_{k,m}=\var\{x^k\approx x^m\}.
\]
It is clear that the proper inclusions
\[
\mathbf B_{1,2} \subset \mathbf B_{2,3} \subset \mathbf B_{3,4} \subset \cdots \subset \mathbf B_{k,k+1} \subset \cdots
\]
hold and that every aperiodic monoid variety is contained in $\mathbf B_{k,k+1}$ for all sufficiently large $k$.
It follows from Theorem~\ref{fin univ M_3 and M_4} that the variety $\mathbf B_{3,4}$ is finitely universal.
In contrast, the variety $\mathbf B_{1,2} = \mathbf{BAND}$ is not finitely universal because the lattice $\mathbb{BAND}$ is distributive~\cite{Wismath-86}; see Theorem~\ref{struct of BAND}.
It is currently unknown if the remaining case $\mathbf B_{2,3}$ is finitely universal.

\begin{question}[Gusev and Lee {\cite[Question~6.1]{Gusev-Lee-20}}]
\label{B_23 fin univ?}
Is the Burnside monoid variety $\mathbf B_{2,3} =\var\{x^2\approx x^3\}$ finitely universal?
\end{question}

\subsection{Lattice universal varieties}
\label{lat univ}
A variety $\mathbf X$ of algebras is \emph{lattice universal} if the lattice $L(\mathbf X)$ contains an interval anti-isomorphic to $\Pi_\infty$.
Simple arguments show that if a variety $\mathbf X$ is lattice universal, then $L(\mathbf X)$ contains the lattice of all varieties of algebras of any fixed finite or countably infinite type.
Recall that the Burnside semigroup variety $\var_\mathsf{sem}\{x^2 \approx x^3\}$ and therefore, the variety $\mathbf{SEM}$, are lattice universal~\cite{Burris-Nelson-71a}.
It is therefore natural to question if the same holds true in $\mathbb{MON}$.

\begin{question}
\label{dual to Pi_infty in MON?}
\quad
\begin{itemize}
\item[a)] Does the lattice $\mathbb{MON}$ contain an anti-isomorphic copy of $\Pi_\infty$?
\item[b)] Does there exist a lattice universal variety of monoids?
\end{itemize}
\end{question}

We note that a locally finite variety of algebras of any type is not lattice universal.
This immediately follows from three folklore results: the subvariety lattice of any locally finite variety of algebras is algebraic, the lattice $\Pi_\infty$ is not coalgebraic, and any subinterval of an algebraic lattice is again algebraic.

Since the lattice dual to $\Pi_\infty$ is not embeddable in $\mathbb{OC}_\mathsf{sem}$ \cite[Corollary~2.4]{Volkov-94}, it follows from Proposition~\ref{MON is sublat of SEM} that the same holds true for $\mathbb{OC}$.

\begin{remark}
\label{Pi_infty is not in OC}
The lattice dual to $\Pi_\infty$ is not embeddable in $\mathbb{OC}$.
\end{remark}

\subsection{Varietal lattices without non-trivial quasi-identities}
\label{without qid}
Since the class of all finite lattices does not satisfy any non-trivial quasi-identity \cite[Corollary~1 after Theorem~3]{Budkin-Gorbunov-75} and every finite lattice is embeddable in $\Pi_n$ for some $n \in \mathbb N$, a stronger version of Lemma~\ref{Sachs Pi} holds.

\begin{lemma}
\label{Pi quasi}
The class $\{ \Pi_n \mid n\in\mathbb N \}$ does not satisfy any non-trivial quasi-identity.
Consequently, the lattice $\Pi_\infty$ of all partitions of $\mathbb N$ also does not satisfy any non-trivial quasi-identity.
\end{lemma}

Combining this lemma with the aforementioned results of Burris and Nelson~\cite{Burris-Nelson-71a} and Volkov~\cite{Volkov-94}, we deduce that the lattices $\mathbb{APER}_\mathsf{sem}$ and $\mathbb{OC}_\mathsf{sem}$ do not satisfy any non-trivial quasi-identity.
By Theorem~\ref{fin univ M_3 and M_4} and Lemma~\ref{Pi quasi}, similar results also hold within the lattice $\mathbb{MON}$.

\begin{corollary}[Gusev and Lee {\cite[Remark~1.3]{Gusev-Lee-20}}]
\label{no qid APER,OC}
The lattices $\mathbb{APER}$ and $\mathbb{OC}$ do not satisfy any non-trivial quasi-identity.
\end{corollary}

On the other hand, the lattices $\mathbb{COM}$ and $\mathbb{BAND}$ are distributive, while the lattice $\mathbb{CR}$ is modular; see Theorems~\ref{struct of COM} and~\ref{struct of BAND} and Proposition~\ref{CR is mod}.

\section{Structure of certain sublattices of $\mathbb{MON}$}
\label{struct of sublat}

\subsection{The lattice $\mathbb{COM}$}
\label{COM}
It is known since the 1960s that the lattice $\mathbb{COM}_\mathsf{sem}$ is countably infinite; this result holds because Perkins~\cite{Perkins-69} has shown that every variety of commutative semigroups is finitely based.
Some characterization of this lattice is provided by Kisielewicz~\cite{Kisielewicz-94}; see also Shevrin \textit{et~al.}\@ \cite[Section~8]{Shevrin-Vernikov-Volkov-09}.
Nevertheless, the lattice $\mathbb{COM}_\mathsf{sem}$ has a very complex structure; in particular, it does not satisfy any non-trivial identity~\cite{Burris-Nelson-71b}.

However, it turns out that the complexity of the lattice $\mathbb{COM}_\mathsf{sem}$ is concentrated exclusively in its nil-part, that is, the lattice of commutative nil-varieties of semigroups.
More precisely, every periodic commutative semigroup variety is the join of a nil-variety and a variety generated by a semigroup with an identity element; this folklore result, first explicitly noted without proof in Korjakov~\cite{Korjakov-82}, follows from Volkov \cite[proof of Proposition~1]{Volkov-89}.
The lattice of commutative nil-varieties has quite a complex structure.
In particular, by Korjakov~\cite{Korjakov-82}, it contains an anti-isomorphic copy of the lattice $\Pi_n$ for any $n\in\mathbb N$, whence it does not satisfy any non-trivial identity.
But nil-varieties ``disappear'' in the case of monoids; see Subsection~\ref{sublat}.
As for the lattice consisting of varieties generated by commutative semigroups with an identity element, it is isomorphic to $\mathbb{COM}$ by Proposition~\ref{MON is sublat of SEM}.
The structure of the lattice $\mathbb{COM}$ turns out to be very simple.

\begin{theorem}[Head~\cite{Head-68}]
\label{struct of COM}
The lattice $\mathbb{COM}$ is obtained by adjoining a greatest element (the variety $\mathbf{COM}$) to the direct product of the lattice of natural numbers under division (the lattice of Abelian periodic group varieties) and the chain
\[
\mathbf T\subset\mathbf{SL} = \mathbf C_ 1 \subset \mathbf C_2 \subset\mathbf C_3 \subset\cdots.
\]
In particular, if $\mathbf V$ is a commutative monoid variety, then either $\mathbf V=\mathbf{COM}$ or $\mathbf V=\mathbf A_k\vee\mathbf X$ for some $k\in\mathbb N$ and some variety $\mathbf X$ from the above chain.
\end{theorem}

\subsection{The lattice $\mathbb{BAND}$}
\label{BAND}
A complete description of the lattice $\mathbb{BAND}_\mathsf{sem}$ of all varieties of idempotent semigroups was independently found by Biryukov~\cite{Biryukov-70}, Fennemore~\cite{Fennemore-71}, and Gerhard~\cite{Gerhard-70}.
This well-known lattice is countably infinite and distributive; see, for example, Evans \cite[Fig.~4]{Evans-71} and Shevrin \textit{et~al.}\@ \cite[Fig.~2]{Shevrin-Vernikov-Volkov-09}.

A complete description of the lattice $\mathbb{BAND}$ was given by Wismath~\cite{Wismath-86}.
To describe this lattice, let $\mathbf B_2 = \mathbf{LRB}$ and for $n \ge 3$, define
\[
\mathbf B_n = \var\{x\approx x^2,\,\mathbf r_n\approx\mathbf s_n\},
\]
where
\begin{align*}
\mathbf r_n= &
\begin{cases}
x_1x_2x_3&\enskip\text{for}\enskip n=3,\\
\mathbf r_{n-1}x_n&\enskip\text{for even}\enskip n\ge 4,\\
x_n\mathbf r_{n-1}&\enskip\text{for odd}\enskip n\ge 5
\end{cases}
\\
\text{and}\enskip\mathbf s_n= &
\begin{cases}
x_1x_2x_3x_1x_3x_2x_3&\enskip\text{for}\enskip n=3,\\
\mathbf s_{n-1}x_n\mathbf r_n&\enskip\text{for even}\enskip n\ge 4,\\
\mathbf r_nx_n\mathbf s_{n-1}&\enskip\text{for odd}\enskip n\ge 5.
\end{cases}
\end{align*}

\begin{theorem}[Wismath~\cite{Wismath-86}]
\label{struct of BAND}
The lattice $\mathbb{BAND}$ is given in Fig.~\ref{L(BAND)}.
\end{theorem}

\begin{figure}[htb]
\unitlength=1mm
\linethickness{0.4pt}
\begin{center}
\begin{picture}(58,91)
\put(19,23){\circle*{1.33}}
\put(19,43){\circle*{1.33}}
\put(19,63){\circle*{1.33}}
\put(29,3){\circle*{1.33}}
\put(29,13){\circle*{1.33}}
\put(29,33){\circle*{1.33}}
\put(29,53){\circle*{1.33}}
\put(29,73){\circle*{1.33}}
\put(29,86){\circle*{1.33}}
\put(39,23){\circle*{1.33}}
\put(39,43){\circle*{1.33}}
\put(39,63){\circle*{1.33}}
\gasset{AHnb=0,linewidth=0.4}
\drawline(29,3)(29,13)(19,23)(39,43)(19,63)(31,75)
\drawline(29,13)(39,23)(19,43)(39,63)(27,75)
\put(29,81){\makebox(0,0)[cc]{$\vdots$}}
\put(18,23){\makebox(0,0)[rc]{$\mathbf B_2 = \mathbf{LRB}$}}
\put(40,23){\makebox(0,0)[lc]{$\overleftarrow{\mathbf B_2} = \mathbf{RRB}$}}
\put(40,43){\makebox(0,0)[lc]{$\mathbf B_3$}}
\put(18,43){\makebox(0,0)[rc]{$\overleftarrow{\mathbf B_3}$}}
\put(18,63){\makebox(0,0)[rc]{$\mathbf B_4$}}
\put(40,63){\makebox(0,0)[lc]{$\overleftarrow{\mathbf B_4}$}}
\put(29,89){\makebox(0,0)[cc]{$\mathbf{BAND}$}}
\put(30,12){\makebox(0,0)[lc]{$\mathbf{SL}$}}
\put(29,0){\makebox(0,0)[cc]{$\mathbf T$}}
\end{picture}
\end{center}
\caption{The lattice $\mathbb{BAND}$}
\label{L(BAND)}
\end{figure}
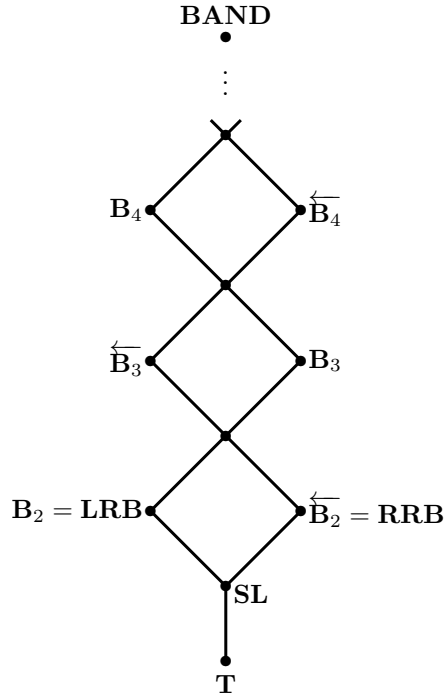

We see that the lattice $\mathbb{BAND}$ is countably infinite and distributive.
In fact, by Proposition~\ref{MON is sublat of SEM}, the distributivity and countability (but not infinitum) of $\mathbb{BAND}$ is inherited from $\mathbb{BAND}_\mathsf{sem}$.

\subsection{The lattice $\mathbb{CR}$}
\label{CR}
Completely regular semigroups can be treated as \emph{unary semigroups}, that is, semigroups with an additional unary operation.
Indeed, for any element $a$ of a completely regular semigroup $S$, denote by $a^{-1}$ the inverse element of $a$ in the maximal subgroup of $S$ that contains $a$.
The mapping $a\mapsto a^{-1}$ is a natural unary operation on $S$.
The identities
\[
(xy)z\approx x(yz),\quad xx^{-1}x\approx x,\quad xx^{-1}\approx x^{-1}x,\quad (x^{-1})^{-1}\approx x
\]
define the variety of all unary completely regular semigroups within the variety of all algebras of type $(2,1)$; denote this variety by $\mathbf{UCR}_\mathsf{sem}$ and its subvariety lattice by $\mathbb{UCR}_\mathsf{sem}$.
Important information about the structure of $\mathbb{UCR}_\mathsf{sem}$ can be found in the monograph by Petrich and Reilly~\cite{Petrich-Reilly-99}.

Every variety of completely regular semigroups can be considered as a variety of unary completely regular semigroups.
Indeed, since such a variety satisfies the identity $x\approx x^{n+1}$ for some $n\in\mathbb N$, the element $x^{2n-1}$ is inverse to $x$, whence the operation of inversion is definable in the language of multiplication.
Therefore, the lattice $\mathbb{CR}_\mathsf{sem}$ is naturally embedded in $\mathbb{UCR}_\mathsf{sem}$.
Practically, all information about the lattice $\mathbb{CR}_\mathsf{sem}$ that is known so far arises as the ``projection'' on $\mathbb{SEM}$ of results about the lattice $\mathbb{UCR}_\mathsf{sem}$; see Shevrin \textit{et~al.}\@ \cite[Section~6]{Shevrin-Vernikov-Volkov-09}.
A description of $\mathbb{UCR}_\mathsf{sem}$ and some of its important sublattices can be found in Pol\'ak \cite{Polak-85,Polak-87,Polak-88}.

Now we turn to varieties of completely regular monoids.
By analogy with the semigroup case, a monoid with an additional unary operation is called a \emph{unary monoid}.
Completely regular monoids can be treated as unary monoids with the same interpretation of the unary operation as in the semigroup case.
Let $\mathbf{UCR}$ denote the variety of all unary completely regular monoids and $\mathbb{UCR}$ denote the lattice of all unary completely regular monoid varieties.

The same arguments as in the semigroup case show that varieties of completely regular monoids can be considered as varieties of unary completely regular monoids.
Therefore, the lattice $\mathbb{CR}$ is naturally embedded in $\mathbb{UCR}$.
There has not been any articles devoted specifically to the lattice $\mathbb{CR}$, but the lattice $\mathbb{UCR}$ was studied and characterized by Vachuska~\cite{Vachuska-93}.

The following statement is a unary completely regular analog of Proposition~\ref{MON is sublat of SEM}.

\begin{proposition}[Vachuska {\cite[Lemma~2.1 and Theorem~2.2]{Vachuska-93}}]
\label{UCR is sublat of UCR_sem}
The mapping from $\mathbb{UCR}$ into $\mathbb{UCR}_\mathsf{sem}$ that sends a variety of unary completely regular monoids generated by a unary monoid $M$ to the variety of unary completely regular semigroups generated by the unary semigroup reduct of $M$ is an embedding of the lattice $\mathbb{UCR}$ into $\mathbb{UCR}_\mathsf{sem}$.
\end{proposition}

Recall that an element $x$ of a lattice $\langle L;\vee,\wedge\rangle$ is \emph{neutral} if
\[
\forall y,z\in L:\,(x\vee y)\wedge(y\vee z)\wedge(z\vee x)=(x\wedge y)\vee(y\wedge z)\vee(z\wedge x).
\]
It is well known that if $a$ is a neutral element in a lattice $L$, then $L$ is decomposable into a subdirect product of the principal ideal and the principal filter of $L$ generated by $a$; see Gr\"atzer \cite[Theorem~254]{Gratzer-11}, for instance.

Given that the variety $\mathbf{SL}_\mathsf{sem}$ of semilattices is a neutral element of the lattice $\mathbb{UCR}_\mathsf{sem}$~\cite{Hall-Jones-80}, it follows from Proposition~\ref{UCR is sublat of UCR_sem} that the variety $\mathbf{SL}$ of semilattice monoids is a neutral element of the lattice $\mathbb{UCR}$.
Therefore, $\mathbb{UCR}$ is a subdirect product of the coideal $[\mathbf{SL},\mathbf{UCR}]$ and the 2-element chain $L(\mathbf{SL})$.
To describe the lattice $\mathbb{UCR}$, it is thus sufficient to characterize the coideal $[\mathbf{SL},\mathbf{UCR}]$, which is performed in Vachuska~\cite{Vachuska-93}.
This is the monoid analog of results of Pol\'ak \cite{Polak-85,Polak-87}.

To describe the results in Vachuska~\cite{Vachuska-93}, we need some definitions and notation.
Let $U^1$ denote the free unary monoid over a countably infinite alphabet with the unary operation $^{-1}$.
Elements of $U^1$ are called \emph{unary words} or simply \emph{words}.
For any $\mathbf u\in U^1$, let $\con(\mathbf u)$ denote the set of letters occurring in $\mathbf u$.
If $|\con(\mathbf u)|>1$, then let $0(\mathbf u)$ [respectively, $1(\mathbf u)$] denote the unary word obtained from the longest initial [respectively, terminal] segment of the word $\mathbf u$ that contains $|\con(\mathbf u)|-1$ letters by omitting all opening brackets such that the segment does not contain the corresponding closing ones [respectively, all expressions in the form $)^{-1}$ such that the segment does not contain the corresponding opening brackets].
For example, if $\mathbf u=x((yx)^{-1}z)^{-1}x$, then $0(\mathbf u)=x(yx)^{-1}$ and $1(\mathbf u)=xzx$.
If $|\con(\mathbf u)|=1$, then define $0(\mathbf u) = 1(\mathbf u) = \lambda$.

For an arbitrary fully invariant congruence $\sim$ on $U^1$, define the relation $\overline\sim$ on $U^1$ recursively as follows: if $\mathbf u = \lambda$, then $\mathbf u\,\overline\sim\,\mathbf v$ if and only if $\mathbf v = \lambda$; if $|\con(\mathbf u)|=1$, then $\mathbf u\,\overline\sim\,\mathbf v$ if and only if $\con(\mathbf u)=\con(\mathbf v)$ and $\mathbf u\sim\mathbf v$; if $|\con(\mathbf u)|>1$, then $\mathbf u\,\overline\sim\,\mathbf v$ if and only if $\con(\mathbf u)=\con(\mathbf v)$, $\mathbf u\sim\mathbf v$, $0(\mathbf u)\,\overline\sim\,0(\mathbf v)$, and $1(\mathbf u)\,\overline\sim\,1(\mathbf v)$.
For any variety $\mathbf V$ of unary completely regular monoids, let $\sim_{\mathbf V}$ denote the fully invariant congruence on $U^1$ that corresponds to $\mathbf V$.
Define a relation $\rho$ on $\mathbb{UCR}$ by $\mathbf V\rho\mathbf W$ if and only if $\overline{\sim_{\mathbf V}}\,=\,\overline{\sim_{\mathbf W}}$; this relation is a complete lattice congruence on $\mathbb{UCR}$ \cite[Proposition~4.6]{Vachuska-93}.
The lattice obtained by adjoining a new least element~0 to $\mathbb{UCR}/\rho$ is denoted by $(\mathbb{UCR}/\rho)_0$.
Let $\Lambda$ be the partially ordered set in Fig.~\ref{Lambda}.

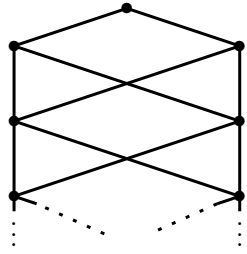
\begin{figure}[htb]
\unitlength=1mm
\linethickness{0.4pt}
\begin{center}
\begin{picture}(30,31)
\put(0,6){\circle*{1.33}}
\put(0,16){\circle*{1.33}}
\put(0,26){\circle*{1.33}}
\put(15,31){\circle*{1.33}}
\put(30,6){\circle*{1.33}}
\put(30,16){\circle*{1.33}}
\put(30,26){\circle*{1.33}}
\put(0,2){\makebox(0,0)[cc]{$\vdots$}}
\put(30,2){\makebox(0,0)[cc]{$\vdots$}}
\gasset{AHnb=0,linewidth=0.4}
\drawline(0,4)(0,26)(15,31)(30,26)(30,4)
\drawline(3,5)(0,6)(30,16)(0,26)
\drawline(27,5)(30,6)(0,16)(30,26)
\drawline[dash={0.52 1.5}{0}](3,5)(12,1)
\drawline[dash={0.52 1.5}{0}](18,1)(27,5)
\end{picture}
\end{center}
\caption{The partially ordered set $\Lambda$}
\label{Lambda}
\end{figure}

\begin{theorem}[Vachuska {\cite[Theorem~3.1]{Vachuska-93}}]
\label{struct of UCR}
The coideal $[\mathbf{SL},\mathbf{UCR}]$ of the lattice $\mathbb{UCR}$ is embeddable in the lattice of all isotone mappings from $\Lambda$ into $(\mathbb{UCR}/\rho)_0$.
\end{theorem}

In fact, the image of the coideal $[\mathbf{SL},\mathbf{UCR}]$ under the embedding mentioned in Theorem~\ref{struct of UCR} is explicitly described in Vachuska~\cite{Vachuska-93}.
Thus, the study of the lattice $\mathbb{UCR}$ is reduced to the study of the lattice $\mathbb{UCR}/\rho$.
Note also that in view of Vachuska \cite[Theorem~3.1 and Lemma~5.2]{Vachuska-93}, the coideal $[\mathbf{SL},\mathbf{UCR}]$ is a subdirect product of countably many copies of the lattice $(\mathbb{UCR}/\rho)_0$.

The construction in Vachuska~\cite{Vachuska-93}, which we reproduced above in order to state Theorem~\ref{struct of UCR}, almost literally repeats the construction introduced by Pol\'ak \cite{Polak-85,Polak-87}.
The main result of these two articles states that the coideal $[\mathbf{SL}_\mathsf{sem},\mathbf{UCR}_\mathsf{sem}]$ of the lattice $\mathbb{UCR}_\mathsf{sem}$ is embeddable in the lattice of all isotone mappings from $\Lambda$ into the ordinal sum of the 3-element non-chain meet-semilattice and $\mathbb{UCR}_\mathsf{sem}/\rho$, where the congruence $\rho$ on $\mathbb{UCR}_\mathsf{sem}$ is defined in exactly the same way as above; see Pol\'ak \cite[Theorem~3.6]{Polak-87}.

The description of the lattice $\mathbb{BAND}$ in Subsection~\ref{BAND} is in fact a very partial case of the results of Vachuska~\cite{Vachuska-93}.
The lattice $\mathbb{BAND}/\rho$ is singleton and this gives a presentation of $\mathbb{BAND}$ as a certain lattice of isotone mappings from $\Lambda$ into the 2-element lattice.

A semigroup $S$ is \emph{regular} if for any $a\in S$, there exists some $x\in S$ such that $axa=a$.
A regular semigroup is \emph{orthodox} if its idempotents form a subsemigroup.
The class of all completely regular orthodox semigroups [respectively, monoids] forms a subvariety in $\mathbf{UCR}_\mathsf{sem}$ [respectively, $\mathbf{UCR}$] defined within $\mathbf{UCR}_\mathsf{sem}$ [respectively, $\mathbf{UCR}$] by the identity $(xx^{-1}yy^{-1})^2\approx xx^{-1}yy^{-1}$.
The lattice of all completely regular orthodox semigroup [respectively, monoid] varieties is denoted by $\mathbb{UOCR}_\mathsf{sem}$ [respectively, $\mathbb{UOCR}$].
The construction in Pol\'ak \cite{Polak-85,Polak-87} turns out to be more transparent if restricted to $\mathbb{UOCR}_\mathsf{sem}$.
The lattice $\mathbb{UOCR}_\mathsf{sem}/\rho$ is isomorphic to the lattice of all varieties of groups, and Pol\'ak \cite[Theorem~4.2(2)]{Polak-88} has obtained a presentation of $\mathbb{UOCR}_\mathsf{sem}$ as a precisely described sublattice of the direct product of countably many copies of the lattice of varieties of groups.
We note that an analogous presentation for the lattice of all completely regular orthodox varieties in plain semigroup setting (without the unary operation in the language) was found earlier by Rasin~\cite{Rasin-82}.
The following proposition is a consequence of the mentioned result from Pol\'ak~\cite{Polak-88} together with Proposition~\ref{UCR is sublat of UCR_sem}.

\begin{proposition}
\label{struct of UOCR}
The lattice $\mathbb{UOCR}$ is a sublattice of the direct product of countably many copies of the lattice of varieties of groups.
\end{proposition}

Along with the results of Pol\'ak \cite{Polak-85,Polak-87,Polak-88}, another fundamental achievement in the study of the lattice $\mathbb{UCR}_\mathsf{sem}$ is the proof that this lattice is Arguesian and therefore, modular.
This was established in three different ways by Pastijn \cite{Pastijn-90,Pastin-91} and Petrich and Reilly~\cite{Petrich-Reilly-90}.
This result and Proposition~\ref{UCR is sublat of UCR_sem} imply the following assertion.

\begin{proposition}
\label{CR is mod}
The lattice $\mathbb{UCR}$ is Arguesian and therefore, modular. 
Consequently, the lattice $\mathbb{CR}$ possesses the same properties.
\end{proposition}

\part{Varieties with restrictions to subvariety lattices}
\label{restrict}

The restrictions on subvariety lattices of monoid varieties considered in this part are divided into three groups and given in Sections \ref{ident}--\ref{other}.
The first group includes lattice identities and related conditions.
The second group consists of finiteness conditions, that is, conditions satisfied by any finite lattice.
The third group contains other conditions that seem worthy of attention, such as decomposability into a direct product, properties related to the notion of lattice dualism, and the property of being a complemented lattice and related conditions.

\section{Identities and related conditions}
\label{ident}

The study of identities and related conditions traditionally attracts great attention when varietal lattices for algebras of various types are considered.
As we have already mentioned in Section~\ref{complex}, the lattice $\mathbb{MON}$ does not satisfy any non-trivial identity or even quasi-identity.
In this section, we consider a number of conditions of the aforementioned type for subvariety lattices of monoid varieties.
We consider in detail the modular and distributive laws, as well as the property of being a chain.
In addition, the Arguesian law and semimodular property are also discussed.

\subsection{Chain varieties}
\label{chain}
A variety of algebras is a \emph{chain variety} if its subvariety lattice is a chain.
Since being a chain is a much stronger property than satisfying the distributive law, the study of chain varieties can be considered a first step in the investigation of varieties with a distributive lattice of subvarieties.
For this reason, classifying chain varieties is typical for the initial stage of studying lattices of varieties of algebras of various types.

In the cases of semigroups and monoids, classifying chain varieties includes the problem of identifying chain varieties of periodic groups.
In the locally finite case, the latter problem was solved by Artamonov~\cite{Artamonov-78}.
But a complete classification of chain group varieties is extremely difficult because by Theorem~\ref{chain gr uncount}, there exist uncountably many non-locally finite group varieties whose subvariety lattice is isomorphic to the 3-element chain.

Non-group chain varieties of semigroups were completely listed by Sukhanov~\cite{Sukhanov-82}.
Gusev and Vernikov~\cite{Gusev-Vernikov-18} found a complete classification of non-group chain varieties of monoids.
To formulate this result, we need some notation.
For any $n\in\mathbb N$, let $S_n$ denote the full symmetric group on the set $\{1,2,\dots,n\}$.
For arbitrary permutations $\pi,\tau\in S_n$, define the words
\begin{align*}
\mathbf w_n(\pi,\tau)&=\biggl(\,\prod_{i=1}^nz_it_i\biggr)x\biggl(\,\prod_{i=1}^nz_{\pi(i)}z_{n+\tau(i)}\biggr)x\biggl(\,\prod_{i=n+1}^{2n} t_iz_i\biggr)\\
\text{and}\enskip\mathbf w_n'(\pi,\tau)&=\biggl(\,\prod_{i=1}^nz_it_i\biggr)x^2\biggl(\,\prod_{i=1}^nz_{\pi(i)}z_{n+\tau(i)}\biggr)\biggl(\,\prod_{i=n+1}^{2n}t_iz_i\biggr).
\end{align*}
Several varieties involving the identities $\sigma_1$, $\sigma_2$, and $\sigma_3$ from Subsection~\ref{without ident} are also required:
\begin{align*}
\mathbf D&=\var\{x^2\approx x^3,\,x^2y\approx yx^2,\,\sigma_1,\,\sigma_2,\,\sigma_3\},\\
\mathbf D_k&=\mathbf D\wedge\var\{x^2y_1y_2\cdots y_k\approx xy_1xy_2x\cdots xy_kx\},\ k\in\mathbb N, \\
\mathbf M_5&=\var\left\{\!\!
\begin{array}{l}
x^2y\approx yx^2,\,x^2yz\approx xyxzx,\\
\sigma_1,\,\sigma_2,\,\mathbf w_n(\pi,\tau)\approx\mathbf w'_n(\pi,\tau)
\end{array}
\!\middle\vert\,n\in\mathbb N,\,\pi,\tau\in S_n\right\},\\
\mathbf M_6&=\var\{x^2y\approx yx^2,\,x^2yz\approx xyxzx,\,\sigma_2,\,\sigma_3\},\\
\text{and}\enskip\mathbf M_7&=\mathbf M_6\wedge\var\{xyzxy\approx yxzxy\}.
\end{align*}

\begin{theorem}[Gusev and Vernikov {\cite[Corollary~7.1]{Gusev-Vernikov-18}}]
\label{chain descript}
The varieties $\mathbf B_2$, $\overleftarrow{\mathbf B_2}$, $\mathbf C_n$, $\mathbf D$, $\mathbf D_k$, $\mathbf F$, $\overleftarrow{\mathbf F}$, $\mathbf F_\ell$, $\overleftarrow{\mathbf F_\ell}$, $\mathbf H_k$, $\overleftarrow{\mathbf H_k}$, $\mathbf I_k$, $\overleftarrow{\mathbf I_k}$, $\mathbf J_k^m$, $\overleftarrow{\mathbf J_k^m}$, $\mathbf M_5$, $\mathbf M_6$, $\overleftarrow{\mathbf M_6}$, $\mathbf M_7$, and $\overleftarrow{\mathbf M_7}$, where $\ell\in\mathbb N\cup\{0\}$ and $k,m,n\in\mathbb N$ with $m\le k$, are the only non-group chain varieties of monoids.
\end{theorem}

The partially ordered set of all non-group chain varieties of monoids together with the variety $\mathbf T$ is shown in Fig.~\ref{all chain var}.

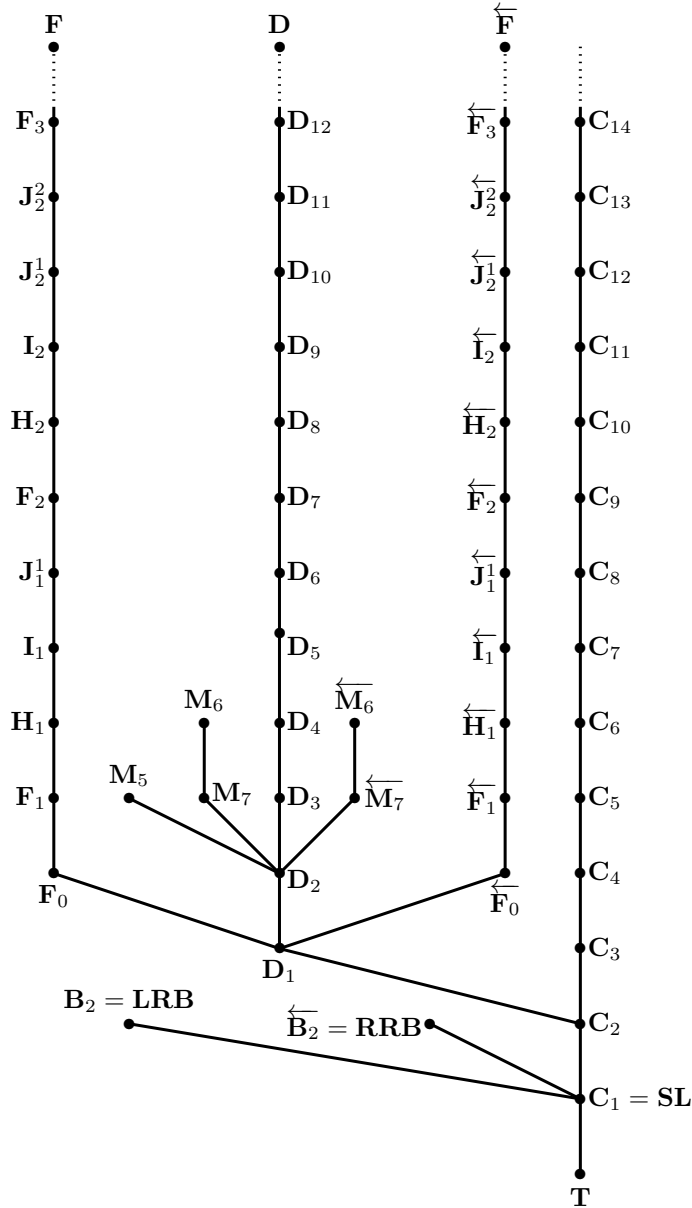
\begin{figure}[htb]
\unitlength=1mm
\linethickness{0.4pt}
\begin{center}
\begin{picture}(91,159)
\put(6,43){\circle*{1.33}}
\put(6,53){\circle*{1.33}}
\put(6,63){\circle*{1.33}}
\put(6,73){\circle*{1.33}}
\put(6,83){\circle*{1.33}}
\put(6,93){\circle*{1.33}}
\put(6,103){\circle*{1.33}}
\put(6,113){\circle*{1.33}}
\put(6,123){\circle*{1.33}}
\put(6,133){\circle*{1.33}}
\put(6,143){\circle*{1.33}}
\put(6,153){\circle*{1.33}}
\put(16,23){\circle*{1.33}}
\put(16,53){\circle*{1.33}}
\put(26,53){\circle*{1.33}}
\put(26,63){\circle*{1.33}}
\put(36,33){\circle*{1.33}}
\put(36,43){\circle*{1.33}}
\put(36,53){\circle*{1.33}}
\put(36,63){\circle*{1.33}}
\put(36,75){\circle*{1.33}}
\put(36,83){\circle*{1.33}}
\put(36,93){\circle*{1.33}}
\put(36,103){\circle*{1.33}}
\put(36,113){\circle*{1.33}}
\put(36,123){\circle*{1.33}}
\put(36,133){\circle*{1.33}}
\put(36,143){\circle*{1.33}}
\put(36,153){\circle*{1.33}}
\put(46,53){\circle*{1.33}}
\put(46,63){\circle*{1.33}}
\put(56,23){\circle*{1.33}}
\put(66,43){\circle*{1.33}}
\put(66,53){\circle*{1.33}}
\put(66,63){\circle*{1.33}}
\put(66,73){\circle*{1.33}}
\put(66,83){\circle*{1.33}}
\put(66,93){\circle*{1.33}}
\put(66,103){\circle*{1.33}}
\put(66,113){\circle*{1.33}}
\put(66,123){\circle*{1.33}}
\put(66,133){\circle*{1.33}}
\put(66,143){\circle*{1.33}}
\put(66,153){\circle*{1.33}}
\put(76,3){\circle*{1.33}}
\put(76,13){\circle*{1.33}}
\put(76,23){\circle*{1.33}}
\put(76,33){\circle*{1.33}}
\put(76,43){\circle*{1.33}}
\put(76,53){\circle*{1.33}}
\put(76,63){\circle*{1.33}}
\put(76,73){\circle*{1.33}}
\put(76,83){\circle*{1.33}}
\put(76,93){\circle*{1.33}}
\put(76,103){\circle*{1.33}}
\put(76,113){\circle*{1.33}}
\put(76,123){\circle*{1.33}}
\put(76,133){\circle*{1.33}}
\put(76,143){\circle*{1.33}}
\gasset{AHnb=0,linewidth=0.4}
\drawline(6,145)(6,43)(36,33)(66,43)(66,145)
\drawline(16,53)(36,43)
\drawline(26,63)(26,53)(36,43)(46,53)(46,63)
\drawline(56,23)(76,13)(76,145)
\drawline(76,3)(76,13)(16,23)
\drawline(76,23)(36,33)(36,145)
\drawline[dash={0.2 0.8}{0}](6,145)(6,153)
\drawline[dash={0.2 0.8}{0}](36,145)(36,153)
\drawline[dash={0.2 0.8}{0}](66,145)(66,153)
\drawline[dash={0.2 0.8}{0}](76,145)(76,153)
\put(16,26){\makebox(0,0)[cc]{$\mathbf B_2=\mathbf{LRB}$}}
\put(55,23){\makebox(0,0)[rc]{$\overleftarrow{\mathbf B_2}=\mathbf{RRB}$}}
\put(77,13){\makebox(0,0)[lc]{$\mathbf C_1=\mathbf{SL}$}}
\put(77,23){\makebox(0,0)[lc]{$\mathbf C_2$}}
\put(77,33){\makebox(0,0)[lc]{$\mathbf C_3$}}
\put(77,43){\makebox(0,0)[lc]{$\mathbf C_4$}}
\put(77,53){\makebox(0,0)[lc]{$\mathbf C_5$}}
\put(77,63){\makebox(0,0)[lc]{$\mathbf C_6$}}
\put(77,73){\makebox(0,0)[lc]{$\mathbf C_7$}}
\put(77,83){\makebox(0,0)[lc]{$\mathbf C_8$}}
\put(77,93){\makebox(0,0)[lc]{$\mathbf C_9$}}
\put(77,103){\makebox(0,0)[lc]{$\mathbf C_{10}$}}
\put(77,113){\makebox(0,0)[lc]{$\mathbf C_{11}$}}
\put(77,123){\makebox(0,0)[lc]{$\mathbf C_{12}$}}
\put(77,133){\makebox(0,0)[lc]{$\mathbf C_{13}$}}
\put(77,143){\makebox(0,0)[lc]{$\mathbf C_{14}$}}
\put(36,156){\makebox(0,0)[cc]{$\mathbf D$}}
\put(36,30){\makebox(0,0)[cc]{$\mathbf D_1$}}
\put(37,42){\makebox(0,0)[lc]{$\mathbf D_2$}}
\put(37,53){\makebox(0,0)[lc]{$\mathbf D_3$}}
\put(37,63){\makebox(0,0)[lc]{$\mathbf D_4$}}
\put(37,73){\makebox(0,0)[lc]{$\mathbf D_5$}}
\put(37,83){\makebox(0,0)[lc]{$\mathbf D_6$}}
\put(37,93){\makebox(0,0)[lc]{$\mathbf D_7$}}
\put(37,103){\makebox(0,0)[lc]{$\mathbf D_8$}}
\put(37,113){\makebox(0,0)[lc]{$\mathbf D_9$}}
\put(37,123){\makebox(0,0)[lc]{$\mathbf D_{10}$}}
\put(37,133){\makebox(0,0)[lc]{$\mathbf D_{11}$}}
\put(37,143){\makebox(0,0)[lc]{$\mathbf D_{12}$}}
\put(6,156){\makebox(0,0)[cc]{$\mathbf F$}}
\put(66,156.5){\makebox(0,0)[cc]{$\overleftarrow{\mathbf F}$}}
\put(6,40){\makebox(0,0)[cc]{$\mathbf F_0$}}
\put(66,39.5){\makebox(0,0)[cc]{$\overleftarrow{\mathbf F_0}$}}
\put(5,53){\makebox(0,0)[rc]{$\mathbf F_1$}}
\put(65,53){\makebox(0,0)[rc]{$\overleftarrow{\mathbf F_1}$}}
\put(5,93){\makebox(0,0)[rc]{$\mathbf F_2$}}
\put(65,93){\makebox(0,0)[rc]{$\overleftarrow{\mathbf F_2}$}}
\put(5,143){\makebox(0,0)[rc]{$\mathbf F_3$}}
\put(65,143){\makebox(0,0)[rc]{$\overleftarrow{\mathbf F_3}$}}
\put(5,63){\makebox(0,0)[rc]{$\mathbf H_1$}}
\put(65,63){\makebox(0,0)[rc]{$\overleftarrow{\mathbf H_1}$}}
\put(5,103){\makebox(0,0)[rc]{$\mathbf H_2$}}
\put(65,103){\makebox(0,0)[rc]{$\overleftarrow{\mathbf H_2}$}}
\put(5,73){\makebox(0,0)[rc]{$\mathbf I_1$}}
\put(65,73){\makebox(0,0)[rc]{$\overleftarrow{\mathbf I_1}$}}
\put(5,113){\makebox(0,0)[rc]{$\mathbf I_2$}}
\put(65,113){\makebox(0,0)[rc]{$\overleftarrow{\mathbf I_2}$}}
\put(5,83){\makebox(0,0)[rc]{$\mathbf J_1^1$}}
\put(65,83){\makebox(0,0)[rc]{$\overleftarrow{\mathbf J_1^1}$}}
\put(5,123){\makebox(0,0)[rc]{$\mathbf J_2^1$}}
\put(65,123){\makebox(0,0)[rc]{$\overleftarrow{\mathbf J_2^1}$}}
\put(5,133){\makebox(0,0)[rc]{$\mathbf J_2^2$}}
\put(65,133){\makebox(0,0)[rc]{$\overleftarrow{\mathbf J_2^2}$}}
\put(16,56){\makebox(0,0)[cc]{$\mathbf M_5$}}
\put(26,66){\makebox(0,0)[cc]{$\mathbf M_6$}}
\put(46,66.5){\makebox(0,0)[cc]{$\overleftarrow{\mathbf M_6}$}}
\put(27,53.5){\makebox(0,0)[lc]{$\mathbf M_7$}}
\put(50,53.5){\makebox(0,0)[cc]{$\overleftarrow{\mathbf M_7}$}}
\put(76,0){\makebox(0,0)[cc]{$\mathbf T$}}
\end{picture}
\end{center}
\caption{All non-group chain varieties of monoids with $\mathbf T$}
\label{all chain var}
\end{figure}

We note that a number of chain varieties of monoids were found prior to the full classification in Theorem~\ref{chain descript}.
The lattice $L(\mathbf C_n)$ was described in Head~\cite{Head-68}, while the lattices $L(\mathbf B_2)$ and $L(\overleftarrow{\mathbf B_2})$ can be found in Wismath~\cite{Wismath-86} (see Theorems~\ref{struct of COM} and~\ref{struct of BAND}); the lattices $L(\mathbf D)$ and $L(\mathbf F_0)$ were described by Lee \cite{Lee-08,Lee-14}; and the lattices $L(\mathbf M_5)$ and $L(\mathbf M_7)$ were described by Jackson~\cite{Jackson-05b}.
The remaining varieties---$\mathbf F$, $\mathbf F_k$, $\mathbf H_k$, $\mathbf I_k$, $\mathbf J_k^m$, $\mathbf M_6$, and their duals---were found by Gusev and Vernikov~\cite{Gusev-Vernikov-18}.

Every non-group chain variety of semigroups is contained in a maximal chain variety, while any non-group non-chain variety of semigroups contains some minimal non-chain subvariety \cite[Corollary~2]{Sukhanov-82}.
However, these statements do not hold for monoid varieties.
Indeed, the variety $\mathbf C_n$ with $n\ge 3$ is not contained in a maximal chain variety (see Fig.~\ref{all chain var}), while non-chain varieties of monoids that do not contain any minimal non-chain subvariety exist \cite[Corollary~7.4]{Gusev-Vernikov-18}.
Another significant feature of the monoid case is that $\mathbf M_5$ is a non-finitely based non-group chain variety of monoids \cite[Proposition~5.1]{Jackson-05b}, while all non-group chain varieties of semigroups are finitely based.
Since the set of all finitely based group varieties is countably infinite, Theorem~\ref{chain gr uncount} implies that there are uncountably many group chain varieties without finite identity basis.
However, explicit examples of non-finitely based chain group varieties have not been found so far.

There is one more interesting consequence of Theorem~\ref{chain descript}.

\begin{corollary}[Gusev and Vernikov {\cite[Corollary~7.6]{Gusev-Vernikov-18}}]
\label{chain is loc fin}
Every non-group chain variety of monoids is contained in some finitely generated variety and so is locally finite.
\end{corollary}

But by Theorem~\ref{chain gr uncount}, this result does not hold for group varieties.

\subsection{Modularity}
\label{mod}
One of the most profound advances in the study of the lattice $\mathbb{SEM}$, due to Volkov in the early 1990s, is the complete description of semigroup varieties with a modular subvariety lattice.
A full account of its proof, along with several of related results, occupied seven articles that were published in 1989--2004; see Shevrin \textit{et~al.}\@ \cite[Section~11]{Shevrin-Vernikov-Volkov-09} for more details.

One would hope to replicate Volkov's achievement in the lattice $\mathbb{MON}$.

\begin{problem}
\label{mod subvar lat}
Describe varieties of monoids with a modular subvariety lattice.
\end{problem}

Presently, this problem is very far from being solved; some complications are highlighted at the end of Subsection~\ref{sublat}.
However, there exist results that are relevant to solving Problem~\ref{mod subvar lat}.
Recall that any monoid variety that is commutative or completely regular has a modular subvariety lattice; see Theorem~\ref{struct of COM} and Proposition~\ref{CR is mod}.
The first example of a monoid variety with a non-modular subvariety lattice was published by Lee~\cite{Lee-12b} in 2012: the variety $\mathbf A_0^1\vee\mathbf{LRB}\vee\mathbf{RRB}$ with $\mathbf A_0^1=\var A_0^1$, where
\[
A_0=\langle a,b\mid a^2=a,\,b^2=b,\,ba=0\rangle=\{a,b,ab,0\};
\]
see Fig.~\ref{L(A_0^1 vee LRB vee RRB)}.
In contrast, the first examples of semigroup varieties with a non-modular subvariety lattice were discovered independently by Je\v{z}ek~\cite{Jezek-69} and Schwabauer~\cite{Schwabauer-69} back in the late 1960s.

\begin{figure}[htb]
\unitlength=1mm
\linethickness{0.4pt}
\begin{center}
\begin{picture}(100,88)
\put(10,33){\circle*{1.33}}
\put(10,43){\circle*{1.33}}
\put(10,53){\circle*{1.33}}
\put(10,63){\circle*{1.33}}
\put(40,28){\circle*{1.33}}
\put(50,3){\circle*{1.33}}
\put(50,13){\circle*{1.33}}
\put(50,18){\circle*{1.33}}
\put(50,23){\circle*{1.33}}
\put(50,33){\circle*{1.33}}
\put(50,43){\circle*{1.33}}
\put(50,53){\circle*{1.33}}
\put(50,73){\circle*{1.33}}
\put(50,83){\circle*{1.33}}
\put(60,28){\circle*{1.33}}
\put(90,33){\circle*{1.33}}
\put(90,43){\circle*{1.33}}
\put(90,53){\circle*{1.33}}
\put(90,63){\circle*{1.33}}
\gasset{AHnb=0,linewidth=0.4}
\drawline(10,33)(50,53)(90,33)
\drawline(10,43)(10,63)(50,83)(90,63)(90,43)
\drawline(50,3)(50,13)(10,33)(10,43)(50,23)(50,13)(90,33)(90,43)(50,23)
\drawline(50,33)(50,43)(10,63)
\drawline(50,43)(90,63)
\drawline(50,53)(50,83)
\drawline(60,28)(10,53)(50,73)(90,53)(40,28)
\put(51,42){\makebox(0,0)[lc]{$\mathbf A_0^1$}}
\put(50,86){\makebox(0,0)[cc]{$\mathbf A_0^1\vee\mathbf{LRB}\vee\mathbf{RRB}$}}
\put(49,18){\makebox(0,0)[rc]{$\mathbf C_2$}}
\put(51,21){\makebox(0,0)[lc]{$\mathbf D_1$}}
\put(39,26){\makebox(0,0)[rc]{$\mathbf F_0$}}
\put(61,25){\makebox(0,0)[lc]{$\overleftarrow{\mathbf F_0}$}}
\put(9,33){\makebox(0,0)[rc]{$\mathbf{LRB}$}}
\put(91,33){\makebox(0,0)[lc]{$\mathbf{RRB}$}}
\put(51,11){\makebox(0,0)[lc]{$\mathbf{SL}$}}
\put(50,0){\makebox(0,0)[cc]{$\mathbf T$}}
\end{picture}
\end{center}
\caption{The lattice $L(\mathbf A_0^1\vee\mathbf{LRB}\vee\mathbf{RRB})$}
\label{L(A_0^1 vee LRB vee RRB)}
\end{figure}
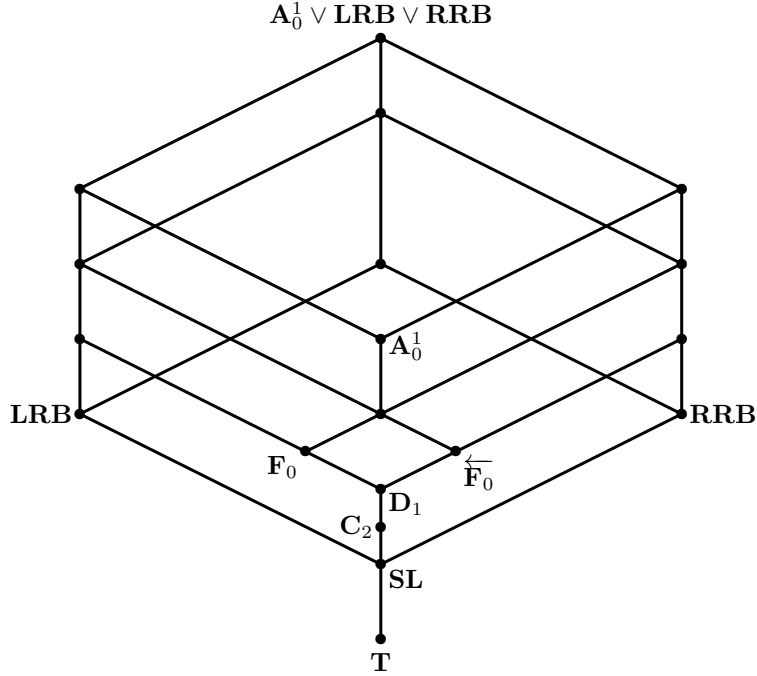

Now we formulate three necessary conditions for a monoid variety to have a modular subvariety lattice.

\begin{proposition}[M.\,V.\@ Volkov, unpublished]
\label{mod non-cr}
If $\mathbf X$ is a non-completely regular monoid variety with a modular subvariety lattice, then every completely regular subvariety of $\mathbf X$ is commutative.
\end{proposition}

\begin{proof}
This follows from Gusev \cite[Lemma~3.1]{Gusev-18b} and line~4 of Table~\ref{forbid var}.
\end{proof}

Note that the non-modularity of the lattice $L(\mathbf A_0^1\vee\mathbf{LRB}\vee\mathbf{RRB})$ follows from Proposition~\ref{mod non-cr}.

\begin{proposition}
\label{mod over C_3}
If $\mathbf X$ is a monoid variety with a modular subvariety lattice and $\mathbf C_3\subseteq\mathbf X$, then $\mathbf X$ is a variety of monoids with central idempotents.
\end{proposition}

\begin{proof}
This is a consequence of Gusev \cite[Lemma~2]{Gusev-20b} and Gusev and Vernikov \cite[Lemma~4.1 and Proposition~4.2]{Gusev-Vernikov-18}.
\end{proof}

\begin{proposition}
\label{mod over F_0}
If $\mathbf X$ is a monoid variety with a modular subvariety lattice and $\mathbf F_0\subseteq\mathbf X$, then $\mathbf X\subseteq\mathbf B_{2,k}$ for some $k \ge 3$.
\end{proposition}

\begin{proof}
This can be easily deduced from Gusev \cite[Lemma~2]{Gusev-20b} and Gusev and Vernikov \cite[Lemma~2.5]{Gusev-Vernikov-18}.
\end{proof}

To date, up to duality, there are only three explicit examples of monoid varieties with non-modular subvariety lattice that are not covered by Propositions \ref{mod non-cr}--\ref{mod over F_0}:
\[
\mathbf M_6\vee\overleftarrow{\mathbf M_7}, \quad \mathbf M_8 = \var S_8^1, \quad \text{and} \quad \mathbf M_9 = \var M_9,
\]
where
\begin{align*}
S_8 &=\left\langle a,b,c \left| 
\begin{array}{l} 
a^2=a,\,b^2=b^3,\,bcb^2=bcb,\\
ca=c,\,abc=ac=ba=b^2c=0 
\end{array} 
\right.\!\!\right\rangle \\
&=\{a,b,c,ab,ab^2,b^2,bc,bcb,cb,cb^2,0\} \\
\text{and}\enskip M_9 &=\langle a,g \,|\, a^3=0, \, g^2=1, \, ag=a, \, ga^2=a^2 \rangle =\{a,g,a^2,ga,0,1\}.
\end{align*}
The lattice $L(\mathbf M_6\vee\overleftarrow{\mathbf M_7})$, modulo the interval $[\mathbf M_7\vee\overleftarrow{\mathbf M_7},\mathbf M_6\vee\overleftarrow{\mathbf M_7}]$, is given in Gusev~\cite{Gusev-19}, while the lattices $L(\mathbf M_8)$ and $L(\mathbf M_9)$ are due to Gusev and O.\,B.\@ Sapir~\cite{Gusev-Sapir_O-22} and Lee~\cite{Lee-22+}, respectively; see Figs.~\ref{L(M_6 vee dual to M_7)}--\ref{L(M_9)}.
The undefined varieties in Figs.~\ref{L(M_8)} and~\ref{L(M_9)} are $\mathbf M_{10} = \var S_{10}^1$ and $\mathbf M_{11} = \var M_{11}$, where
\begin{align*}
S_{10}&=\left\langle a,b,c \left| 
\begin{array}{l} 
a^2=a^3,\,b^2=b^3,\,bcb^2=bcb,\\
c^2=ba=ca=ac=b^2c=ab^2=0 
\end{array} 
\right.\!\!\right\rangle\\
&=\{a,b,c,cb,cb^2,b^2,bc,bcb,ab,abc,abcb,a^2,a^2b,a^2bc,a^2bcb,0\} \\
\text{and}\enskip M_{11} & = \langle a,g \,|\, a^2=0, \, g^2=1, \, ag=a \rangle =\{a,g,ga,0,1\}.
\end{align*}

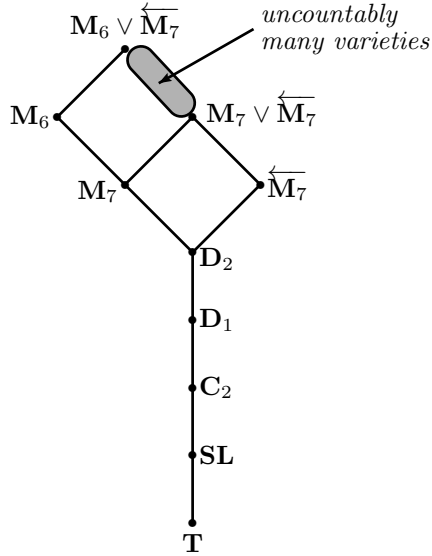
\begin{figure}[htb]
\unitlength=0.9mm
\linethickness{0.4pt}
\begin{center}
\begin{picture}(60,80)
\put(7,63){\circle*{1.33}}
\put(17,53){\circle*{1.33}}
\put(17,73){\circle*{1.33}}
\put(27,3){\circle*{1.33}}
\put(27,13){\circle*{1.33}}
\put(27,23){\circle*{1.33}}
\put(27,33){\circle*{1.33}}
\put(27,43){\circle*{1.33}}
\put(27,63){\circle*{1.33}}
\put(37,53){\circle*{1.33}}
\gasset{AHnb=0,linewidth=0.4}
\drawline(27,3)(27,43)(7,63)(17,73)
\drawline(27,43)(37,53)(27,63)(17,53)
\put(22,68)
{
\begin{rotate}{-47}
\drawoval[fillgray=0.7](-0.85,-0.55,12.7,4,2)
\end{rotate}
}
\gasset{AHnb=1,AHLength=2,linewidth=0.4}
\drawline(36,76)(22,68)
\put(28,23){\makebox(0,0)[lc]{$\mathbf C_2$}}
\put(28,33){\makebox(0,0)[lc]{$\mathbf D_1$}}
\put(28,42){\makebox(0,0)[lc]{$\mathbf D_2$}}
\put(6,63){\makebox(0,0)[rc]{$\mathbf M_6$}}
\put(17,77){\makebox(0,0)[cc]{$\mathbf M_6\vee\overleftarrow{\mathbf M_7}$}}
\put(16,52){\makebox(0,0)[rc]{$\mathbf M_7$}}
\put(29,64){\makebox(0,0)[lc]{$\mathbf M_7\vee\overleftarrow{\mathbf M_7}$}}
\put(38,53){\makebox(0,0)[lc]{$\overleftarrow{\mathbf M_7}$}}
\put(28,13){\makebox(0,0)[lc]{$\mathbf{SL}$}}
\put(27,0){\makebox(0,0)[cc]{$\mathbf T$}}
\put(37,78){\makebox(0,0)[lc]{\emph{uncountably}}}
\put(37,74){\makebox(0,0)[lc]{\emph{many varieties}}}
\end{picture}
\end{center}
\caption{The lattice $L(\mathbf M_6\vee\protect\overleftarrow{\mathbf M_7})$}
\label{L(M_6 vee dual to M_7)}
\end{figure}

\begin{figure}[htb]
\unitlength=0.9mm
\linethickness{0.4pt}
\begin{center}
\begin{picture}(41,88)
\put(5,53){\circle*{1.33}}
\put(15,43){\circle*{1.33}}
\put(15,63){\circle*{1.33}}
\put(20,68){\circle*{1.33}}
\put(25,3){\circle*{1.33}}
\put(25,13){\circle*{1.33}}
\put(25,23){\circle*{1.33}}
\put(25,33){\circle*{1.33}}
\put(25,53){\circle*{1.33}}
\put(25,73){\circle*{1.33}}
\put(25,83){\circle*{1.33}}
\put(35,43){\circle*{1.33}}
\put(35,63){\circle*{1.33}}
\gasset{AHnb=0,linewidth=0.4}
\drawline(25,3)(25,33)(5,53)(25,73)(35,63)(15,43)
\drawline(25,33)(35,43)(15,63)
\drawline(25,73)(25,83)
\put(36.5,63){\makebox(0,0)[lc]{$\mathbf A_0^1$}}
\put(26.5,23){\makebox(0,0)[lc]{$\mathbf C_2$}}
\put(26.5,32){\makebox(0,0)[lc]{$\mathbf D_1$}}
\put(13.5,43){\makebox(0,0)[rc]{$\mathbf F_0$}}
\put(36.5,43.5){\makebox(0,0)[lc]{$\overleftarrow{\mathbf F_0}$}}
\put(4,53){\makebox(0,0)[rc]{$\mathbf F_1$}}
\put(25,86){\makebox(0,0)[cc]{$\mathbf M_8$}}
\put(18.5,69){\makebox(0,0)[rc]{$\mathbf M_{10}$}}
\put(26.5,13){\makebox(0,0)[lc]{$\mathbf{SL}$}}
\put(25,0){\makebox(0,0)[cc]{$\mathbf T$}}
\end{picture}
\end{center}
\caption{The lattice $L(\mathbf M_8)$}
\label{L(M_8)}
\end{figure}
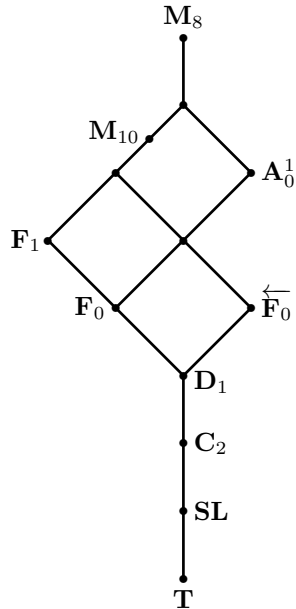

\begin{figure}[htb]
\unitlength=0.9mm
\linethickness{0.4pt}
\begin{center}
\begin{picture}(47,78)
\put(2,33){\circle*{1.33}}
\put(2,43){\circle*{1.33}}
\put(2,53){\circle*{1.33}}
\put(12,3){\circle*{1.33}}
\put(12,13){\circle*{1.33}}
\put(12,23){\circle*{1.33}}
\put(12,33){\circle*{1.33}}
\put(12,43){\circle*{1.33}}
\put(17,43){\circle*{1.33}}
\put(17,53){\circle*{1.33}}
\put(17,63){\circle*{1.33}}
\put(27,13){\circle*{1.33}}
\put(27,23){\circle*{1.33}}
\put(27,33){\circle*{1.33}}
\put(27,43){\circle*{1.33}}
\put(27,53){\circle*{1.33}}
\put(32,73){\circle*{1.33}}
\put(42,53){\circle*{1.33}}
\put(42,63){\circle*{1.33}}
\gasset{AHnb=0,linewidth=0.4}
\drawline(2,33)(2,53)(32,73)(42,63)(42,53)(27,43)(27,13)(12,3)(12,23)(2,33)(17,43)(17,63)(27,53)(27,43)(12,33)(12,43)(42,63)
\drawline(2,43)(12,33)(12,13)(27,23)
\drawline(12,23)(27,33)(17,43)
\drawline(2,43)(17,53)(27,43)
\drawline(2,53)(12,43)
\put(27,10){\makebox(0,0)[cc]{$\mathbf A_2$}}
\put(11.5,21.5){\makebox(0,0)[rc]{$\mathbf C_2$}}
\put(2,30){\makebox(0,0)[cc]{$\mathbf C_3$}}
\put(11.5,31.5){\makebox(0,0)[rc]{$\mathbf D_1$}}
\put(11.5,41.5){\makebox(0,0)[rc]{$\mathbf D_2$}}
\put(32,76){\makebox(0,0)[cc]{$\mathbf M_9$}}
\put(44,50){\makebox(0,0)[cc]{$\mathbf M_{11}$}}
\put(11,13){\makebox(0,0)[rc]{$\mathbf{SL}$}}
\put(12,0){\makebox(0,0)[cc]{$\mathbf T$}}
\end{picture}
\end{center}
\caption{The lattice $L(\mathbf M_9)$}
\label{L(M_9)}
\end{figure}
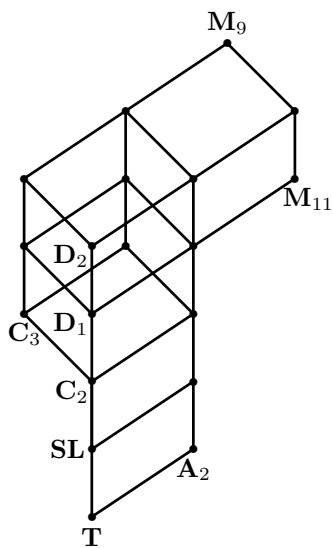

A variety $\mathbf X$ of algebras is \emph{almost modular} if its subvariety lattice is not modular but the subvariety lattice of each proper subvariety of $\mathbf X$ is modular.
Up to duality, $\mathbf C_2\vee\mathbf{LRB}$, $\mathbf A_0^1\vee\mathbf F_1$, and $\mathbf M_9$ are the only monoid varieties currently known to be almost modular.
The varieties $\mathbf C_2\vee\mathbf{LRB}$ and $\mathbf A_0^1\vee\mathbf F_1$ are aperiodic, while $\mathbf M_9$ is of period two.

\begin{question}
\label{generalize M_9?}
Is there a way to generalize $\mathbf M_9$ to an almost modular variety of any odd prime period?
\end{question}

In conclusion of this subsection, we suggest the following problems as feasible stepping stones toward solving Problem~\ref{mod subvar lat}.

\begin{problem}
\label{mod subvar lat partial}
\quad
\begin{itemize}
\item[a)] Describe varieties of monoids with central idempotents whose subvariety lattice is modular.
\item[b)] Describe overcommutative varieties of monoids whose subvariety lattice is modular.
\end{itemize}
\end{problem}

In view of Proposition~\ref{mod over C_3}, Part~b) of Problem~\ref{mod subvar lat partial} is a particular case of Part~a).

\subsection{Distributivity}
\label{distr}
Volkov's contributions to semigroup varieties with a modular subvariety lattice, mentioned at the beginning of Subsection~\ref{mod}, contains essential information about semigroup varieties with a distributive subvariety lattice, which results in an almost complete description modulo group varieties; see Shevrin \textit{et~al.}\@ \cite[Section~11]{Shevrin-Vernikov-Volkov-09}.

Although the lattice $\mathbb{GR}$ of periodic group varieties is modular, a description of periodic group varieties with distributive subvariety lattice has remained very elusive, especially in view of Theorem~\ref{chain gr uncount}.
Therefore, in the investigation of distributive lattices of monoid varieties, it is logical to focus on non-group varieties.

\begin{problem}
\label{distr subvar lat}
Describe varieties of monoids with a distributive subvariety lattice modulo varieties of groups.
\end{problem}

\begin{problem}
\label{distr subvar lat aperiodic}
Describe aperiodic varieties of monoids with a distributive subvariety lattice.
\end{problem}

Similar to Problem~\ref{mod subvar lat}, both Problems~\ref{distr subvar lat} and~\ref{distr subvar lat aperiodic} are very far from being completely solved.
Some reasons that complicate the solution of these problems are highlighted at the end of Subsection~\ref{sublat}.

Recall from Theorems~\ref{struct of COM} and~\ref{struct of BAND} that the lattices $\mathbb{COM}$ and $\mathbb{BAND}$ are distributive.
Overcommutative varieties with a distributive subvariety lattice also exist, obvious examples are the variety $\mathbf{COM}$ and its unique cover $\mathbf{COM}\vee\mathbf D_1$; see Remark~\ref{COM vee D_1}.
In contrast, the subvariety lattice of every overcommutative semigroup variety does not satisfy any non-trivial identity~\cite{Burris-Nelson-71b}.
Another class of varieties with a distributive subvariety lattice is the class of chain varieties.
In addition to commutative, band, and chain varieties, a number of examples of monoid varieties with the distributive subvariety lattice are given in or can be extracted from Figs.~\ref{L(A_0^1 vee LRB vee RRB)}--\ref{L(E)}.
Notable progress in solving Problem~\ref{distr subvar lat aperiodic} has recently been made when a characterization was found for all varieties of aperiodic monoids with central idempotents whose subvariety lattice is distributive~\cite{Gusev-22+b}.
To describe these varieties, several new words are required: define $S_0 = S_1$ and for $m,k,\ell \in \mathbb N \cup \{ 0 \}$ and $\rho \in S_{m+k+\ell}$, let
\begin{align*}
\mathbf c_{m,k,\ell}(\rho)&=\biggl(\prod_{i=1}^m z_it_i\biggr)xyt\biggl(\prod_{i=m+1}^{m+k} z_it_i\biggr)x\biggl(\prod_{i=1}^{m+k+\ell} z_{i\rho}\biggr)y\biggl(\prod_{i=m+k+1}^{m+k+\ell} t_iz_i\biggr)\\[-3pt]
\text{and}\enskip\mathbf c_{m,k,\ell}^\prime(\rho)&=\biggl(\prod_{i=1}^m z_it_i\biggr)yxt\biggl(\prod_{i=m+1}^{m+k} z_it_i\biggr)x\biggl(\prod_{i=1}^{m+k+\ell} z_{i\rho}\biggr)y\biggl(\prod_{i=m+k+1}^{m+k+\ell} t_iz_i\biggr),
\end{align*}
and let $\overleftarrow{\mathbf c}_{m,k,\ell}(\rho)$ and $\overleftarrow{\mathbf c}_{m,k,\ell}'(\rho)$ be the words obtained by writing $\mathbf c_{m,k,\ell}(\rho)$ and $\mathbf c_{m,k,\ell}^\prime(\rho)$ in reverse order.
Then for each $n \in \mathbb N$, define the varieties
\begin{align*}
\mathbf Q_n & =\var\left\{\!\!
\begin{array}{l}
x^n\approx x^{n+1},\,x^2y\approx yx^2,\\
\mathbf w_n(\pi,\tau)\approx\mathbf w_n'(\pi,\tau),\\
\mathbf c_{m,k,\ell}(\rho)\approx\mathbf c_{m,k,\ell}'(\rho),\\[1pt]
\overleftarrow{\mathbf c}_{m,k,\ell}(\rho)\approx \overleftarrow{\mathbf c}_{m,k,\ell}'(\rho)
\end{array}
\middle\vert
\begin{array}{l}
n\in\mathbb N,\,\pi,\tau\in S_n,\\
m,k,\ell\in\mathbb N\cup\{0\},\\
\rho\in S_{m+k+\ell}
\end{array}
\!\!\!\right\},\\
\mathbf R_n & =\var\{x^n\approx x^{n+1},\,x^ny\approx yx^n,\,x^2y\approx xyx\},\\
\text{and}\enskip\mathbf S_n & =\var\{x^n\approx x^{n+1},\,x^2y\approx yx^2,\,\sigma_2,\,\sigma_3\}.
\end{align*}

\begin{theorem}[Gusev {\cite[Theorem~1.1]{Gusev-22+b}}]
\label{A_cen}
A variety of aperiodic monoids with central idempotents has a distributive subvariety lattice if and only if it is contained in $\mathbf Q_n$, $\mathbf R_n$, $\overleftarrow{\mathbf R_n}$, $\mathbf S_n$, or $\overleftarrow{\mathbf S_n}$ for some $n \in \mathbb N$.
\end{theorem}

The structure of the lattices $L(\mathbf Q_n)$, $L(\mathbf R_n)$, and $L(\mathbf S_n)$ are so complicated that it is impossible to fully illustrate them with clarity.
But it is possible to exhibit the subvariety lattice of some subvarieties of $\mathbf S_n$, such as $\mathbf D\vee\mathbf M_6$ \cite[Corollary~5.3]{Gusev-Vernikov-21}; see Fig.~\ref{L(D vee M_6)}.

\begin{figure}[htb]
\unitlength=1mm
\linethickness{0.4pt}
\begin{center}
\begin{picture}(34,91)
\put(7,3){\circle*{1.33}}
\put(7,13){\circle*{1.33}}
\put(7,23){\circle*{1.33}}
\put(7,33){\circle*{1.33}}
\put(7,43){\circle*{1.33}}
\put(7,53){\circle*{1.33}}
\put(7,63){\circle*{1.33}}
\put(7,76){\circle*{1.33}}
\put(17,48){\circle*{1.33}}
\put(17,58){\circle*{1.33}}
\put(17,68){\circle*{1.33}}
\put(17,81){\circle*{1.33}}
\put(27,53){\circle*{1.33}}
\put(27,63){\circle*{1.33}}
\put(27,73){\circle*{1.33}}
\put(27,86){\circle*{1.33}}
\gasset{AHnb=0,linewidth=0.4}
\drawline(7,3)(7,65)
\drawline(7,43)(27,53)(27,75)
\drawline(7,53)(27,63)
\drawline(7,63)(27,73)
\drawline(7,76)(27,86)
\drawline(17,48)(17,70)
\drawline[dash={0.2 0.8}{0}](7,63)(7,76)
\drawline[dash={0.2 0.8}{0}](17,68)(17,81)
\drawline[dash={0.2 0.8}{0}](27,73)(27,86)
\put(6,23){\makebox(0,0)[rc]{$\mathbf C_2$}}
\put(6,76){\makebox(0,0)[rc]{$\mathbf D$}}
\put(27,89){\makebox(0,0)[cc]{$\mathbf D\vee\mathbf M_6$}}
\put(19,84){\makebox(0,0)[rc]{$\mathbf D\vee\mathbf M_7$}}
\put(6,33){\makebox(0,0)[rc]{$\mathbf D_1$}}
\put(6,43){\makebox(0,0)[rc]{$\mathbf D_2$}}
\put(6,53){\makebox(0,0)[rc]{$\mathbf D_3$}}
\put(6,63){\makebox(0,0)[rc]{$\mathbf D_4$}}
\put(28,50){\makebox(0,0)[cc]{$\mathbf M_6$}}
\put(18,45){\makebox(0,0)[cc]{$\mathbf M_7$}}
\put(6,13){\makebox(0,0)[rc]{$\mathbf{SL}$}}
\put(7,0){\makebox(0,0)[cc]{$\mathbf T$}}
\end{picture}
\end{center}
\caption{The lattice $L(\mathbf D\vee\mathbf M_6)$}
\label{L(D vee M_6)}
\end{figure}
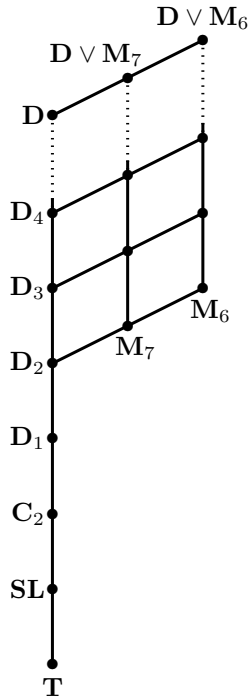

The following provides some reasonable problems the solutions of which would further contribute toward solving Problem~\ref{distr subvar lat} completely.

\begin{problem}
\label{distr subvar lat tasks}
\quad
\begin{itemize}
\item[a)] Describe varieties of aperiodic monoids with commuting idempotents whose subvariety lattice is distributive.
\item[b)] Describe overcommutative varieties of monoids whose subvariety lattice is distributive.
\end{itemize}
\end{problem}

Besides characterizing varieties having a distributive subvariety lattice, another approach toward solving Problem~\ref{distr subvar lat} is to locate substantial examples of varieties with a subvariety lattice that is modular but not distributive.
The first and currently only known example was recently found: the variety $\mathbf M_{12}\vee\overleftarrow{\mathbf M_{12}}$, where $\mathbf M_{12}=\var S(x^2y)$.

\begin{proposition}[Gusev and Lee~\cite{Gusev-Lee-21}]
\label{mod not distr}
The lattice $L(\mathbf M_{12}\vee\overleftarrow{\mathbf M_{12}})$ is modular but not distributive; see Fig.~\ref{L(M_12 vee dual to M_12)}.
\end{proposition}

\begin{figure}[htb]
\unitlength=1mm
\linethickness{0.4pt}
\begin{center}
\begin{picture}(59,69)
\put(6,43){\circle*{1.33}}
\put(21,33){\circle*{1.33}}
\put(21,53){\circle*{1.33}}
\put(36,3){\circle*{1.33}}
\put(36,13){\circle*{1.33}}
\put(36,23){\circle*{1.33}}
\put(36,43){\circle*{1.33}}
\put(36,53){\circle*{1.33}}
\put(36,63){\circle*{1.33}}
\put(51,33){\circle*{1.33}}
\put(51,53){\circle*{1.33}}
\gasset{AHnb=0,linewidth=0.4}
\drawline(21,33)(36,43)(36,63)
\drawline(21,53)(36,43)
\drawline(36,3)(36,23)(6,43)(36,63)(51,53)(36,43)(51,33)(36,23)
\put(37,21){\makebox(0,0)[lc]{$\mathbf C_2$}}
\put(52,33){\makebox(0,0)[lc]{$\mathbf C_3$}}
\put(20,32){\makebox(0,0)[rc]{$\mathbf D_1$}}
\put(5,43){\makebox(0,0)[rc]{$\mathbf D_2$}}
\put(37,53){\makebox(0,0)[lc]{$\mathbf M_{12}$}}
\put(36,66){\makebox(0,0)[cc]{$\mathbf M_{12}\vee\overleftarrow{\mathbf M_{12}}$}}
\put(52,54){\makebox(0,0)[lc]{$\overleftarrow{\mathbf M_{12}}$}}
\put(37,13){\makebox(0,0)[lc]{$\mathbf{SL}$}}
\put(36,0){\makebox(0,0)[cc]{$\mathbf T$}}
\end{picture}
\end{center}
\caption{The lattice $L(\mathbf M_{12}\vee\protect\overleftarrow{\mathbf M_{12}})$}
\label{L(M_12 vee dual to M_12)}
\end{figure}
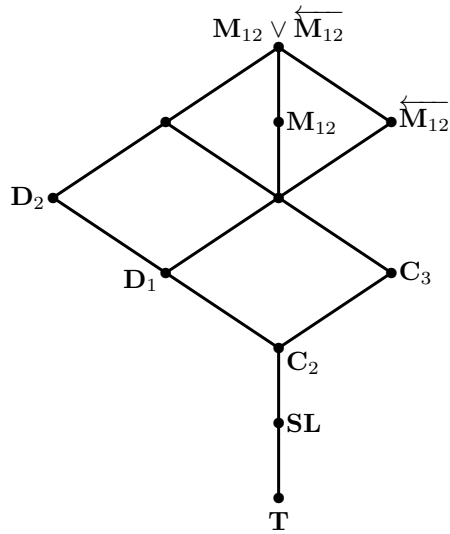

\subsection{Other identities and related restrictions}
\label{other ident}
Besides the distributive and modular laws, another important lattice identity is the Arguesian law; see Gr\"atzer \cite[Subsection~V.4.4]{Gratzer-11}, for instance.
It is general knowledge that any Arguesian lattice is modular but the converse is false.
For subvariety lattices of semigroup varieties, the properties of being modular and Arguesian are equivalent \cite[Theorem~11.2]{Shevrin-Vernikov-Volkov-09}.
However, it is unknown if the same holds true for subvariety lattices of monoid varieties.

\begin{question}
\label{mod not arg?}
Is there a variety of monoids whose subvariety lattice is modular but not Arguesian?
\end{question}

The following result concerning a variety $\mathbf X$ and its greatest group subvariety $\mathsf{Gr}(\mathbf X)$ can be deduced from Rasin \cite[Corollary~5]{Rasin-82} and Proposition~\ref{MON is sublat of SEM} or from Pol\'ak~\cite{Polak-88} and Proposition~\ref{UCR is sublat of UCR_sem}.

\begin{proposition}
\label{ident orthodox}
For any variety $\mathbf X$ of orthodox completely regular monoids, the lattices $L(\mathbf X)$ and $L(\mathsf{Gr}(\mathbf X))$ satisfy the same non-trivial identities.
\end{proposition}

It is unknown if Proposition~\ref{ident orthodox} holds for other non-orthodox varieties.

\begin{question}
\label{ident cr?}
Is the analog of Proposition~\ref{ident orthodox} true for arbitrary varieties of completely regular monoids?
\end{question}

In order to estimate the complexness of the class of monoid varieties whose subvariety lattice satisfies a non-trivial identity, it is useful to establish whether or not this class is closed under joins.
The following question is presently open.

\begin{question}
\label{ident join?}
Are there monoid varieties $\mathbf X_1$ and $\mathbf X_2$ such that each of the lattices $L(\mathbf X_1)$ and $L(\mathbf X_2)$ satisfies some non-trivial identity, while the lattice $L(\mathbf X_1\vee\mathbf X_2)$ does not satisfy any non-trivial identity?
\end{question}

For comparison, we note that the analogous question for semigroup varieties is answered in the affirmative.
For instance, it is long known that the subvariety lattices of $\mathbf X_1 = \var_\mathsf{sem} L_2^1$ and $\mathbf X_2 = \var_\mathsf{sem} S(x)$ are distributive and of width two~\cite{Evans-71}, but since the join $\mathbf X_1 \vee \mathbf X_2$ is finitely universal \cite[Subsection~6.3]{Gusev-Lee-20}, its subvariety lattice does not satisfy any non-trivial identity.

Recall that a lattice $\langle L;\vee,\wedge\rangle$ is
\begin{align*}
&\text{\emph{upper semimodular} if}&&\forall x,y\in L:\,x\enskip\text{covers}\enskip x\wedge y\longrightarrow x\vee y\enskip\text{covers}\enskip y;\\
&\text{\emph{lower semimodular} if}&&\forall x,y\in L:\,x\vee y\enskip\text{covers}\enskip x\longrightarrow y\enskip\text{covers}\enskip x\wedge y.
\end{align*}
The subvariety lattice of a semigroup variety is upper semimodular if and only if it is modular, while there is a semigroup variety whose subvariety lattice is lower semimodular but not modular; see Shevrin \textit{et~al.}\@ \cite[Subsection~11.3]{Shevrin-Vernikov-Volkov-09}.
Analog of the latter claim for monoid varieties is true because the lattice $L(\mathbf M_9)$ is lower semimodular but not modular; see Fig.~\ref{L(M_9)}.
We do not know whether the analog of the former claim holds true or not.

\begin{question}
\label{umod not mod?}
Is there a monoid variety whose subvariety lattice is upper semimodular but not modular?
\end{question}

A number of other restrictions to a subvariety lattice related to the subject matter of this section were considered for lattices of semigroup varieties with varying degrees of success.
These restrictions include upper and lower semidistributivity, belonging to an arbitrary given quasi-variety of modular lattices or certain other lattice quasi-varieties, and the property of having width two; see Shevrin \textit{et~al.}\@ \cite[Section~11]{Shevrin-Vernikov-Volkov-09}.
For monoid varieties, all these restrictions have not yet been considered in the literature.

\section{Finiteness conditions}
\label{finite}

As usual, by \emph{finiteness condition} for lattices, we mean a lattice property that holds in every finite lattice.
Some of the most important finiteness conditions on a lattice include the property of being finite, the ascending chain condition, and the descending chain condition.
Another interesting finiteness condition is the property of having a \emph{finite width} in the sense that all anti-chains in the lattice are finite.
The investigation of varieties of semigroups or monoids with such restrictions on their subvariety lattices seems important and interesting, but it turns out to be very difficult.
As a result, for each of the above finiteness condition $\theta$, semigroup varieties whose subvariety lattice satisfies $\theta$ have so far not been completely described.
Moreover, all corresponding problems are very far from being completely solved, even modulo group varieties; see Shevrin \textit{et~al.}\@ \cite[Section~10]{Shevrin-Vernikov-Volkov-09}.

Regarding subvariety lattices of monoid varieties, none of the aforementioned finiteness conditions has been systematically examined.
For the finite width condition, not much is known beyond chain varieties.
For the other three conditions---being finite and satisfying the ascending chain and descending chain conditions---only several examples are published and analogs of Proposition~\ref{ident orthodox} can be easily deduced from known results.
Since very few examples are available, we are able to describe them all in this section.

\subsection{Small varieties}
\label{small}
A variety of algebras with a finite subvariety lattice is said to be \emph{small}.
Small varieties of monoids have not yet been specifically studied.
However, results of Rasin~\cite{Rasin-82} and Proposition~\ref{MON is sublat of SEM} (or results of Pol\'ak~\cite{Polak-88} and Proposition~\ref{UCR is sublat of UCR_sem}) imply the following result.

\begin{proposition}
\label{small orthodox}
For any variety $\mathbf X$ of orthodox completely regular monoids, $\mathbf X$ is small if and only if $\mathsf{Gr}(\mathbf X)$ is small and $\mathbf{BAND}\nsubseteq\mathbf X$.
\end{proposition}

It is unknown if Proposition~\ref{small orthodox} holds for other non-orthodox varieties.

\begin{question}
\label{small cr?}
Is the analog of Proposition~\ref{small orthodox} true for arbitrary varieties of completely regular monoids?
\end{question}

A \emph{Cross variety} is a variety that is finitely based, finitely generated, and small.
Practically, all current information on small varieties outside the completely regular case is a certain number of examples obtained in the study of Cross varieties.
Every finitely generated group variety is Cross~\cite{Oates-Powell-64}, but the analog of this result does not hold for semigroup or monoid varieties; for instance, there exist finitely generated varieties of semigroups~\cite{Trakhtman-88} and of monoids \cite{Jackson-Lee-18,Jackson-Zhang-21} that are both non-finitely based and with uncountably many subvarieties.

Cross monoid varieties constitute an important subclass of the class of small monoid varieties.
A non-Cross variety is \emph{almost Cross} if all its proper subvarieties are Cross.
It is natural to investigate almost Cross varieties since by Zorn's lemma, every non-Cross variety contains an almost Cross subvariety.
Many articles are devoted to Cross or almost Cross monoid varieties.
We will not survey all these works here since it is outside the scope of this survey.
We restrict our attention to only articles in which the Cross or almost Cross property of a variety is established through a complete description of its subvariety lattice.
Further, we will only exhibit here Cross varieties and almost Cross varieties that have not appeared above.

Most known examples of almost Cross monoid varieties arise in the study of \emph{limit varieties}, that is, non-finitely based varieties whose proper subvarieties are all finitely based.
It follows from Zorn's lemma that every non-finitely based variety contains a limit subvariety; this is a main motivation to study limit varieties.

By Theorem~\ref{chain gr uncount}, there are uncountably many limit varieties of periodic groups.
However, explicit examples of such varieties have not yet been found.
Presently, up to duality, there are only five explicit examples of non-group limit varieties of monoids whose subvariety lattices are known.
The first two of these five examples, due to Jackson~\cite{Jackson-05b}, are the varieties $\mathbf M_5$ and $\mathbf M_7\vee\overleftarrow{\mathbf M_7}$; their subvariety lattices can be found inside Figs.~\ref{all chain var} and~\ref{L(M_6 vee dual to M_7)}, respectively.

The third limit monoid variety, due to Zhang and Luo~\cite{Zhang-Luo-arx}, is the variety $\mathbf M_{13}\vee\overleftarrow{\mathbf M_{13}}$, where $\mathbf M_{13}$ is generated by the monoid $S^1$ obtained from
\[
S=\langle a,b,c\mid a^2=a,\,b^2=b,\,ab=ca=0,\,ac=cb=c\rangle=\{a,b,c,ba,bc,0\}.
\]
The lattice $L(\mathbf M_{13}\vee\overleftarrow{\mathbf M_{13}})$ is shown in Fig.~\ref{L(M_13 vee dual to M_13)}; the undefined variety here is $\mathbf Q=\var Q^1$, where
\[
Q=\langle a,b,c\mid a^2=a,\,ab=b,\,ca=c,\,ac=ba=cb=0\rangle=\{a,b,c,bc,0\}.
\]

\begin{figure}[htb]
\unitlength=1mm
\linethickness{0.4pt}
\begin{center}
\begin{picture}(36,99)
\put(8,43){\circle*{1.33}}
\put(8,63){\circle*{1.33}}
\put(8,83){\circle*{1.33}}
\put(18,3){\circle*{1.33}}
\put(18,13){\circle*{1.33}}
\put(18,23){\circle*{1.33}}
\put(18,33){\circle*{1.33}}
\put(18,53){\circle*{1.33}}
\put(18,73){\circle*{1.33}}
\put(18,93){\circle*{1.33}}
\put(28,43){\circle*{1.33}}
\put(28,63){\circle*{1.33}}
\put(28,83){\circle*{1.33}}
\gasset{AHnb=0,linewidth=0.4}
\drawline(18,3)(18,33)(28,43)(8,63)(28,83)(18,93)(8,83)(28,63)(8,43)(18,33)
\put(7,63){\makebox(0,0)[rc]{$\mathbf A_0^1$}}
\put(19,23){\makebox(0,0)[lc]{$\mathbf C_2$}}
\put(19,32){\makebox(0,0)[lc]{$\mathbf D_1$}}
\put(7,43){\makebox(0,0)[rc]{$\mathbf F_0$}}
\put(29,43){\makebox(0,0)[lc]{$\overleftarrow{\mathbf F_0}$}}
\put(7,83){\makebox(0,0)[rc]{$\mathbf M_{13}$}}
\put(18,96){\makebox(0,0)[cc]{$\mathbf M_{13}\vee\overleftarrow{\mathbf M_{13}}$}}
\put(29,83){\makebox(0,0)[lc]{$\overleftarrow{\mathbf M_{13}}$}}
\put(29,63){\makebox(0,0)[lc]{$\mathbf Q$}}
\put(19,13){\makebox(0,0)[lc]{$\mathbf{SL}$}}
\put(18,0){\makebox(0,0)[cc]{$\mathbf T$}}
\end{picture}
\end{center}
\caption{The lattice $L(\mathbf M_{13}\vee\protect\overleftarrow{\mathbf M_{13}})$}
\label{L(M_13 vee dual to M_13)}
\end{figure}
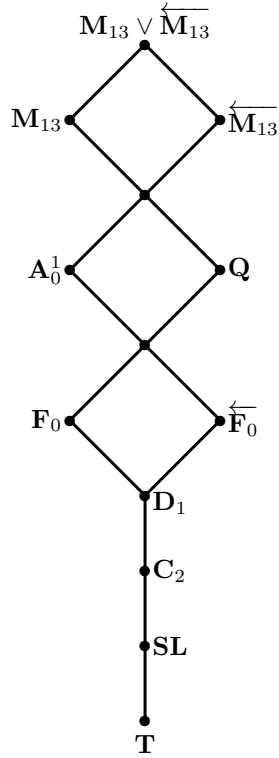

The fourth example of a limit monoid variety, constructed by Gusev~\cite{Gusev-20a}, is
\[
\mathbf M_{14}=\var
\left\{\!\!
\begin{array}{l}
x^2y^2\approx y^2x^2,\ xyx\approx xyx^2,\ xyzxy\approx yxzxy,\\
xyxztx\approx xyxzxtx,\ \mathbf v_n(\pi)\approx\mathbf v_n'(\pi)
\end{array}
\!\middle\vert\!
\begin{array}{l}
n\in\mathbb N,\\
\pi\in S_n
\end{array}
\!\!\right\},
\]
where
\begin{align*}
\mathbf v_n(\pi) & =xz_{\pi(1)} z_{\pi(2)}\cdots z_{\pi(n)}x\biggl(\,\prod_{i=1}^nt_iz_i\biggr) \\
\text{and}\enskip\mathbf v_n'(\pi) & =x^2z_{\pi(1)} z_{\pi(2)}\cdots z_{\pi(n)}\biggl(\,\prod_{i=1}^nt_iz_i\biggr).
\end{align*}
The lattice $L(\mathbf M_{14})$ is shown in Fig.~\ref{L(M_14)}, where
\[
\mathbf M_{15}=\mathbf M_{14}\wedge\var\{\sigma_3\}\enskip\text{and}\enskip\mathbf M_{16}=\mathbf M_{14}\wedge\var\{yx^2zy\approx xyxzy\}.
\]
Finite monoids that generate the variety $\mathbf M_{14}$ and each of its subvarieties can be found in O.\,B.\@ Sapir \cite[Theorem~7.1]{Sapir_O-21}.

\begin{figure}[htb]
\unitlength=1mm
\linethickness{0.4pt}
\begin{center}
\begin{picture}(40,98)
\put(5,53){\circle*{1.33}}
\put(15,43){\circle*{1.33}}
\put(15,63){\circle*{1.33}}
\put(15,73){\circle*{1.33}}
\put(15,83){\circle*{1.33}}
\put(15,93){\circle*{1.33}}
\put(25,3){\circle*{1.33}}
\put(25,13){\circle*{1.33}}
\put(25,23){\circle*{1.33}}
\put(25,33){\circle*{1.33}}
\put(25,53){\circle*{1.33}}
\put(35,43){\circle*{1.33}}
\gasset{AHnb=0,linewidth=0.4}
\drawline(15,43)(25,53)
\drawline(25,3)(25,33)(5,53)(15,63)(15,93)
\drawline(25,33)(35,43)(15,63)
\put(26,23){\makebox(0,0)[lc]{$\mathbf C_2$}}
\put(26,32){\makebox(0,0)[lc]{$\mathbf D_1$}}
\put(14,43){\makebox(0,0)[rc]{$\mathbf F_0$}}
\put(36,43){\makebox(0,0)[lc]{$\overleftarrow{\mathbf F_0}$}}
\put(4,53){\makebox(0,0)[rc]{$\mathbf F_1$}}
\put(15,96){\makebox(0,0)[cc]{$\mathbf M_{14}$}}
\put(16,83){\makebox(0,0)[lc]{$\mathbf M_{15}$}}
\put(16,73){\makebox(0,0)[lc]{$\mathbf M_{16}$}}
\put(26,13){\makebox(0,0)[lc]{$\mathbf{SL}$}}
\put(25,0){\makebox(0,0)[cc]{$\mathbf T$}}
\end{picture}
\end{center}
\caption{The lattice $L(\mathbf M_{14})$}
\label{L(M_14)}
\end{figure}
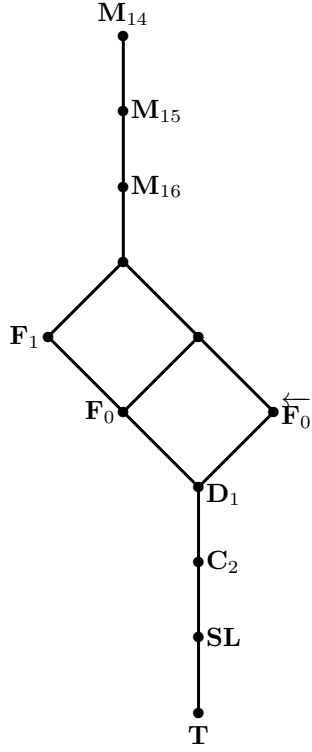

Finally, the fifth limit monoid variety, due to Gusev and O.\,B.\@ Sapir~\cite{Gusev-Sapir_O-22}, is the variety $\mathbf M_8$; the lattice $L(\mathbf M_8)$ is shown in Fig.~\ref{L(M_8)}.
A few more explicit examples of limit varieties of aperiodic monoids have recently been found \cite{Gusev-Li_Y-Zhang-arx,Sapir_O-arx}, but descriptions of their subvariety lattices are unknown.

The aforementioned examples of limit varieties play an important role in describing almost Cross varieties in some important classes of monoid varieties.
For instance, the limit varieties $\mathbf M_5$ and $\mathbf M_7\vee\overleftarrow{\mathbf M_7}$, together with $\mathbf D$, are the only almost Cross varieties of aperiodic monoids with central idempotents~\cite{Lee-13}; this result was recently generalized to varieties of aperiodic monoids with commuting idempotents.

\begin{proposition}[Gusev~\cite{Gusev-arx}]
\label{all limit}
The varieties $\mathbf M_5$, $\mathbf M_7\vee\overleftarrow{\mathbf M_7}$, $\mathbf M_{14}$, $\overleftarrow{\mathbf M_{14}}$, $\mathbf D$, $\mathbf F$, $\overleftarrow{\mathbf F}$,
\[
\mathbf Y =\var\{xyx\approx xyx^2,x^2y^2\approx y^2x^2,\,xyzxy\approx yxzxy,\,x^2yzy\approx xyxzy\approx yx^2zy\},
\]
and $\overleftarrow{\mathbf Y}$ are the only almost Cross varieties of aperiodic monoids with commuting idempotents.
\end{proposition}

The lattice $L(\mathbf Y)$ is given in Fig.~\ref{L(Y)}; the undefined varieties in this lattice are
\[
\mathbf Y_n=\mathbf Y \wedge\var\{xyt_1\mathbf f_1t_2\mathbf f_2\cdots t_{n+1}\mathbf f_{n+1}\approx yxt_1\mathbf f_1t_2\mathbf f_2\cdots t_{n+1}\mathbf f_{n+1}\},
\]
where $\mathbf f_{2i-1}=x$ and $\mathbf f_{2i}=y$ for all $i\in\mathbb N$ \cite[Proposition~3.1]{Gusev-arx}.

\begin{figure}[htb]
\unitlength=1mm
\linethickness{0.4pt}
\begin{center}
\begin{picture}(44,118)
\put(8,53){\circle*{1.33}}
\put(8,73){\circle*{1.33}}
\put(18,43){\circle*{1.33}}
\put(18,63){\circle*{1.33}}
\put(18,83){\circle*{1.33}}
\put(18,93){\circle*{1.33}}
\put(18,103){\circle*{1.33}}
\put(18,113){\circle*{1.33}}
\put(28,3){\circle*{1.33}}
\put(28,13){\circle*{1.33}}
\put(28,23){\circle*{1.33}}
\put(28,33){\circle*{1.33}}
\put(28,53){\circle*{1.33}}
\put(28,73){\circle*{1.33}}
\put(38,43){\circle*{1.33}}
\put(38,63){\circle*{1.33}}
\gasset{AHnb=0,linewidth=0.4}
\drawline(18,83)(18,105)
\drawline(28,3)(28,33)(8,53)
\drawline(28,33)(38,43)(18,63)(8,53)
\drawline(28,73)(18,83)(8,73)(18,63)(28,73)(38,63)(18,43)
\drawline[dash={0.2 0.8}{0}](18,105)(18,112)
\put(29,23){\makebox(0,0)[lc]{$\mathbf C_2$}}
\put(29,31){\makebox(0,0)[lc]{$\mathbf D_1$}}
\put(17,43){\makebox(0,0)[rc]{$\mathbf F_0$}}
\put(39,43){\makebox(0,0)[lc]{$\overleftarrow{\mathbf F_0}$}}
\put(7,53){\makebox(0,0)[rc]{$\mathbf F_1$}}
\put(7,73){\makebox(0,0)[rc]{$\mathbf M_{16}$}}
\put(39,63){\makebox(0,0)[lc]{$\mathbf Q$}}
\put(29,13){\makebox(0,0)[lc]{$\mathbf{SL}$}}
\put(29,0){\makebox(0,0)[cc]{$\mathbf T$}}
\put(18,116){\makebox(0,0)[cc]{$\mathbf Y$}}
\put(30,73){\makebox(0,0)[lc]{$\mathbf Y_1$}}
\put(20,83){\makebox(0,0)[lc]{$\mathbf Y_2$}}
\put(19,93){\makebox(0,0)[lc]{$\mathbf Y_3$}}
\put(19,103){\makebox(0,0)[lc]{$\mathbf Y_4$}}
\end{picture}
\end{center}
\caption{The lattice $L(\mathbf Y)$}
\label{L(Y)}
\end{figure}
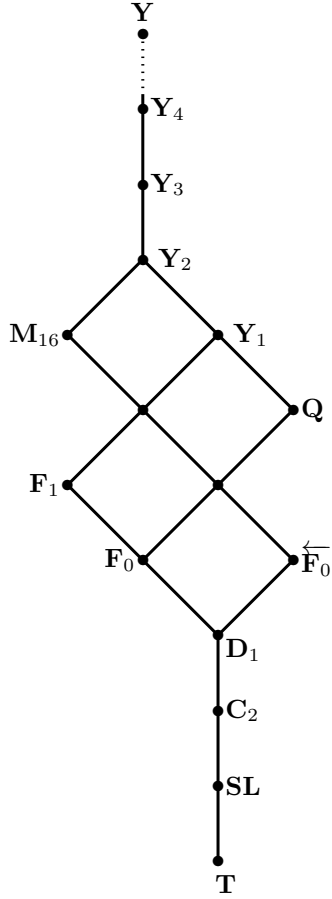

One more countably infinite series of Cross monoid varieties, $\mathbf E_n$ with $n\in\mathbb N$, and an almost Cross variety $\mathbf E_\infty$ will be described in Subsection~\ref{max and min}.

It is obvious that the class of small varieties of any algebras is closed under meets.
In general, however, this class need not be closed under joins or covers.
In particular, M.\,V.\@ Sapir~\cite{Sapir_M-91} has shown that the class of small semigroup varieties is closed under neither joins nor covers.
The same result also holds for monoid varieties due to examples of Gusev~\cite{Gusev-19} and Jackson and Lee~\cite{Jackson-Lee-18}.

\begin{proposition}
\label{small is not closed}
The class of small monoid varieties is closed under neither joins nor covers.
\end{proposition}

\begin{proof}
This class is not closed under joins because there exist pairs $(\mathbf X_1,\mathbf X_2)$ of Cross varieties for which the join $\mathbf X_1\vee\mathbf X_2$ is non-small, for example, $(\mathbf M_6,\overleftarrow{\mathbf M_7})$ \cite[Theorem 1.1(ii)]{Gusev-19}, $(\mathbf C_3,\mathbf A_0^1\vee\mathbf{LRB}\vee\mathbf{RRB})$, and $(\mathbf C_3,\mathbf G)$, where $\mathbf G$ is any finitely generated non-Abelian group variety \cite[Subsection~3.2]{Jackson-Lee-18}.
It is also not closed under covers since the Cross variety $\mathbf M_6$ is covered by the non-small variety $\mathbf M_6\vee\overleftarrow{\mathbf M_7}$ \cite[Theorem 1.1(i)]{Gusev-19}; see Fig.~\ref{L(M_6 vee dual to M_7)}.
\end{proof}

It is of interest to note that each of the joins $\mathbf X_1\vee\mathbf X_2$ given in the proof of Proposition~\ref{small is not closed} has the extreme property of containing uncountably many subvarieties.
Therefore the following question is relevant: are there two small semigroup varieties whose join contains uncountably many subvarieties?
This question, first posed by Jackson \cite[Question~3.15]{Jackson-00}, remains open; it is one of only a few questions where the answer is known for monoid varieties but unknown for semigroup varieties.
Another example will be given in Subsection~\ref{max and min}.

\subsection{The ascending and descending chain conditions}
\label{max and min}
It is convenient to say that a variety $\mathbf X$ \emph{satisfies ACC} [respectively, \emph{DCC}] if the lattice $L(\mathbf X)$ satisfies the ascending chain condition [respectively, descending chain condition].

First, we mention here two results that can be deduced from Rasin \cite[Corollaries~7 and~8]{Rasin-82} and Proposition~\ref{MON is sublat of SEM} (or from Pol\'ak~\cite{Polak-88} and Proposition~\ref{UCR is sublat of UCR_sem}).

\begin{proposition}
\label{max min orthodox}
For any variety $\mathbf X$ of orthodox completely regular monoids,
\begin{itemize}
\item[a)] $\mathbf X$ satisfies ACC if and only if $\mathsf{Gr}(\mathbf X)$ satisfies ACC and $\mathbf{BAND}\nsubseteq\mathbf X$;
\item[b)] $\mathbf X$ satisfies DCC if and only if $\mathsf{Gr}(\mathbf X)$ satisfies DCC.
\end{itemize}
\end{proposition}

It is unknown if Proposition~\ref{max min orthodox} holds for other non-orthodox varieties.

\begin{question}
\label{max min cr?}
Is the analog of Proposition~\ref{max min orthodox} true for arbitrary varieties of completely regular monoids?
\end{question}

We have little information about monoid varieties satisfying ACC or DCC outside the completely regular case.
M.\,V.\@ Sapir~\cite{Sapir_M-91} demonstrated that the class of semigroup varieties that satisfy DCC is closed under neither joins nor covers.
But whether or not the class of semigroup varieties that satisfy ACC is closed under joins or covers remains an open question \cite[Question~10.2]{Shevrin-Vernikov-Volkov-09}.
This is another question---after the one given at the end of Subsection~\ref{small}---where the answer is unknown for semigroup varieties but known for monoid varieties.

\begin{proposition}
\label{ACC and DCC are not closed}
The class of monoid varieties that satisfy ACC and the class of monoid varieties that satisfy DCC are closed under neither joins nor covers.
\end{proposition}

\begin{proof}
These two classes are not closed under joins because for any of the explicit pairs $(\mathbf X_1,\mathbf X_2)$ of Cross varieties given in the proof of Proposition~\ref{small is not closed}, the join $\mathbf X_1 \vee\mathbf X_2$ violates both ACC and DCC \cite{Gusev-19,Jackson-Lee-18}.
The two classes are also not closed under covers because the Cross variety $\mathbf M_6$ in Fig.~\ref{L(M_6 vee dual to M_7)} is covered by the variety $\mathbf M_6\vee\overleftarrow{\mathbf M_7}$ which violates both ACC and DCC; specifically, the subinterval $[\mathbf M_7\vee\overleftarrow{\mathbf M_7},\mathbf M_6\vee\overleftarrow{\mathbf M_7}]$ of $L(\mathbf M_6\vee\overleftarrow{\mathbf M_7})$ violates the ascending and descending chain conditions~\cite{Gusev-19}.
\end{proof}

There exist monoid varieties that satisfy DCC but violate ACC, for example, $\mathbf{BAND}$, $\mathbf{COM}$, $\mathbf D$, $\mathbf F$, and $\mathbf Y$; see Figs.~\ref{L(BAND)}, \ref{all chain var}, and~\ref{L(Y)} and Theorem~\ref{struct of COM}.
These varieties are all non-finitely generated; a finitely generated example is the variety $\mathbf E=\var E^1$, where
\[
E=\langle a,b,c\mid a^2=ab=0,\,b^2=bc=b,\,c^2=cb=c,\,ba=ca=a\rangle=\{a,b,c,ac,0\}.
\]
The lattice $L(\mathbf E)$ is described in Jackson and Lee~\cite{Jackson-Lee-18}; see Fig.~\ref{L(E)}.
The undefined varieties in $L(\mathbf E)$ are
\[
\mathbf E_n=
\begin{cases}
\mathbf E\wedge\var\{\mathbf e_1t_1\mathbf e_2t_2\cdots\mathbf e_nt_nx^2y^2\approx\mathbf e_1t_1\mathbf e_2t_2\cdots\mathbf e_nt_ny^2x^2\}&\text{if}\enskip n\in\mathbb N,\\
\mathbf E\wedge\var\{x^2y^2tx^2y^2\approx x^2y^2ty^2x^2\}&\text{if}\enskip n=\infty,
\end{cases}
\]
where $\mathbf e_{2i-1}=x^2$ and $\mathbf e_{2i}=y^2$ for all $i\in\mathbb N$.

\begin{figure}[htb]
\unitlength=1mm
\linethickness{0.4pt}
\begin{center}
\begin{picture}(46,118)
\put(10,33){\circle*{1.33}}
\put(10,53){\circle*{1.33}}
\put(20,43){\circle*{1.33}}
\put(20,63){\circle*{1.33}}
\put(30,3){\circle*{1.33}}
\put(30,13){\circle*{1.33}}
\put(30,23){\circle*{1.33}}
\put(30,33){\circle*{1.33}}
\put(30,53){\circle*{1.33}}
\put(30,73){\circle*{1.33}}
\put(30,83){\circle*{1.33}}
\put(30,93){\circle*{1.33}}
\put(30,103){\circle*{1.33}}
\put(30,113){\circle*{1.33}}
\put(40,43){\circle*{1.33}}
\put(40,63){\circle*{1.33}}
\gasset{AHnb=0,linewidth=0.4}
\drawline(20,63)(30,73)(40,63)(20,43)
\drawline(30,3)(30,33)(10,53)(10,33)(30,13)
\drawline(30,33)(40,43)(20,63)(10,53)
\drawline(30,73)(30,95)
\drawline(30,103)(30,113)
\drawline[dash={0.2 0.8}{0}](30,95)(30,103)
\put(31,23){\makebox(0,0)[lc]{$\mathbf C_2$}}
\put(31,32){\makebox(0,0)[lc]{$\mathbf D_1$}}
\put(30,116){\makebox(0,0)[cc]{$\mathbf E$}}
\put(32,73){\makebox(0,0)[lc]{$\mathbf E_1$}}
\put(31,83){\makebox(0,0)[lc]{$\mathbf E_2$}}
\put(31,93){\makebox(0,0)[lc]{$\mathbf E_3$}}
\put(31,103){\makebox(0,0)[lc]{$\mathbf E_\infty$}}
\put(19,43){\makebox(0,0)[rc]{$\mathbf F_0$}}
\put(41,43){\makebox(0,0)[lc]{$\overleftarrow{\mathbf F_0}$}}
\put(9,33){\makebox(0,0)[rc]{$\mathbf{LRB}$}}
\put(41,63){\makebox(0,0)[lc]{$\mathbf Q$}}
\put(31,13){\makebox(0,0)[lc]{$\mathbf{SL}$}}
\put(30,0){\makebox(0,0)[cc]{$\mathbf T$}}
\end{picture}
\end{center}
\caption{The lattice $L(\mathbf E)$}
\label{L(E)}
\end{figure}
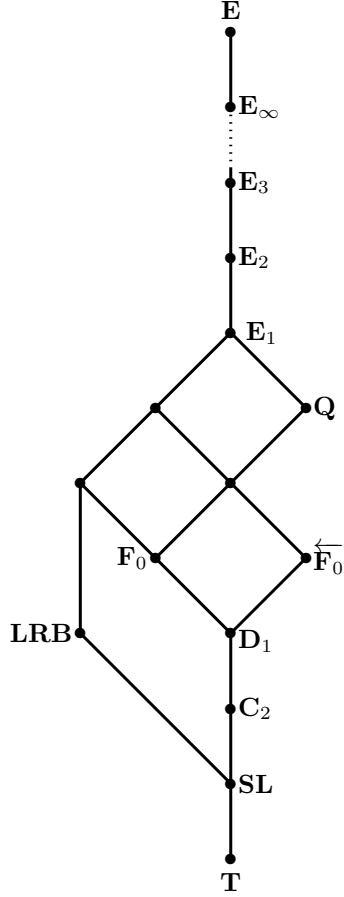

The following question remains open.

\begin{question}
\label{asc but not desc?}
Is there a monoid variety that satisfies ACC but violates DCC?
\end{question}

Note that an example of a semigroup variety that satisfies ACC but violates DCC can be found in M.\,V.\@ Sapir~\cite{Sapir_M-91}.

The following question was first posed in A\v{\i}zen\v{s}tat and Boguta~\cite{Aizenshtat-Boguta-79} and later repeated in Shevrin \textit{et~al.}\@ \cite[Question~10.3]{Shevrin-Vernikov-Volkov-09}: is there a non-small semigroup variety that satisfies both ACC and DCC? 
This question remains open so far, and it is natural to formulate it for monoid varieties.

\begin{question}
\label{asc and desc but inf?}
Is there a non-small monoid variety that satisfies both ACC and DCC?
\end{question}

\section{Other restrictions}
\label{other}

In this section, we consider three different types of restrictions on a subvariety lattice of monoid varieties: decomposability into a direct product, self-duality, and complementability and related conditions.
These properties have not yet been considered for monoid varieties, but substantial information---up to complete classification in some cases---easily follows from known results.

\subsection{Decomposability into a direct product}
\label{decompos}
A variety of algebras is \emph{decomposable} if its subvariety lattice is decomposable into a direct product of non-singleton lattices.
Decomposable varieties of semigroups were studied by Vernikov; see Shevrin \textit{et~al.}\@ \cite[Subsection~13.3]{Shevrin-Vernikov-Volkov-09}.

Two varieties of algebras $\mathbf X_1$ and $\mathbf X_2$ of the same type are \emph{disjoint} if $\mathbf X_1\wedge\mathbf X_2=\mathbf T$.
Simple lattice-theoretic arguments show that if $\langle L;\vee,\wedge\rangle$ is a modular, 0-distributive lattice and $x,y\in L$ are such that $x\wedge y=0$, then the interval $[0,x\vee y]$ in $L$ is isomorphic to the direct product of the intervals $[0,x]$ and $[0,y]$.
Combining this claim with Observation~\ref{0-distr} and Proposition~\ref{CR is mod}, we obtain the following result.

\begin{proposition}
\label{decompos cr}
If $\mathbf V$ is a completely regular monoid variety such that $\mathbf V=\mathbf X_1\vee\mathbf X_2$ for some disjoint varieties $\mathbf X_1$ and $\mathbf X_2$, then $L(\mathbf V)\cong L(\mathbf X_1)\times L(\mathbf X_2)$.
Therefore, a completely regular monoid variety is decomposable if and only if it is the join of two non-trivial disjoint varieties.
\end{proposition}

Since the lattice $\mathbb{CR}_\mathsf{sem}$ is modular and 0-distributive, the semigroup analog of Proposition~\ref{decompos cr} is valid as well.

Recall that Proposition~\ref{number of covers cr} states that every completely regular monoid variety has an infinite number of covers in the lattice $\mathbb{MON}$.
We are now ready to provide a proof of this result.

\begin{proof}[Proof of Proposition~\ref{number of covers cr}]
Let $\mathbf V$ be any completely regular monoid variety.
The join of an arbitrary infinite family of group atoms of $\mathbb{MON}$ equals $\mathbf{COM}$.
Therefore, $\mathbf V$ contains only a finite number of varieties of the form $\mathbf A_p$ with prime $p$.
In other words, there is an infinite set $\Pi$ of prime numbers such that $\mathbf A_p\nsubseteq\mathbf V$ for any $p\in\Pi$.
Now we can apply Observation~\ref{atoms} and Proposition~\ref{decompos cr} with the conclusion that $\mathbf V\vee\mathbf A_p$ covers $\mathbf V$ for any $p\in\Pi$.
\end{proof}

Some necessary conditions of decomposability can also be established for varieties of monoids that are not completely regular.

\begin{proposition}
\label{decompos non-cr}
Suppose that $\mathbf V$ is any non-completely regular monoid variety such that $L(\mathbf V)\cong L(\mathbf X_1)\times L(\mathbf X_2)$ for some non-trivial subvarieties $\mathbf X_1$ and $\mathbf X_2$ of $\mathbf V$.
Then either $\mathbf X_1$ or $\mathbf X_2$ equals $\mathbf A_n$ for some $n\ge 2$.
\end{proposition}

\begin{proof}
Clearly, $\mathbf X_1$ and $\mathbf X_2$ are disjoint.
Further, since $\mathbf V$ is not completely regular, $\mathbf C_2\subseteq\mathbf V$ by line~4 of Table~\ref{forbid var}.
Since the lattice $L(\mathbf C_2)$ has only one atom, $\mathbf C_2$ is contained in either $\mathbf X_1$ or $\mathbf X_2$, say, in $\mathbf X_1$.
Then $\mathbf{SL}\subset\mathbf C_2\subseteq\mathbf X_1$.
It follows that $\mathbf{SL}\nsubseteq\mathbf X_2$ because $\mathbf X_1$ and $\mathbf X_2$ are disjoint.
Now we can apply line~3 of Table~\ref{forbid var} and conclude that $\mathbf X_2$ is a group variety.
Suppose that $\mathbf X_2$ is non-Abelian.
Then $\mathbf D_1\subset\mathbf C_2\vee\mathbf X_2$ \cite[Lemma~3.1]{Gusev-18b}, whence $\mathbf D_1\vee\mathbf X_2\subseteq\mathbf C_2\vee\mathbf X_2$; the reverse inclusion is evident, so that $\mathbf D_1\vee\mathbf X_2=\mathbf C_2\vee\mathbf X_2$.
Since the lattice $L(\mathbf D_1)$ has only one atom (see Fig.~\ref{all chain var}), the inclusion $\mathbf D_1\subseteq\mathbf X_1$ follows.
We see that $\mathbf C_2\subset\mathbf D_1\subseteq\mathbf X_1$ but $\mathbf C_2\vee\mathbf X_2=\mathbf D_1\vee\mathbf X_2$.
This contradicts the claim that $L(\mathbf V)\cong L(\mathbf X_1)\times L(\mathbf X_2)$.
Consequently, $\mathbf X_2$ is an Abelian group variety, that is, $\mathbf X_2=\mathbf A_n$ for some $n\ge 2$.
\end{proof}

\begin{proposition}
\label{decompos one more nec}
Let $\mathbf V$ be an aperiodic monoid variety with a unique cover $\mathbf U$ in $\mathbb{MON}$ such that $\mathbf U$ is a meet of all varieties that properly contain $\mathbf V$. 
If a monoid variety $\mathbf W$ is decomposable, then $\mathbf V\nsubseteq\mathbf W$.
\end{proposition}

\begin{proof}
By Proposition~\ref{number of covers cr}, the variety $\mathbf V$ is not completely regular.
Suppose that $\mathbf W$ is decomposable and $\mathbf V\subseteq\mathbf W$.
Then $\mathbf W$ is not completely regular too.
According to Proposition~\ref{decompos non-cr}, $L(\mathbf W)\cong L(\mathbf X)\times L(\mathbf A_n)$ for some subvariety $\mathbf X$ of $\mathbf W$ and some $n\ge 2$ such that $\mathbf A_n\subseteq\mathbf W$.
Since $\mathbf U$ is a meet of all subvarieties of $\mathbf X$ that properly contain $\mathbf V$ and maximal aperiodic monoid varieties do not exist, the variety $\mathbf U$ is aperiodic.
Hence $\mathbf U\nsubseteq\mathbf A_n$.
On the other hand, $\mathbf U\subset\mathbf V\vee\mathbf A_n$ by the choice of $\mathbf U$, which implies that $\mathbf U\vee\mathbf A_n\subseteq\mathbf V\vee\mathbf A_n$.
The reverse inclusion is evident, thus $\mathbf U\vee\mathbf A_n=\mathbf V\vee\mathbf A_n$.
Since $\mathbf U$ is aperiodic, the lattice $L(\mathbf U)$ has only one atom, whence $\mathbf U\subseteq\mathbf X$.
We see that $\mathbf V\subset\mathbf U\subseteq\mathbf X$ but $\mathbf V\vee\mathbf A_n=\mathbf U\vee\mathbf A_n$.
This contradicts the claim that $L(\mathbf W)\cong L(\mathbf X)\times L(\mathbf A_n)$.
\end{proof}

In particular, Propositions~\ref{unique cover general} and~\ref{decompos one more nec} and Remark~\ref{COM,M_2,N_k} imply the following result.

\begin{remark}
\label{out of decompos}
Every decomposable monoid variety does not contain any of the varieties $\mathbf M_2$ and $\mathbf N_k$ for any $k\in\mathbb N$.
\end{remark}

In view of Proposition~\ref{decompos non-cr}, if $\mathbf V$ is a non-completely regular decomposable monoid variety, then $\mathbf V=\mathbf A_n\vee\mathbf X$ for some $n\ge 2$ and $\mathbf X \subseteq \mathbf V$ such that $\mathbf A_n \wedge \mathbf X = \mathbf T$.
To describe all decomposable varieties, it is natural to first of all consider the case when the variety $\mathbf X$ is aperiodic.
This leads to the following problem which seems to be very difficult.

\begin{problem}
\label{A_n vee aper}
Let $n\ge 2$. 
Describe all aperiodic monoid varieties $\mathbf X$ such that $L(\mathbf A_n\vee\mathbf X)\cong L(\mathbf A_n)\times L(\mathbf X)$.
\end{problem}

Remark~\ref{out of decompos} shows that there are aperiodic varieties $\mathbf X$ that do not satisfy the property indicated in Problem~\ref{A_n vee aper}.
A weakened variant of this problem (Problem~\ref{hsd descr}) will be posed in Subsection~\ref{hered selfdual}.

\subsection{Hereditarily selfdual varieties}
\label{hered selfdual}
A variety of algebras is \emph{selfdual} if it has a selfdual subvariety lattice.
Natural examples of selfdual varieties are small chain varieties.
It is clear from Figs.~\ref{L(M_14)} and~\ref{L(Y)} that $\mathbf M_{14}$, $\mathbf Y_5$, and the varieties dual to them are selfdual.

\begin{proposition}
\label{selfdual commut or band}
\quad
\begin{itemize}
\item[a)] A commutative monoid variety $\mathbf V$ is selfdual if and only if $\mathbf V\ne\mathbf{COM}$.
\item[b)] A variety of band monoids is selfdual if and only if it coincides with $\mathbf{T}$, $\mathbf{SL}$, $\mathbf B_k$, or $\overleftarrow{\mathbf B_k}$ for some $k\ge 2$.
\end{itemize}
\end{proposition}

Parts a) and b) of this proposition follow from Theorems~\ref{struct of COM} and~\ref{struct of BAND}, respectively.

The problem of describing selfdual varieties of semigroups or monoids is very difficult and has not received much attention.
One approach to the problem is to consider a stronger property: a variety of algebras is \emph{hereditarily selfdual} if all its subvarieties are selfdual.
Hereditarily selfdual varieties of semigroups, modulo chain group varieties, were described by Vernikov \cite[Theorem~1]{Vernikov-90}.

Being hereditarily selfdual is quite a strong restriction on a variety, but varieties with this property is of some interest.
Indeed, as we will see below, hereditarily selfdual varieties of monoids (and as well as semigroups) occupy an intermediate position between the fully studied class of small chain varieties (see Theorem~\ref{chain descript} and Fig.~\ref{all chain var}) and the class of varieties with a distributive subvariety lattice, which is still far from completely classified (see Subsection~\ref{distr}).
Thus, the investigation of hereditarily selfdual varieties of monoids can be considered an intermediate step in the study of monoid varieties with a distributive lattice of subvarieties.

It is evident that every small chain variety of any algebras is hereditarily selfdual.
The following assertion shows that, in a large class of varieties of algebras that includes both semigroup varieties and monoid ones, the consideration of hereditarily selfdual varieties is reduced, in a sense, to the problem of ``interaction'' between small chain varieties.

\begin{proposition}[Vernikov {\cite[Proposition~2]{Vernikov-90}}]
\label{hsd univ alg}
Suppose that $\mathbf V$ is any variety of algebras such that any two different atoms in the lattice $L(\mathbf V)$ are covered by their join.
Then $\mathbf V$ is hereditarily selfdual if and only if $L(\mathbf V)$ is a direct product of a finite number of finite chains.
\end{proposition}

In particular, if a hereditarily selfdual variety $\mathbf V$ satisfies the hypothesis of Proposition~\ref{hsd univ alg}, then the lattice $L(\mathbf V)$ is distributive.

The fact that any two different atoms of $\mathbb{MON}$ are covered by their join follows from Observation~\ref{atoms} and Proposition~\ref{decompos cr}.

Proposition~\ref{selfdual commut or band} immediately implies that a commutative monoid variety $\mathbf V$ is hereditarily selfdual if and only if $\mathbf V\ne\mathbf{COM}$.
A description of hereditarily selfdual completely regular monoid varieties modulo chain group varieties follows from Propositions~\ref{decompos cr} and~\ref{hsd univ alg} and Theorem~\ref{chain descript}.

\begin{proposition}
\label{hsd cr}
A completely regular monoid variety $\mathbf V$ is hereditarily selfdual if and only if $\mathbf V=\mathbf G_1\vee\mathbf G_2\vee\cdots\vee\mathbf G_k\vee\mathbf X$ for some pairwise disjoint small chain group varieties $\mathbf G_1,\mathbf G_2,\dots,\mathbf G_k$ and some $\mathbf X\in\{\mathbf T,\mathbf{SL},\mathbf{LRB},\mathbf{RRB}\}$.
\end{proposition}

As for hereditarily selfdual monoid varieties that are not completely regular, combining Proposition~\ref{hsd univ alg} with line~4 of Table~\ref{forbid var} and Proposition~\ref{decompos non-cr}, we obtain the following necessary condition.

\begin{proposition}
\label{hsd nec}
Suppose that $\mathbf V$ is any non-completely regular monoid variety that is hereditarily selfdual.
Then $\mathbf V=\mathbf A_n\vee\mathbf X$ for some $n\in\mathbb N$ and small chain variety $\mathbf X$ such that $\mathbf C_2\subseteq\mathbf X$.
\end{proposition}

We note that by Proposition~\ref{selfdual commut or band}, the variety $\mathbf A_n\vee\mathbf C_k$ is hereditarily selfdual for any $n,k\in\mathbb N$.
In view of Proposition~\ref{hsd nec}, the problem of completely classifying hereditarily selfdual monoid varieties modulo chain group varieties is equivalent to the following problem.

\begin{problem}
\label{hsd descr}
Find all non-group small chain varieties of monoids $\mathbf X$ such that $L(\mathbf A_n\vee\mathbf X)\cong L(\mathbf A_n)\times L(\mathbf X)$ for any $n\ge 2$.
\end{problem}

In view of Proposition~\ref{selfdual commut or band} and Fig.~\ref{all chain var}, to solve this problem, it suffices to consider the cases when a small chain variety $\mathbf X$ contains $\mathbf D_1$ and is contained in one of the following varieties: $\mathbf D$, $\mathbf F$, $\mathbf M_5$, and $\mathbf M_6$.

\subsection{Complementability and related properties}
\label{compl}
A lattice $\langle L;\vee,\wedge\rangle$ with~0 and~1 is a \emph{lattice with upper semicomplements} if, for any $x\in L\setminus\{0\}$, there exists some $y\in L\setminus\{1\}$ such that $x\vee y=1$.

\begin{theorem}
\label{compl and usemicompl}
For any variety $\mathbf V$ of monoids, the following are equivalent:
\begin{itemize}
\item[a)] $L(\mathbf V)$ is a lattice with upper semicomplements;
\item[b)] $L(\mathbf V)$ is a lattice with complements;
\item[c)] $L(\mathbf V)$ is a finite Boolean algebra;
\item[d)] $\mathbf V$ is the join of a finite number of atoms of the lattice $\mathbb{MON}$.
\end{itemize}
\end{theorem}

For convenience of references, we formulate the following result which immediately follows from the main result of Vernikov and Volkov~\cite{Vernikov-Volkov-82}.

\begin{lemma}
\label{compl group}
Conditions~b)--d) of Theorem~\ref{compl and usemicompl} are equivalent for any variety $\mathbf V$ of groups.
\end{lemma}

\begin{proof}[Proof of Theorem~\ref{compl and usemicompl}]
By Dierks \textit{et~al.}\@ \cite[Corollary~1]{Dierks-Erne-Reinhold-94}, conditions~a) and~b) are equivalent for varieties of arbitrary algebras.\footnote
{Before the publication of Dierks \textit{et~al.}\@~\cite{Dierks-Erne-Reinhold-94}, Vernikov \cite[Theorem~1(f)]{Vernikov-92} proved that~a) and~b) are equivalent for varieties of algebras with a 0-distributive subvariety lattice.
In view of Observation~\ref{0-distr}, this implies the equivalence of these two claims for monoid varieties.}
Since the implication c)\,$\to$\,d) is evident, it remains to verify the implications b)\,$\to$\,c) and d)\,$\to$\,b).

\smallskip

b)\,$\to$\,c) If $\mathbf V\nsupseteq\mathbf{SL}$, then $\mathbf V$ is a group variety by line~3 of Table~\ref{forbid var}, whence $L(\mathbf V)$ is a finite Boolean algebra by Lemma~\ref{compl group}.
Therefore, assume that $\mathbf V\supseteq\mathbf{SL}$.
Let $\mathbf U$ be the complement of $\mathbf{SL}$ in $L(\mathbf V)$.
Then $\mathbf U$ is a group variety.
Since $\mathbf V=\mathbf U\vee\mathbf{SL}$, Proposition~\ref{decompos cr} implies that $L(\mathbf V)$ is a direct product of $L(\mathbf U)$ and the 2-element chain $L(\mathbf{SL})$.
If a subvariety $\mathbf X$ of $\mathbf U$ has no complements in $L(\mathbf U)$, then the variety $\mathbf X\vee\mathbf{SL}$ has no complements in $L(\mathbf V)$.
Therefore, $L(\mathbf U)$ is a lattice with complements.
Then $L(\mathbf U)$ is a finite Boolean algebra by Lemma~\ref{compl group}, whence $L(\mathbf V)$ is a finite Boolean algebra too.

\smallskip

d)\,$\to$\,b) Let $\mathbf V=\bigvee_{i=1}^n\mathbf X_i$, where $\mathbf X_1,\mathbf X_2,\dots,\mathbf X_n$ are distinct atoms of the lattice $\mathbb{MON}$.
If $\mathbf X_1,\mathbf X_2,\dots,\mathbf X_n$ are group varieties, then it suffices to refer to Lemma~\ref{compl group}.
Suppose now that one of the varieties $\mathbf X_1,\mathbf X_2,\dots,\mathbf X_n$, say $\mathbf X_n$, coincides with $\mathbf{SL}$.
Then $\mathbf U=\bigvee_{i=1}^{n-1}\mathbf X_i$ is a group variety with $\mathbf V=\mathbf U\vee\mathbf{SL}$.
As in the previous paragraph, Proposition~\ref{decompos cr} implies that $L(\mathbf V)$ is a direct product of $L(\mathbf U)$ and the 2-element chain $L(\mathbf{SL})$.
By Lemma~\ref{compl group}, the lattice $L(\mathbf U)$ has complements, so that the lattice $L(\mathbf V)$ also has complements.
\end{proof}

In view of Observation~\ref{atoms}, Theorem~\ref{compl and usemicompl} gives a complete classification of monoid varieties that possess the properties in a)--c), as well as all standard stronger versions of complementability, such as relative complementness, uniqueness of complements, etc.

One can note that the exact analog of Theorem~\ref{compl and usemicompl} is true for semigroup varieties; see Shevrin \textit{et~al.}\@ \cite[Theorem~13.1]{Shevrin-Vernikov-Volkov-09} or Vernikov and Volkov \cite[Proposition~1]{Vernikov-Volkov-82} and Dierks \textit{et~al.}\@ \cite[Corollary~1]{Dierks-Erne-Reinhold-94}.
This is one of the few cases when the properties of the lattices $\mathbb{SEM}$ and $\mathbb{MON}$ completely coincide, and this coincidence is not a consequence of Proposition~\ref{MON is sublat of SEM}; other examples of the same type are given by Proposition~\ref{hsd univ alg} or by Proposition~\ref{decompos cr} and its semigroup analog.

The lattice property of having \emph{lower semicomplements}---a property dual to having upper semicomplements---is not interesting for varietal lattices since by very simple lattice-theoretical arguments, a complete atomic lattice is a lattice with lower semicomplements if and only if its greatest element is the join of all its atoms.
This fact and Theorem~\ref{compl and usemicompl} immediately imply the following result.

\begin{observation}
\label{lsemicompl}
The subvariety lattice of a monoid variety $\mathbf V$ is a lattice with lower semicomplements if and only if either $L(\mathbf V)$ is a lattice with complements or $\mathbf V=\mathbf{COM}$.
\end{observation}

The exact semigroup analog of this result does not hold because the lattice $\mathbb{SEM}$ contains non-commutative atoms, namely the varieties $\mathbf{LZ}$ and $\mathbf{RZ}$.
To obtain such an analog, one should change the equality $\mathbf V=\mathbf{COM}$ to $\mathbf V=\mathbf{COM}\vee\mathbf X$, where $\mathbf X$ is any of the following varieties: $\mathbf T$, $\mathbf{LZ}$, $\mathbf{RZ}$, and $\mathbf{LZ}\vee\mathbf{RZ}$.

\part{Distinctive elements in $\mathbb{MON}$}
\label{elements}

\section{Special elements}
\label{spec elem}

\subsection{Concrete types of special elements}
\label{concrete se}
In lattice theory, considerable attention is given to the study of so-called special elements of different types.
We will mention nine types of special elements: neutral, standard, costandard, distributive, codistributive, cancellable, modular, lower-modular, and upper-modular elements.
Neutral elements were defined in Subsection~\ref{CR}.
Note that an element $x$ in a lattice $\langle L;\vee,\wedge\rangle$ is neutral if and only if, for all $y,z\in L$, the sublattice of $L$ generated by $x$, $y$, and $z$ is distributive; see Gr\"atzer \cite[Theorem~254]{Gratzer-11}, for instance.
An element $x \in L$ is
\begin{align*}
&\text{\emph{standard} if}&&\forall y,z\in L:\,(x\vee y)\wedge z=(x\wedge z)\vee(y\wedge z);\\
&\text{\emph{distributive} if}&&\forall y,z\in L:\,x\vee(y\wedge z)=(x\vee y)\wedge(x\vee z);\\
&\text{\emph{cancellable} if}&&\forall y,z\in L:\,x\vee y=x\vee z\ \&\ x\wedge y=x\wedge z\longrightarrow y=z;\\
&\text{\emph{modular} if}&&\forall y,z\in L:\,y\le z\longrightarrow(x\vee y)\wedge z=(x\wedge z)\vee y;\\
&\text{\emph{lower-modular} if}&&\forall y,z\in L:\,x\le y\longrightarrow x\vee(y\wedge z)=y\wedge(x\vee z).
\end{align*}
\emph{Costandard}, \emph{codistributive}, and \emph{upper-modular} elements are defined dually to standard, distributive, and lower-modular ones, respectively.\footnote
{Modular, upper-modular, lower-modular, and cancellable elements are named differently in a number of works.
In particular, modular [respectively, upper-modular, cancellable] elements are called left modular [respectively, right modular, separating] in Stern~\cite{Stern-99}.
The term ``modular element'' is used in the literature not only in the sense defined above but also in a number of other ways.}
Significant information about special elements in a lattice can be found in Gr\"atzer \cite[Section~III.2]{Gratzer-11} or Stern \cite[Sections~2.1 and~2.2]{Stern-99}, for instance.

There exist several interrelations between the different types of elements defined above.
It is evident that a neutral element is both standard and costandard, a [co]standard element is cancellable, a cancellable element is modular, and a [co]distribu\-tive element is lower-modular [respectively, upper-modular].
It is also well known that a [co]standard element is [co]distributive; see Gr\"atzer \cite[Theorem~253]{Gratzer-11}, for instance.
A summary of these interrelations is given in Fig.~\ref{se interrelat abstr}.

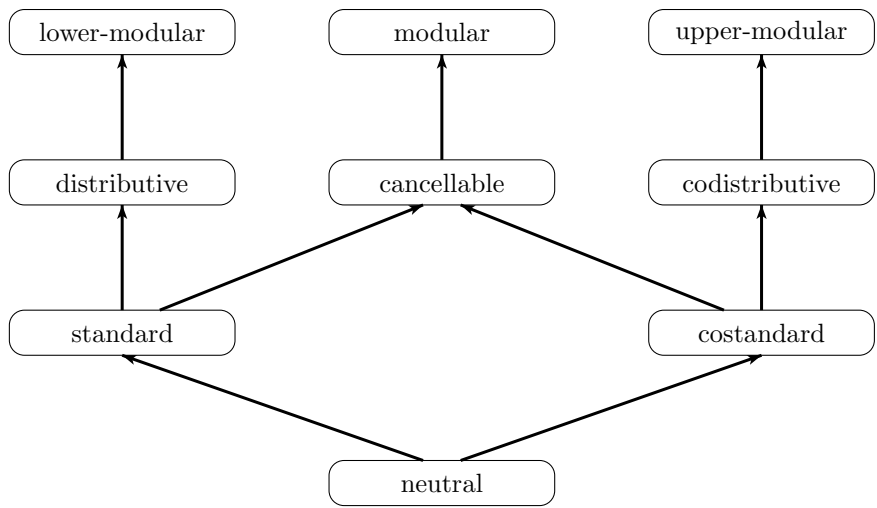
\begin{figure}[tbh]
\begin{center}
\unitlength=1mm
\linethickness{0.4pt}
\begin{picture}(115,66)
\drawoval(15,23,30,6,2)
\drawoval(15,43,30,6,2)
\drawoval(15,63,30,6,2)
\drawoval(57.5,3,30,6,2)
\drawoval(57.5,43,30,6,2)
\drawoval(57.5,63,30,6,2)
\drawoval(100,23,30,6,2)
\drawoval(100,43,30,6,2)
\drawoval(100,63,30,6,2)
\put(57.5,43){\makebox(0,0)[cc]{cancellable}}
\put(100,43){\makebox(0,0)[cc]{codistributive}}
\put(100,23){\makebox(0,0)[cc]{costandard}}
\put(15,43){\makebox(0,0)[cc]{distributive}}
\put(15,63){\makebox(0,0)[cc]{lower-modular}}
\put(57.5,63){\makebox(0,0)[cc]{modular}}
\put(57.5,3){\makebox(0,0)[cc]{neutral}}
\put(15,23){\makebox(0,0)[cc]{standard}}
\put(100,63){\makebox(0,0)[cc]{upper-modular}}
\gasset{AHnb=1,AHLength=2,linewidth=0.4}
\drawline(15,26)(15,40)
\drawline(15,46)(15,60)
\drawline(20,26)(55,40)
\drawline(55,6)(15,20)
\drawline(57.5,46)(57.5,60)
\drawline(60,6)(100,20)
\drawline(95,26)(60,40)
\drawline(100,26)(100,40)
\drawline(100,46)(100,60)
\end{picture}
\caption{Interrelations between types of special elements in lattices}
\label{se interrelat abstr}
\end{center}
\end{figure}

Over the past~15 years, a number of articles were devoted to the examination of elements of the aforementioned types in the lattice $\mathbb{SEM}$ and some of its sublattices, such as $\mathbb{COM}_\mathsf{sem}$ and $\mathbb{OC}_\mathsf{sem}$.
Initial results were overviewed in Shevrin \textit{et~al.}\@ \cite[Section~14]{Shevrin-Vernikov-Volkov-09}.
Significantly more results were systematized in a more recent survey by Vernikov~\cite{Vernikov-15} devoted entirely to the problems under discussion.
Results concerning the lattice $\mathbb{SEM}$ that were in this survey or obtained later can be briefly summarized as follows.
Six types of elements in $\mathbb{SEM}$---neutral, standard, costandard, distributive, cancellable, and lower-modular---have been completely determined.
For codistributive, modular, and upper-modular elements in $\mathbb{SEM}$, strong necessary conditions and descriptions in wide and important partial cases have been found.
For modular elements, a non-trivial sufficient condition is known too.
Further, some interrelations between special elements of different types in $\mathbb{SEM}$ that do not hold in abstract lattices have also been found; specifically, for any variety $\mathbf X$ of semigroups, $\mathbf X$ is standard if and only if $\mathbf X$ is distributive, $\mathbf X$ is costandard if and only if $\mathbf X$ is neutral, and $\mathbf X$ is modular whenever $\mathbf X$ is lower-modular.
For a description of all cancellable elements in $\mathbb{SEM}$, see Shaprynskii \textit{et~al.}\@ \cite[Theorem~1.1]{Shaprynskii-Skokov-Vernikov-19}; results on other special elements described in this paragraph can be found in Vernikov \cite[Section~3]{Vernikov-15}.

Regarding special elements in the lattice $\mathbb{MON}$, not much is known until recently, when some significant progress has been made \cite{Gusev-18b,Gusev-20b,Gusev-22+a,Gusev-Lee-21}.
To date, the exact six types of special elements---neutral, standard, costandard, distributive, cancellable, and lower-modular---are completely described in the lattice $\mathbb{MON}$ as in the lattice $\mathbb{SEM}$.
However, as we will see below, the interrelations between the types of special elements in the lattices $\mathbb{MON}$ and $\mathbb{SEM}$ differ drastically.
The aforementioned classifications of the six types of special elements are summarized in the following three theorems.

\begin{theorem}[Gusev \cite{Gusev-18b,Gusev-20b,Gusev-22+a}]
\label{neutr,stand,distr,lmod}
For any monoid variety $\mathbf V$, the following are equivalent:
\begin{itemize}
\item[a)] $\mathbf V$ is a lower-modular element in $\mathbb{MON}$;
\item[b)] $\mathbf V$ is a distributive element in $\mathbb{MON}$;
\item[c)] $\mathbf V$ is a standard element in $\mathbb{MON}$;
\item[d)] $\mathbf V$ is a neutral element in $\mathbb{MON}$;
\item[e)] $\mathbf V$ coincides with $\mathbf T$, $\mathbf{SL}$, or $\mathbf{MON}$.
\end{itemize}
\end{theorem}

Specifically, the equivalences~d)\,$\leftrightarrow$\,e), c)\,$\leftrightarrow$\,e), and a)\,$\leftrightarrow$\,b)\,$\leftrightarrow$\,e) were established in the articles \cite{Gusev-18b,Gusev-20b,Gusev-22+a}, respectively.

The number of neutral elements of a lattice can be considered as a kind of measure of its complexity.
The least and the greatest elements of any lattice $L$, if they exist, are neutral in $L$.
The existence of a unique neutral element of the lattice $\mathbb{MON}$ that is different from $\mathbf T$ and $\mathbf{MON}$ once again demonstrates how complex this lattice is.
For comparison, we note that the lattice $\mathbb{SEM}$ contains exactly three neutral elements that are different from $\mathbf T$ and $\mathbf{SEM}$: the variety $\mathbf{SL}_\mathsf{sem}$ of semilattices, the variety $\mathbf{ZM}$ of semigroups with zero multiplication, and their join $\mathbf{SL}_\mathsf{sem}\vee\mathbf{ZM}$; see Shevrin \textit{et~al.}\@ \cite[Theorem~14.2]{Shevrin-Vernikov-Volkov-09} or Vernikov \cite[Theorem~3.4]{Vernikov-15}.
We note also that the lattice $\mathbb{CR}$ has infinitely many neutral elements; in particular, it follows from Trotter~\cite{Trotter-89} and Proposition~\ref{UCR is sublat of UCR_sem} that every variety of band monoids is neutral in $\mathbb{CR}$.

The four types of special elements in Theorem~\ref{neutr,stand,distr,lmod}a)--d) coincide in the lattice $\mathbb{MON}$, and there are only three of them: $\mathbf T$, $\mathbf{SL}$, and $\mathbf{MON}$.
This differs sharply from the situation with the lattice $\mathbb{SEM}$, where the set of all lower-modular elements of $\mathbb{SEM}$ is uncountably infinite, the set of all standard elements of $\mathbb{SEM}$ is countably infinite, and the set of all neutral elements of $\mathbb{SEM}$ is finite \cite[Theorems 3.2--3.4]{Vernikov-15}.
Further, an element of $\mathbb{SEM}$ is standard if and only if it is distributive \cite[Theorem~3.3]{Vernikov-15}.

\begin{theorem}[Gusev {\cite[Theorem~1.2]{Gusev-18b}}]
\label{costand}
For any monoid variety $\mathbf V$, the following are equivalent:
\begin{itemize}
\item[a)] $\mathbf V$ is a modular and upper-modular element in $\mathbb{MON}$;
\item[b)] $\mathbf V$ is a costandard element in $\mathbb{MON}$;
\item[c)] $\mathbf V$ is one of the varieties $\mathbf T$, $\mathbf{SL}$, $\mathbf C_2$, or $\mathbf{MON}$.
\end{itemize}
\end{theorem}

Contrary to the semigroup case, we see from Theorems~\ref{neutr,stand,distr,lmod} and~\ref{costand} that for elements in the lattice $\mathbb{MON}$, the properties of being neutral and costandard are not equivalent, while the properties of being neutral and standard are equivalent.
Besides that, Theorems~\ref{neutr,stand,distr,lmod} and~\ref{costand} imply that any standard element of $\mathbb{MON}$ is costandard but the converse does not hold.
In contrast, in the lattice $\mathbb{SEM}$, any costandard element is standard but the converse is false; see Vernikov \cite[Theorems~3.3 and~3.4]{Vernikov-15}.

\begin{theorem}[Gusev and Lee~\cite{Gusev-Lee-21}]
\label{cancel}
A monoid variety is a cancellable element of the lattice $\mathbb{MON}$ if and only if it coincides with $\mathbf T$, $\mathbf{SL}$, $\mathbf C_2$, $\mathbf D_1$, or $\mathbf{MON}$.
\end{theorem}

Theorem~\ref{cancel} shows that there are only five cancellable elements in the lattice $\mathbb{MON}$.
In contrast, the set of all cancellable elements of the lattice $\mathbb{SEM}$ is countably infinite \cite[Theorem~1.1]{Shaprynskii-Skokov-Vernikov-19}.

Some results on upper-modular elements and codistributive elements of $\mathbb{MON}$ are also available.

\begin{proposition}[Gusev {\cite[Propositions~1.3 and~1.4]{Gusev-18b}}]
\label{umod,codistr}
\quad
\begin{itemize}
\item[a)] Every proper monoid variety that is an upper-modular element of the lattice $\mathbb{MON}$ is either commutative or completely regular.
\item[b)] Every commutative monoid variety is a codistributive element of the lattice $\mathbb{MON}$.
\end{itemize}
\end{proposition}

Since any codistributive element of a lattice is upper-modular, Proposition~\ref{umod,codistr} completely reduces the problem of describing codistributive or upper-modular elements in $\mathbb{MON}$ to completely regular varieties.

However, there are some essential difficulties here.
The lattice $\mathbb{GR}$ is modular but not distributive.
Therefore, it contains a sublattice isomorphic to the 5-element modular non-distributive lattice $L_5$ in Fig.~\ref{modular non-distributive}.
Clearly, if a group variety $\mathbf G$ is one of three pairwise non-comparable elements $a,b,c \in L_5$, then $\mathbf G$ is a non-codistributive element of $\mathbb{MON}$.
We see that an examination of codistributive elements of $\mathbb{MON}$ is closely related to that of group varieties with a distributive subvariety lattice, but as observed earlier, classifying such group varieties is extremely difficult; see Theorem~\ref{chain gr uncount}.
Therefore, it is logical to restrict our attention to codistributive elements of $\mathbb{MON}$ that are aperiodic varieties.
Combining this restriction with the observation from the previous paragraph, it suffices to examine varieties of band monoids that are codistributive elements of $\mathbb{MON}$.
But to date, we do not even know the answer to the following question.

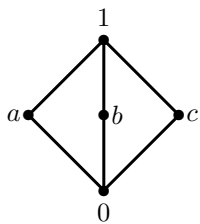
\begin{figure}[htb]
\unitlength=1mm
\linethickness{0.4pt}
\begin{center}
\begin{picture}(26,28)
\put(3,13){\circle*{1.33}}
\put(13,3){\circle*{1.33}}
\put(13,13){\circle*{1.33}}
\put(13,23){\circle*{1.33}}
\put(23,13){\circle*{1.33}}
\gasset{AHnb=0,linewidth=0.4}
\drawline(13,3)(3,13)(13,23)(23,13)(13,3)(13,23)
\put(13,0){\makebox(0,0)[cc]{0}}
\put(13,26){\makebox(0,0)[cc]{1}}
\put(2,13){\makebox(0,0)[rc]{$a$}}
\put(14,13){\makebox(0,0)[lc]{$b$}}
\put(24,13){\makebox(0,0)[lc]{$c$}}
\end{picture}
\end{center}
\caption{The lattice $L_5$}
\label{modular non-distributive}
\end{figure}

\begin{question}
\label{codistr bands?}
Is there a variety of band monoids that is a non-codistributive element of the lattice $\mathbb{MON}$?
\end{question}

In other words, is there a variety $\mathbf B$ of band monoids such that the equality
\[
\mathbf B\wedge(\mathbf X_1 \vee\mathbf X_2) = (\mathbf B\wedge\mathbf X_1)\vee(\mathbf B\wedge \mathbf X_2)
\]
fails for some monoid varieties $\mathbf X_1$ and $\mathbf X_2$?
It follows from Pastijn and Trotter \cite[Corollary~5.9]{Pastijn-Trotter-98} and Proposition~\ref{MON is sublat of SEM} that this equality holds whenever the varieties $\mathbf X_1$ and $\mathbf X_2$ both are locally finite.

The following analog of Question~\ref{codistr bands?} is also open.

\begin{question}
\label{umod bands?}
Is there a variety of band monoids that is a non-upper-modular element of the lattice $\mathbb{MON}$?
\end{question}

Note that the analogs of Questions~\ref{codistr bands?} and~\ref{umod bands?} for semigroup varieties are also currently open.
Besides that, we do not know of examples of upper-modular but non-codistributive elements of $\mathbb{MON}$.
This makes the following question relevant.

\begin{question}
\label{umod is codistr?}
Is every upper-modular element of $\mathbb{MON}$ codistributive?
\end{question}

Regarding the modularity of an element in $\mathbb{MON}$, a necessary condition has recently been found.

\begin{proposition}[Gusev and Lee {\cite[Proposition~4.3]{Gusev-Lee-21}}]
\label{mod elem}
Every proper monoid variety that is a modular element of the lattice $\mathbb{MON}$ satisfies the identities $x^2\approx x^3$ and $x^2y\approx yx^2$.
\end{proposition}

However, no modular element of $\mathbb{MON}$ is currently known to be non-cancellable.

\begin{question}
\label{mod is canc?}
Is every modular element of $\mathbb{MON}$ cancellable?
\end{question}

We summarize in Fig.~\ref{se interrelat MON} all known interrelations between different types of elements in the lattice $\mathbb{MON}$.
Dashed arrows in this figure correspond to interrelations for which it is unknown whether they hold or not.
Green ovals correspond to completely described types of elements, while yellow ovals correspond to those about which partial information is known.

\begin{figure}[tbh]
\begin{center}
\unitlength=1mm
\linethickness{0.4pt}
\begin{picture}(85,66)
\drawoval[fillcolor=green](15,43,25,6,2)
\drawoval[fillcolor=yellow](15,63,25,6,2)
\drawoval[fillcolor=green](42.5,3,85,6,2)
\drawoval[fillcolor=green](42.5,23,25,6,2)
\drawoval[fillcolor=yellow](67.5,43,25,6,2)
\drawoval[fillcolor=yellow](67.5,63,25,6,2)
\put(15,43){\makebox(0,0)[cc]{cancellable}}
\put(67.5,43){\makebox(0,0)[cc]{codistributive}}
\put(42.5,23){\makebox(0,0)[cc]{costandard}}
\put(15,63){\makebox(0,0)[cc]{modular}}
\put(42.5,3){\makebox(0,0)[cc]{neutral = standard = distributive = lower-modular}}
\put(67.5,63){\makebox(0,0)[cc]{upper-modular}}
\put(18.5,53){\makebox(0,0)[lc]{\large ?}}
\put(71,53){\makebox(0,0)[lc]{\large ?}}
\gasset{AHnb=1,AHLength=2,linewidth=0.4}
\drawline(12.5,46)(12.5,60)
\drawline[dash={1 1}{0}](17.5,60)(17.5,46)
\drawline(40,26)(15,40)
\drawline(42.5,6)(42.5,20)
\drawline(45,26)(67.5,40)
\drawline(65,46)(65,60)
\drawline[dash={1 1}{0}](70,60)(70,46)
\end{picture}
\end{center}
\caption{Interrelations between types of elements in the lattice $\mathbb{MON}$}
\label{se interrelat MON}
\end{figure}
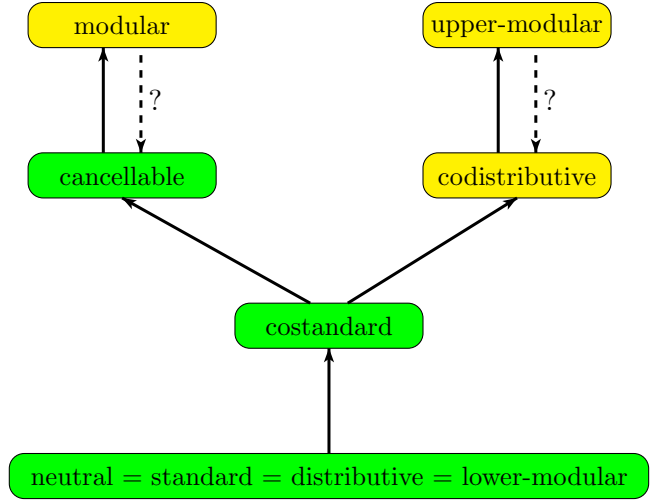

\subsection{Id-elements}
\label{Id-elem}
All types of special elements introduced above, except cancellable elements, are defined by the same scheme.
Namely, we take a particular lattice identity and consider it as an open formula.
Then, one of the letters is left free while all the others are subjected to a universal quantifier.\footnote
{Formally speaking, the definitions of modular, lower-modular, and upper-modular elements are based on a lattice quasi-identity rather than an identity.
But we give such definitions for the sake of brevity only. 
Since the modularity law may be written as an identity, it is fairly easy to redefine these types of elements in the language of lattice identities.}
This approach can easily be generalized in a natural way to an arbitrary lattice identity.

Let $\varepsilon$ be a lattice identity of the form $\mathbf s\approx\mathbf t$, where $\mathbf s$ and $\mathbf t$ are terms in the language of lattice operations $\vee$ and $\wedge$.
Suppose that these terms depend on letters $x_1,\dots,x_n$ and $1\le i\le n$.
Then an element $x$ of a lattice $L$ is an \mbox{$(\varepsilon,i)$}-\emph{element} of $L$ if for all $x_1,\dots,x_{i-1},x_{i+1},\dots,x_n\in L$, the equality
\[
\mathbf s(x_1,\dots,x_{i-1},x,x_{i+1},\dots,x_n)=\mathbf t(x_1,\dots,x_{i-1},x,x_{i+1},\dots,x_n)
\]
holds.
An element of a lattice $L$ is an \emph{Id-element} of $L$ if it is an \mbox{$(\varepsilon,i)$}-element of $L$ for some non-trivial identity $\varepsilon$ depending on letters $x_1,\dots,x_n$ with $1\le i\le n$.

For an element $a$ of a lattice $L$, we put $(a]=\{x\in L\mid x\le a\}$.
If $a\in L$ and the lattice $(a]$ satisfies the identity $\mathbf p(x_1,\dots,x_n)\approx\mathbf q(x_1,\dots,x_n)$, then
\[
\mathbf p(a\wedge x_1,\dots,a\wedge x_n)=\mathbf q(a\wedge x_1,\dots,a\wedge x_n)
\]
for all $x_1,\dots,x_n\in L$ because $a\wedge x_1,\dots,a\wedge x_n\in(a]$.
Therefore, in this situation, $a$ is an \mbox{$(\varepsilon,n+1)$}-element of $L$ with the following identity $\varepsilon$ depending on letters $x_1,\dots,x_{n+1}$:
\[
\mathbf p(x_{n+1}\wedge x_1,\dots,x_{n+1}\wedge x_n)\approx\mathbf q(x_{n+1}\wedge x_1,\dots,x_{n+1}\wedge x_n).
\]
So, we have the following statement from Shaprynski\v{\i}~\cite{Shaprynskii-12}: if $a$ is an element of a lattice $L$ such that the ideal $(a]$ of $L$ satisfies some non-trivial identity, then $a$ is an Id-element of $L$.

A monoid [respectively, semigroup] variety is an \emph{Id-variety} if it is an Id-element of the lattice $\mathbb{MON}$ [respectively, $\mathbb{SEM}$].
The following assertion is a specialization of the aforementioned result for the lattice $\mathbb{MON}$.
The analogous claim is true for semigroup varieties.

\begin{observation}
\label{ident and Id-var}
If $\mathbf V$ is a monoid variety and the lattice $L(\mathbf V)$ satisfies some non-trivial identity, then $\mathbf V$ is an Id-variety.
\end{observation}

It is verified by Shaprynski\v{\i} \cite[Theorems~1 and~2]{Shaprynskii-12} that all proper overcommutative semigroup varieties are non-Id-varieties and there exist periodic non-Id-varieties of semigroups.
The analog of the former fact does not hold for monoid varieties.
Indeed, by Proposition~\ref{umod,codistr}b), the variety $\mathbf{COM}$ is an Id-variety; this follows also from Observation~\ref{ident and Id-var} and the fact that the lattice $L(\mathbf{COM})$ is distributive by Theorem~\ref{struct of COM}.
As noted in Subsection~\ref{mod}, the lattice $L(\mathbf{COM}\vee\mathbf D_1)$ is also distributive.
In view of Observation~\ref{ident and Id-var}, the variety $\mathbf{COM}\vee\mathbf D_1$ provides one more example of overcommutative Id-variety of monoids.
The following question is still open.

\begin{question}
\label{non-Id-var exist?}
Is there a [periodic] non-Id-variety of monoids?
\end{question}

The proof of Theorem~2 in Shaprynski\v{\i}~\cite{Shaprynskii-12} implies that if a lattice $L$ contains a sublattice $K$ that is anti-isomorphic to $\Pi_\infty$, then $K$ contains an element that is not an Id-element of $L$.
Thus, an affirmative answer to Question~\ref{dual to Pi_infty in MON?}a) implies an affirmative answer to Question~\ref{non-Id-var exist?}.

\section{Definable varieties and sets of varieties}
\label{definable}

\subsection{Definability}
\label{def}
A subset $A$ of a lattice $\langle L;\vee,\wedge\rangle$ is \emph{definable} in $L$ if there exists a first-order formula $\Phi(x)$ with one free variable $x$ in the language of lattice operations $\vee$ and $\wedge$ with the following property: for an element $z\in L$, the sentence $\Phi(z)$ is true if and only if $z\in A$.
An element $a\in L$ is \emph{definable} in $L$ if the set $\{a\}$ is definable in $L$.

The importance of definable elements and subsets of a lattice is due to their close connection with automorphisms of the lattice.
Indeed, it is clear that if a subset $A$ of a lattice $L$ is definable in $L$ and $\varphi$ is any automorphism of $L$, then $\varphi(a)\in A$ for all $a\in A$; in particular, if $a$ is definable in $L$, then $\varphi(a)=a$.

The notion of definable sets of elements deeply generalizes, in a sense, the notion of special elements.
Indeed, if $\varepsilon$ is a lattice identity depending on the letters $x_1,\dots,x_n$ and $1\le i\le n$, then the set of all \mbox{$(\varepsilon,i)$}-elements of a lattice $L$ is evidently definable in $L$.
The same is true for the set of all cancellable elements of a lattice.

Definable elements and subsets of the lattice $\mathbb{SEM}$ have been examined by Je\v{z}ek and McKenzie~\cite{Jezek-McKenzie-93} and Vernikov~\cite{Vernikov-12}; see also Shevrin \textit{et~al.}\@ \cite[Section~15]{Shevrin-Vernikov-Volkov-09}.
In particular, an arbitrary commutative semigroup variety is definable in $\mathbb{SEM}$ \cite[Corollary~15.1]{Shevrin-Vernikov-Volkov-09}.
For brevity, we say that a monoid variety or a set of monoid varieties is \emph{definable} if it is definable in $\mathbb{MON}$.
There has not been any work devoted to the study of definable monoid varieties or sets of monoid varieties, but some results on this topic can be easily deduced from existing results.

\begin{proposition}
\label{def sets and var}
The set of all varieties of commutative [overcommutative, periodic, group, Abelian group, aperiodic, completely regular, band] monoid varieties and each of the varieties $\mathbf{COM}$, $\mathbf{BAND}$, and $\mathbf C_n$ for any $n\in \mathbb N$ are all definable.
\end{proposition}

\begin{proof}
The set of all atoms of an arbitrary lattice $L$ is definable in $L$.
Further, if $L$ is a complete lattice and a subset $A$ of $L$ is definable in $L$, then the element $\bigvee A$ is also definable in $L$.
The variety $\mathbf{COM}$ is the join of all atoms of the lattice $\mathbb{MON}$ and so is definable.
This immediately implies definability of the classes of all commutative, all overcommutative, and all periodic varieties (as varieties $\mathbf V$ with $\mathbf V\subseteq\mathbf{COM}$, $\mathbf{COM}\subseteq\mathbf V$, and $\mathbf{COM}\nsubseteq\mathbf V$, respectively).

Further, the variety $\mathbf C_1 = \mathbf{SL}$ is definable because it is a unique neutral element of the lattice $\mathbb{MON}$ different from the least and the greatest elements of this lattice; see Theorem~\ref{neutr,stand,distr,lmod}.
This allows us to prove definability of the sets of all group varieties (as varieties $\mathbf V$ with $\mathbf{SL}\nsubseteq\mathbf V$, see line~3 of Table~\ref{forbid var}), all Abelian group varieties (as varieties that are commutative and group varieties simultaneously), and all aperiodic varieties (as varieties that do not contain group atoms of the lattice $\mathbb{MON}$).

The variety $\mathbf C_2$ is definable because it is a unique costandard but non-neutral element of $\mathbb{MON}$; see Theorems~\ref{neutr,stand,distr,lmod} and~\ref{costand}.
Then we have definability of the sets of all completely regular varieties (as varieties $\mathbf V$ with $\mathbf C_2\nsubseteq\mathbf V$, see line~4 of Table~\ref{forbid var}) and all band varieties (as varieties that are completely regular and aperiodic simultaneously).
The variety $\mathbf{BAND}$ is then definable as the greatest band variety.

It remains to check that the variety $\mathbf C_n$ with any $n\ge 3$ is definable.
Here we use the fact that the set of all chain varieties is definable.
Then the variety $\mathbf C_3$ is definable because it is a unique commutative chain variety that covers $\mathbf C_2$; see Fig.~\ref{all chain var}.
Finally, let $n\ge 4$.
By induction on $n$, the variety $\mathbf C_n$ is definable because it is a unique chain variety that covers $\mathbf C_{n-1}$; see Fig.~\ref{all chain var} again.
\end{proof}

However, the following question remains open.

\begin{question}
\label{comm is def?}
Is every commutative monoid variety definable?
\end{question}

Proposition~\ref{def sets and var} and Theorem~\ref{struct of COM} reduce this question to the following: is the variety $\mathbf A_n$ definable for each $n\ge 2$? 
Note that, for any $n\ge 2$, the variety of Abelian groups of exponent $n$ (considered as a semigroup variety) is definable in the lattice $\mathbb{SEM}$; see Shevrin \textit{et~al.}\@ \cite[Corollary~15.1]{Shevrin-Vernikov-Volkov-09} or Vernikov \cite[Theorem~5.7]{Vernikov-12}.

\subsection{Semidefinability}
\label{semidefinable}
If $\mathbf V$ is a monoid variety such that $\mathbf V\ne\overleftarrow{\mathbf V}$, then $\mathbf V$ is evidently non-definable because for an arbitrary first-order formula $\Phi(x)$ in the lattice language, the sentences $\Phi(\mathbf V)$ and $\Phi(\overleftarrow{\mathbf V})$ are true or false simultaneously.
The following definition is thus very natural: a monoid variety $\mathbf V$ is \emph{semidefinable} if the set $\{\mathbf V,\overleftarrow{\mathbf V}\}$ is definable.

\begin{proposition}
\label{band semidef}
Every variety of band monoids is semidefinable.
\end{proposition}

\begin{proof}
It is evident that every definable variety is semidefinable and if a variety $\mathbf V$ is semidefinable, then the variety $\mathbf V\vee\overleftarrow{\mathbf V}$ is definable.
The varieties $\mathbf{SL}$ and $\mathbf{BAND}$ are definable by Proposition~\ref{def sets and var}.
By Fig.~\ref{L(BAND)}, it suffices to check that the set $\{\mathbf B_n,\overleftarrow{\mathbf B_n}\}$, for any $n\ge 2$, is definable.
We will use the fact that the set of all band varieties is definable by Proposition~\ref{def sets and var}.
The set $\{\mathbf B_2,\overleftarrow{\mathbf B_2}\}$ is definable because $\mathbf B_2$ and $\overleftarrow{\mathbf B_2}$ are the only varieties of band monoids that cover $\mathbf{SL}$; see Fig.~\ref{L(BAND)}.
By induction on $n$, for any $n>2$, the set $\{\mathbf B_n,\overleftarrow{\mathbf B_n}\}$ is definable because $\mathbf B_n$ and $\overleftarrow{\mathbf B_n}$ are the only varieties of band monoids that cover $\mathbf B_{n-1}\vee\overleftarrow{\mathbf B_{n-1}}$; see Fig.~\ref{L(BAND)} again.
\end{proof}

One can verify semidefinability of many other monoid varieties.
For example, it is easy to check that each non-group chain variety of monoids is semidefinable.

It is clear that the set of all semidefinable varieties is countably infinite.
Since the lattice $\mathbb{MON}$ is uncountably infinite, non-semidefinable monoid varieties exist, but we do not know of any explicit example.

\begin{problem}
\label{non-semidef}
Find an example of a non-semidefinable monoid variety.
\end{problem}

This problem is closely related to Question~\ref{aut exist?}a).
Indeed, if Question~\ref{aut exist?}a) is affirmatively answered, and if $\varphi$ is a non-trivial automorphism of the lattice $\mathbb{MON}$ that is different from $\delta$, then any variety $\mathbf V$ such that $\varphi(\mathbf V)\ne\mathbf V$ is non-semidefinable.

\subsection*{Acknowledgements.} The authors are indebted to Marcel Jackson and Olga Sapir for very helpful discussions, and to Mikhail Volkov for important information on pseudovarieties in Subsection~\ref{pseudovar}.

\end{document}